\documentclass[11pt]{amsart}
\usepackage{graphicx}
\usepackage{amssymb,bm}
\usepackage[usenames,dvipsnames]{xcolor}
\usepackage{hyperref}

\date{}

\theoremstyle{plain}
\newtheorem{theorem}{Theorem}[section]
\newtheorem{ex}[theorem]{Example}
\newtheorem{lemma}[theorem]{Lemma}
\newtheorem{conjecture}[theorem]{Conjecture}
\newtheorem{proposition}[theorem]{Proposition}
\newtheorem{corollary}[theorem]{Corollary}
\newtheorem{question}[theorem]{Question}

\theoremstyle{definition}
\newtheorem{definition}[theorem]{Definition}

\newtheorem{remark}[theorem]{Remark}

\newcommand{\FA}{\mathfrak{A}}
\newcommand{\s}{\mathbf{s}}
\renewcommand{\S}{\mathbb{S}}

\renewcommand{\L}{\mathbb{L}}

\newcommand{\D}{\mathbb{D}}
\newcommand{\N}{\mathbb{N}}
\newcommand{\Z}{\mathbb{Z}}

\newcommand{\R}{\mathbb{R}}
\newcommand{\C}{\mathbb{C}}
\newcommand{\SA}{\mathcal{A}}

\newcommand{\SO}{\mathcal{O}}
\newcommand{\SR}{\mathcal{R}}
\newcommand{\ST}{\mathcal{T}}

\newcommand{\BA}{\mathbf{A}}
\newcommand{\La}{\Lambda}

\newcommand{\Br}{\operatorname{Br}}

\newcommand{\tb}{\operatorname{tb}}
\newcommand{\rot}{\operatorname{rot}}

\newcommand{\Symp}{\operatorname{Symp}}

\newcommand{\Sh}{\operatorname{Sh}}
\newcommand{\Aug}{\operatorname{Aug}}

\newcommand{\Gr}{\operatorname{Gr}}

\newcommand{\Mat}{\operatorname{Mat}}
\newcommand{\LGr}{\operatorname{LGr}}
\newcommand{\ind}{\operatorname{ind}}
\newcommand{\w}{\mathfrak{w}}

\newcommand{\SL}{\mathrm{SL}}

\newcommand{\alignbox}[1]{\rlap{$#1$}\phantom{t_{2n-1}^{-1} + K_{c+1}^{R,(1)} b_{2n-2} + K_{c}^{R,(2)} Z_{2k+2}^{(2)}}}

\newcommand{\dfn}[1]{{\textbf{#1}}}

\begin{document}

	\title{Cluster structures from Legendrian double twist knots: a comparative approach}
	
	\author{James Hughes}
\address{University of Southern California, Dept. of Mathematics, Los Angeles, CA 90007, USA}
\email{hughesja@usc.edu}
\author{Jiajie Ma}
\address{Duke University, Dept. of Mathematics, Durham, NC 27708, USA}
\email{jason.ma@duke.edu}

\begin{abstract} 
We undertake a systematic comparison of cluster algebras associated to Legendrian knots of the same smooth knot type. We describe a cluster structure on the $m$-graded augmentation variety of Legendrian representatives of a family of double twist knots that includes all orientably Lagrangian fillable  twist knots. For different Legendrian representatives of the same smooth knot type, we show that these cluster structures are equivalent. Along the way, we characterize the representatives that admit Maslov-$m$ exact Lagrangian fillings, construct a Catalan number of such fillings for each representative, and show that any Maslov-0 fillable representative is Legendrian isotopic to the $(-1)$-framed closure of a positive braid. Finally, we give two Legendrian representatives of a cabled twist knot that yield equivalent cluster structures of infinite type.
\end{abstract}


	\maketitle

\section{Introduction}\label{sec: intro}

\subsection*{Scientific Context}
The classification of Legendrian knots and their exact Lagrangian fillings is a central topic in low-dimensional contact and symplectic topology; see e.g.\ \cite{EliashbergPolterovich96, EtnyreHonda01b, EtnyreHonda05, EliashbergFraser09, YuPan, CasalsGao, CasalsGao23}. 
Starting with Chekanov, Floer-theoretic invariants, in the form of the Chekanov-Eliashberg differential graded algebra (DGA), have played a central role in this pursuit, famously distinguishing a pair of Legendrian twist knots smoothly isotopic to $m(5_2)$ in \cite{Chekanov}. Over a decade later, further refinement of the DGA and related invariants led to a complete classification of Legendrian twist knots in \cite{EtnyreNgVertesi13}. 

More recently, the DGA has also played a key role in efforts to construct and distinguish exact Lagrangian fillings of Legendrian links.\footnote{The classification of Lagrangian fillings is listed in \cite{K3ProblemList} as Problem 1.67.} Specifically, the DGA $\SA(\La)$ of a Legendrian link $\La$ helps to distinguish exact Lagrangian fillings of $\La$ through the data of augmentations, i.e. maps from $\SA(\La)$ to a ground field $\Bbbk$; see e.g. \cite{EHK, YuPan, CasalsNg, CSHW2023}. An exact Lagrangian filling of $\La$ induces an augmentation of $\SA(\La)$, and we refer to the space of such maps as the augmentation variety $\Aug(\La, \Bbbk)$ of $\La$; see Section~\ref{ssec:aug} for a precise definition. 

To date, the most effective tool for distinguishing fillings comes from establishing a connection between the theory of cluster algebras and the augmentation variety or its sheaf-theoretic counterpart, the moduli of microlocal-rank one sheaves. Such a connection was originally investigated in \cite{STWZ} using the combinatorics of plabic graphs and conjugate surfaces. For braid positive Legendrians $\La_\beta$, $\beta\in \Br_n^+$, \cite{GSW} establishes an isomorphism between the coordinate ring $\C[\Aug(\La_\beta)]$ of the augmentation variety and the cluster algebras studied in \cite{ShenWeng}, providing one of the first results  distinguishing infinitely many fillings. This approach was further developed in \cite{CasalsWeng}, where the authors considered Legendrians arising from a combinatorial construction they call a grid plabic graph. They showed that the sheaf moduli of these Legendrians admits a cluster structure in a precise geometric sense that we explain in Section~\ref{ssec: braid_vties}. Their work is the first to produce cluster structures from Legendrians that are not braid-positive, and the family of links they study includes Legendrian representatives of the twist knots $K_{n}$ with $n\in 2\Z, n\leq -4$.   
Given the relative success of the program outlined above, one might ask if cluster theory can also be used to distinguish different Legendrian representatives of the same smooth knot type. This question has a negative answer in the one case in which it is known; the first version of the preprint \cite{CasalsWengPreprint} (now published as \cite{CasalsWeng}) contains a computation describing identical cluster structures associated to two Legendrian representatives of the twist knot $m(7_2)$. However, independent of its utility for distinguishing Legendrian knots, the study of Lagrangian fillings for different Legendrian representatives of the same smooth knot type represents a natural frontier for further exploration.

\subsection*{Main results}

In this work, we undertake the first systematic comparison of cluster structures for different Legendrian representatives of the same smooth knot type. One immediate challenge in approaching this problem lies in finding a likely family of knots to explore: candidates must have multiple Legendrian representatives, admit orientable Lagrangian fillings, and be relatively amenable to the computations involved in producing cluster structures. We start by considering the family of double twist knots, $K_{b, c}$ with $b\leq -5$ odd and $c \geq 2$ even;\footnote{When $c$ is odd, our methods for investigating cluster structures are still valid; we restrict to $c$ even, however, to ensure that $K_{b, c}$ is a knot, allowing us to distinguish certain Legendrian representatives using the results of \cite{OzsvathStipsicz}.} see Section~\ref{ssec: motivation} for an explanation of our choice in knot type. A schematic for $K_{b,c}$ is pictured in Figure~\ref{fig: DoubleTwist} along with a collection of Legendrian representatives $\La[w, c]$. Here, $w$ denotes a word of length $-b-1$ in the letters $Z$ and $S$, while $c$ represents the number of positive crossings in the upper right hand region of the front diagram. While there is not yet a complete classification of the Legendrian representatives of double twist knots, we can apply the methods of \cite{EtnyreNgVertesi13} to distinguish Legendrian representatives of $\La[w, c]$, providing a conjecturally complete collection of max-tb representatives of the form $\La[(ZS)^{z^+}(SZ)^{z^-}(SS)^{\frac{m}{2}}, c]$ for non-negative integers $z^+,$ $z^-$, and $m=(-b-1)-2(z^++z^-)$. 

\begin{center}
		\begin{figure}[h!]{ \includegraphics[width=\textwidth]{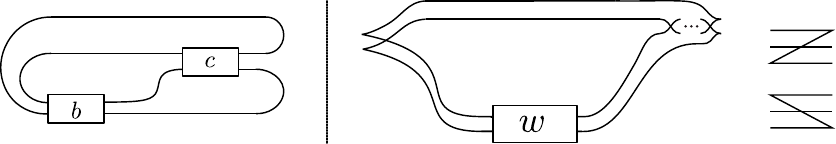}}\caption{A double twist knot $K_{b, c}$ (left) and a front projection for Legendrian representatives $\La[w, c]$  of $K_{b, c}$ (right). The region labeled by $w$ represents a word of length $-b-1$ made up of $S$ tangles (bottom right) and $Z$ tangles (top right), while the region at the top right of the front diagram contains $c$ positive crossings.}
			\label{fig: DoubleTwist}\end{figure}
	\end{center}

One salient feature of double twist knots is that a graded count of augmentations of the Legendrian contact homology DGA distinguishes some, but not all, of the Legendrian representatives of $K_{b, c}$. Nevertheless, the DGAs associated to each Legendrian representative admit the same number of $m$-graded augmentations. This equality suggests a comparison of $m$-graded augmentation varieties and Lagrangian fillings with Maslov number $m$. 

\begin{theorem}\label{thm: main}
Let $w=(ZS)^{z^+}(SZ)^{z^-}(SS)^{\frac{m}{2}}$ and $w'=(ZS)^{{z^+}'}(SZ)^{{z^-}'}(SS)^{\frac{m'}{2}}$. Then there is an isomorphism of augmentation varieties $\Aug_m(\La[w, c]) \cong \Aug_{m'}(\La[w', c])$ with  quasi-equivalent cluster structures, where every cluster chart of $\Aug_m(\La[w,c])$ (resp. $\Aug_{m'}(\La[w', c])$) is induced by an exact Lagrangian filling with Maslov number $m$ (resp. $m'$). 
\end{theorem}

\begin{remark}
    
In the statement of the theorem above, the term `cluster structure' is used in the specific sense of \cite[Section 2.8]{CasalsWeng}. The theorem can be understood as a generalization of the Casals-Weng example mentioned earlier in both the family of Legendrians being considered, as well as the construction of fillings with nonzero Maslov number.
\end{remark} 

In addition to the quasi-cluster equivalences obtained in Theorem~\ref{thm: main}, we also find that the cluster structures are of Dynkin type $A_{n-1}$. These finite type cluster varieties have a Catalan number $C_{n}=\frac{1}{n+1}{2n \choose n}$ of cluster charts that we use to distinguish the fillings we construct. To our knowledge, the following corollary represents the first result distinguishing orientable Lagrangian fillings of nonzero Maslov number.

\begin{corollary}
There are at least $C_{c}$ exact Lagrangian fillings of $\La[(ZS)^{z^+}(SZ)^{z^-}(SS)^{\frac{m}{2}}, c]$ with Maslov number $m$.  
\end{corollary}

As in the case of the fillings of the unique max-tb representative of $(2, n)$ torus links studied in \cite{YuPan}, each of the $C_{c}$ fillings of Theorem~\ref{thm: main} can be realized by a sequence of crossing resolutions, with different fillings corresponding to different orders of resolving these crossings. We conjecture, in line with \cite[Conjecture 5.1]{CasalsLagSkel}, that the fillings we construct comprise an exhaustive collection.

As in previous work on cluster algebras and Legendrians, the main geometric ingredient of the construction is the identification of a collection of homology cycles in a filling that bound Lagrangian compressing disks in the complement. These cycles, referred to as $\L$-compressing cycles, provide both the combinatorial data of the cluster algebra as well as the necessary input for defining regular functions in the coordinate ring $\C[\Aug_m(\La[w, c])]$ that correspond to cluster variables. We wish to emphasize that in this work, we identify these $\L$-compressing cycles without constructing the filling as a Legendrian weave, the sole method previously used. While our method of identifying $\L$-compressing cycles is not substantively different from previous methods, requiring only repeated application of a simple local model that previously appeared in \cite{Hughes2021b}, we believe that this departure is nevertheless noteworthy, particularly because the Lagrangian fillings with nonzero Maslov number cannot be realized as Legendrian weaves. 

\begin{remark}\label{rmk: LennyExample}
L.~Ng has recently shared with the authors an example of a Legendrian knot $\La$ of smooth knot type $12n454$ admitting Maslov-$0$ Lagrangian fillings that is not Legendrian isotopic to the boundary of a Legendrian weave; see Figure~\ref{fig: LennyExample}. In particular, this Legendrian knot admits an augmentation $\epsilon$ with respect to which the second linearized contact homology $LCH_2^\epsilon(\La)$ is nonzero, whereas the linearized contact homology of the $(-1)$-closure of a positive braid can only be supported in degree $-1,0,$ and $1$. On the other hand, $\La$ admits a Maslov-$0$ filling via three saddle cobordisms depicted in Figure~\ref{fig: LennyExample}. This further demonstrates a need for techniques that are applicable beyond Legendrian weaves. 
\end{remark}

\begin{center}
    \begin{figure}[h!]{ \includegraphics[width=.6\textwidth]{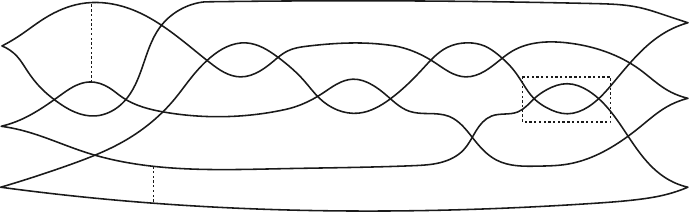}}\caption{Legendrian representative of $12_{n454}$. Dotted lines denote 1-handle attachments and the dotted box represents an application of the local move in Figure~\ref{fig: D4minus} used to construct a Maslov-0 Lagrangian filling.}
			\label{fig: LennyExample}\end{figure}
	\end{center}

\subsubsection*{Braid representatives of $\La[w,c]$} Despite the fact that our method of proof for Theorem~\ref{thm: main} is adapted for fillings with nonzero Maslov number, we also show that the method of Legendrian weaves is sufficient to establish the theorem's conclusions in the Maslov-0 setting.
Legendrian weaves are the primary tool used in many of the constructions linking fillings to cluster theory to date, and they require that the Legendrian link in question be presented as the $(-1)$-closure of a positive braid; see Figure~\ref{fig: -1closure} for a diagram depicting these closures. We produce such Legendrian representatives by an inductive computation.

\begin{theorem}\label{thm: braidification}
    Every Maslov-0 fillable Legendrian of the form $\La[w, c]$ is Legendrian isotopic to the $(-1)$-closure of a positive braid. 
\end{theorem}

\begin{center}
		\begin{figure}[h!]{ \includegraphics[width=.3\textwidth]{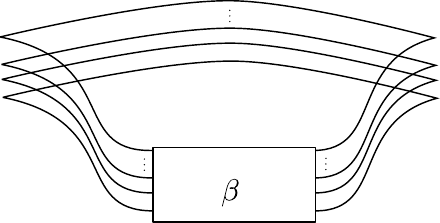}}\caption{Front projection depicting the $(-1)$-closure of a positive braid $\beta$.}
			\label{fig: -1closure}\end{figure}
	\end{center}

In addition to allowing access to Legendrian weaves, our inductive computation sheds light on the question of which Legendrian knots admit Maslov-0 fillings; c.f.\ \cite[Problem 1.68]{K3ProblemList} and \cite[Conjecture 1.3]{HaydenSabloff}. Specifically, Theorem~\ref{thm: braidification} provides additional context for the possible approach proposed in \cite[Question 1.6]{HughesMa25}, which asks whether there exist Maslov-0 fillable Legendrian links not Legendrian isotopic to the $(-1)$-closure of a positive braid. While the question is answered affirmatively by the example of L.~Ng mentioned in Remark~\ref{rmk: LennyExample}, this is in contrast with Theorem~\ref{thm: braidification}, which, to our knowledge, gives the first result resolving this question for multiple Legendrian representatives of a family of smooth knot types.\footnote{The main result of \cite{HughesMa25} does not guarantee an answer to this question for every Legendrian representative of the class of links we study, namely quasipositive 3-braid closures.} In particular, Theorem~\ref{thm: braidification} implies that every Maslov-0 fillable Legendrian twist knot is isotopic to the $(-1)$-closure of a positive braid. The question of whether every fillable smooth knot type admits a fillable Legendrian representative that is the $(-1)$-closure of a positive braid remains open. 

Our work constructing and distinguishing Maslov-$m$ fillings together with the discussion of Maslov-0 fillability above raises the following natural refinement of \cite[Problem 1.68]{K3ProblemList}:

\begin{question}\label{q: Maslov-m fillability}
For a fixed $m\in 2\Z$, determine the smooth knot types that have Legendrian representatives admitting Maslov-$m$ exact Lagrangian fillings.
\end{question}

Note that the Maslov-$m$ fillable max-tb representatives we consider in this work exhibit the property that each is related to a Legendrian representative admitting Maslov-0 fillings by a sequence of $SZ$ moves. A brief search of the Legendrian knot atlas \cite{ChongchitmateNgPre} reveals no likely examples of small crossing number knots with Maslov-$m$ fillable Legendrian representatives for $m\in 2\Z_{>0}$ but without any Maslov-0 fillable Legendrian representatives. Moreover, the obstruction we use in \cite{HughesMa25} to rule out fillability of certain quasipositive 3-braid closures works equally well in the orientable and Maslov-0 cases. While this somewhat sparse evidence suggests that the answer to the $m=0$ version of Question~\ref{q: Maslov-m fillability} is identical to the $m$-agnostic version posed in \cite[Problem 1.68]{K3ProblemList}, we know of no compelling geometric reason for why this might be the case.

\subsubsection*{Comparing Legendrian representatives with infinitely many fillings} In order to extend our comparison of different Legendrian representatives beyond the finite cluster type, we consider two max-tb Legendrian representatives of the $(2, 6)$-cable of the twist knot $m(7_2)$. We use Legendrian weaves to construct Lagrangian fillings for these two Legendrian representatives and exhibit a quasi-cluster equivalence between their augmentation varieties. We also verify that each of these representatives admits infinitely many fillings.

\begin{theorem}\label{thm: infty-equivalent}
    Let $\La$ and $\La'$ be the Legendrian representatives of the $(2, 6)$-cable of $m(7_2)$ isotopic to the $(-1)$-closures of the braids $\beta_{(2, 6)}$ and $\beta'_{(2,6)}$ recorded in Section~\ref{ssec: cable_computation}. Then $\La$ and $\La'$ admit infinitely many fillings, and $\Aug(\La)$ and $\Aug(\La')$ are quasi-equivalent cluster varieties of infinite type.    
\end{theorem}

Note that the braids $\beta_{(2,6)}$ and $\beta'_{(2,6)}$ contain a half twist and satisfy the technical condition that their Demazure product is equal to $w_0$. Consequently, \cite[Theorem 1.1]{CGGLSS} implies that their augmentation varieties admit cluster structures.  While smaller cables of twist knots (e.g. $n$-copies) are Lagrangian fillable, they are not isotopic to a braid of the form considered in loc. cit. and it is not immediately clear how to associate a cluster structure to their augmentation varieties.

\subsubsection*{Smooth isotopy invariance of cluster structures}

We conclude this introduction with a brief discussion of reasons we might expect to see smooth isotopy invariance of cluster structures associated to a knot type. Given a Legendrian $\La$ with DGA $(\mathcal{A}(\La), \partial)$, we denote by $I$ the two-sided ideal of $SA$ generated by $\{\partial a\ |\ a \text{ is a generator of } \SA(\La)\}$. The characteristic algebra of $\La$ is defined to be the quotient $\mathcal{A}(\La)/I$. Considered without grading, the abelianized characteristic algebra $\mathcal{A}(\La)/I$ is conjectured to only depend on the smooth knot type of $\La$ \cite[Conjecture 3.14]{Ng03}. Indeed, one can readily demonstrate that the (abelianized ungraded) characteristic algebra is invariant under $SZ$-moves, providing partial evidence for this conjecture. 
Replacing the characteristic algebra by the $\Aug_2(\La)$ leads to a variant of the aforementioned conjecture, and the equivalence of cluster structures obtained in Theorems~\ref{thm: main} and \ref{thm: infty-equivalent} provides similar partial evidence for the following: 

  \begin{conjecture}\label{conj}
  Let $\La$ be a Legendrian knot and suppose $\Aug_2(\La)$ admits a cluster structure induced by orientable fillings of $\La$. If $\La'$ is smoothly isotopic to $\La$, then $\Aug_2(\La')$ admits a cluster structure quasi-equivalent to that of $\Aug_2(\La).$ Moreover, the embedded orientable fillings of $\La$ and $\La'$ are both in bijection with the cluster seeds of $\Aug_2(\La)$. 
  \end{conjecture}

The second part of Conjecture~\ref{conj} can be seen as extending \cite[Conjecture 5.1]{CasalsLagSkel} from the Maslov-0 setting to cover orientable fillings of any Maslov number. When $\Aug_2(\La)$ admits multiple irreducible components, the statement of Conjecture~\ref{conj} should be understood as restricting to a given irreducible component of $\Aug_2(\La)$ or possibly a subset of an irreducible component induced by a collection of embedded fillings.

\begin{remark}
The appearance of the 2-graded augmentation variety above is intended to account for the possibility of a Legendrian knot admitting fillings with different Maslov numbers. In Section~\ref{ssec:maslov}, we show that this phenomenon does occur, describing a family of Legendrians $\La_k$ that admit Maslov-0 fillings \emph{and} Maslov-$k$ fillings for all even $k\geq 0$.
\end{remark}

In addition to the smooth isotopy invariance of the characteristic algebra of $\La$, the 2-graded ruling polynomial of $\La$ is also invariant under smooth isotopy. This polynomial is equal to the minimal $a$-degree term of the HOMFLY polynomial \cite[Theorem 4.1]{Rutherford06} and records the number of 2-graded rulings of $\La$; see loc. cit. for further background and details. If we denote by $\Delta$ the half twist element in the braid group, then for $(-1)$-closures $\La
(\beta\Delta)$ satisfying the hypotheses of \cite[Theorem 1.1]{CGGLSS}, one can use \cite[Theorem 1.1]{ACSHLLW} to show that the number of $\L$-compressing cycles of a filling $L$ of $\La(\beta\Delta)$ is equal to the subleading coefficient of the (0-graded) ruling polynomial. Beyond the setting of such braid closures, experimental evidence confirms the use of the subleading coefficient of the 0-graded ruling polynomial as a surprisingly robust heuristic for the number of $\L$-compressing cycles of Maslov-0 fillings. We expect that the 2-graded version contains similar information, as it does in our case. The remaining terms of the polynomial record point counts of $\Aug_2(\La)$ \cite{HenryRutherford15}, and they likely contain information about the intersection of $\L$-compressing cycles. The smooth isotopy invariance of the 2-graded ruling polynomial therefore provides evidence for the smooth isotopy invariance of cluster structures. 

We wish to note, however, that the use of the ruling polynomial as a heuristic is somewhat limited in full generality, as a recent result of Capovilla-Searle and Pan shows that any polynomial with terms of even degree can be obtained as the ruling polynomial of a Legendrian link \cite{capovillasearle2026rulingpolynomialslegendrianlinks}. In particular, their examples are Lagrangian fillable but the fillings don't appear to bound $\L$-compressing disks, and in many cases can be shown not to. Rather, the terms in the ruling polynomials of the examples they construct seem to encode smooth isotopy classes of (possibly immersed) Lagrangian fillings.

\subsection*{Statement on AI use} The authors acknowledge the use of generative AI to write the Sage code used to produce the mutation sequence in the proof of Theorem \ref{thm: infty-equivalent-body}. Generative AI was not used in formulating, writing, or proving any of the remaining content in this paper.

\subsection*{Acknowledgments}
Many thanks to Lenny Ng for posing the question that sparked this project and for his helpful advice throughout. Thanks also to Agniva Roy and Daping Weng for helpful conversations. JH was partially supported by NSF grant DMS-2446019.

\subsection*{Organization}
In Section~\ref{sec: geometric_background} we give the necessary symplectic topology background, followed by algebraic background on the Chekanov-Eliashberg DGA and cluster varieties in Section~\ref{sec: algebra_background}. Section~\ref{sec: DoubleTwistKnots} distinguishes Legendrian representatives of double twist knots and includes a proof of Theorem~\ref{thm: braidification}. In Section~\ref{sec: DGA computations}, we compute the augmentation variety $\Aug_m(\La[w,c])$. Building on this, we describe its cluster structure and prove Theorem~\ref{thm: main} in Section~\ref{sec: cluster-double-twist}. Finally, we conclude the paper with our comparison of two cables of $m(7_2)$ in Section~\ref{sec: cable_computation}. 

\section{Geometric background}\label{sec: geometric_background}

\subsection{Legendrians links and Lagrangian fillings}
\label{ssec:contact}

Let $(Y,\xi)$ be a $(2n+1)$-dimensional contact manifold. An $n$-dimensional submanifold is \dfn{Legendrian} if it is everywhere tangent to the contact plane field $\xi$. In this paper, we consider primarily the Euclidean space $\R^3$ equipped with the standard contact structure $\xi_{st} = \ker \alpha_{st}$, where $\alpha_{st} = dz-ydx$. In this dimension, the Legendrian submanifolds are knots and links, which we assume to be oriented. We study these links up to Legendrian isotopy, meaning an ambient isotopy through a family of Legendrians. Since a generic Legendrian link avoids a point in $\R^3$, we sometimes also treat our links as residing in the contact $3$-sphere $(\mathbb{S}^3, \xi_{st})$ \cite[Section 3.2]{Geiges08}, where we retain the notation $\xi_{st}$ for the standard contact structure. 

We will often visualize a Legendrian link via its image under the front projection $\Pi_{xz}:\R^3 \to \R^2_{xz}$, or the Lagrangian projection $\Pi_{xy}:\R^3 \to \R^2_{xy}$. We assume throughout that all Legendrian links are chord generic, meaning that $\Pi_{xy}(\La)$ is a generic immersion whose self-intersections are transverse double-points corresponding to Reeb chords. There are two classical invariants associated to a Legendrian link. The \dfn{Thurston-Bennequin number} $\tb(\Lambda)$ is the writhe of $\Pi_{xz}(\Lambda)$ minus the number of left (or right) cusps, and the \dfn{rotation number} $\rot(\Lambda)$ is given by $\frac{1}{2}(c_- - c_+)$ where $c_-,c_+$ are the number of cusps oriented downwards and upwards in $\Pi_{xz}(\La)$, respectively.

Next, we turn to the exact Lagrangian cobordisms of Legendrian links in $\R^3$. Recall that the symplectization $\Symp(M, \ker(\alpha))$ of a contact manifold $(M, \ker(\alpha))$ is the symplectic manifold $(\R_t\times M, d(e^t\alpha))$. Given two Legendrian links $\La_-, \La_+\subseteq (\R^3,\xi_{st})$, 
 an \dfn{exact Lagrangian cobordism} $L \subseteq \Symp(\R^3, \ker(\alpha_{st}))$ from $\La_-$ to $\La_+$ is an embedded exact Lagrangian submanifold such that for some $T>0$,

	\begin{enumerate}
		\item $d(e^t\alpha_{st})|_\Sigma=0$,
		\item $\Sigma\cap ((-\infty, T]\times \R^3)=(-\infty, T]\times \La_-$ 
		\item $\Sigma\cap ([T, \infty)\times \R^3)=[T, \infty) \times \La_+$ 
		\item $e^t\alpha_{st}|_\Sigma=df$ for some function $f: \Sigma\to \R$ that is constant on $(-\infty, T]\times \La_-$ and $[T, \infty)\times \La_+$. 
	\end{enumerate}

We study exact Lagrangian cobordisms up to Hamiltonian isotopy, see e.g. \cite[Section 3.6]{OhSympTopVol1}, and denote a concatenation or a composition of two exact Lagrangian cobordisms $L_1,L_2$ from top to bottom by $L_1 \circ L_2$. An \dfn{exact Lagrangian filling} of a Legendrian link $\La\subseteq (\R^3, \xi_{st})$ is an exact Lagrangian cobordism $L$ from $\emptyset$ to $\La$ in $\Symp(\R^3, \ker(\alpha_{st}))$. Mirroring our perspective on Legendrians in $\mathbb{S}^3$, we may view $L$ as embedded in the symplectic 4-ball with boundary $\partial L\subseteq (\S^3, \xi_{st})$. It is a standard fact from \cite[Theorem 1.3]{Chantraine10} that a Lagrangian fillable Legendrian link $\La$ must attain the maximal Thurston-Bennequin number in its knot type and that the exact Lagrangian fillings of $\La$ are topologically the same. Specifically, we have $\chi(L) = -\tb(\La)$ for a filling $L$ of $\La$.

We will often construct exact Lagrangian cobordisms between Legendrian links in $\R^3$ by composing \dfn{elementary cobordisms} introduced in \cite{EHK}. The three types of elementary cobordisms are:

\begin{itemize}
    \item \textbf{Isotopy cobordisms.} Given a Legendrian isotopy $\varphi_t$ from $\La_+$ to $\La_-$, a small perturbation of the trace of $\varphi_t$ yields an exact Lagrangian cylinder from $\La_-$ to $\La_+$ in $\Symp(\R^3, \ker(\alpha_{st}))$; 
    \item \textbf{Saddle cobordisms.} A Reeb chord of a Legendrian link $\La_+$ is called \emph{contractible} if there is a Legendrian isotopy of $\La_+$ that induces a planar isotopy of $\Pi_{xy}(\La_+)$ and ends in a Legendrian link where the height of the Reeb chord is arbitrarily small. Given a contractible Reeb chord of $\La$ corresponding to a positive crossing in $\Pi_{xy}(\La_+)$, the saddle cobordism is a Lagrangian $1$-handle attachment at this chord. In the Lagrangian projection, this is represented by replacing the crossing of $\La_+$ with its oriented resolution to produce $\La_-$. See Figure~\ref{fig: Saddle} for a diagrammatic illustration in the front and Lagrangian projections. We will also refer to this cobordism as a \emph{pinch move} or as ``resolving" the Reeb chord.
    \item \textbf{Minimum cobordisms.} By \cite{EliashbergPolterovich96}, the max-tb Legendrian unknot $U$ admits a unique exact Lagrangian filling up to isotopy. If a Legendrian link $\La_+$ is a split union of $\La_-$ and $U$, the minimum cobordism from $\La_-$ to $\La_+$ is the disjoint union of the cylinder $\R \times \La_-$ and the unique filling of $U$. Topologically, this constitutes a $0$-handle attachment.
\end{itemize}

An exact Lagrangian cobordism is \dfn{decomposable} if it is Hamiltonian isotopic to a finite concatenation of elementary cobordisms. Note that by \cite{DRG26}, not all exact Lagrangian cobordisms are decomposable.

\subsection{Maslov index of a Lagrangian submanifold}
\label{ssec:maslov}

For our purpose of constructing exact Lagrangian fillings with nonzero Maslov numbers, in this section, we review the definition of this invariant and we develop the necessary machinery to compute the Maslov number for a decomposable cobordism between two Legendrian links. Readers who are more interested in the combinatorial construction of Maslov-$m$ fillings need only consult Theorem~\ref{thm: MaslovNumber} in order to consider themselves well-equipped to explore further sections. Following the general framework of \cite{Seidel08}, our approach relies on the notion of a \emph{grading} or \emph{Maslov potential} on a Lagrangian submanifold. This concept adapts naturally to the Legendrian setting, yielding the Maslov potential used to compute the grading of Reeb chords \cite{EkholmEtnyreSullivan05a}, which is further specialized in \cite{Ng03} into a combinatorial Maslov potential for Legendrian links in $\R^3$. By clarifying the relationship between these perspectives, we show that the Maslov number for a decomposable cobordism from $\La_-$ to $\La_+$ can be computed purely combinatorially from a Maslov potential on $\La_+$.

Let $\omega = \omega_{st}$ or $d(e^t\alpha_{st})$. We begin by reviewing the definition of the Maslov index of a Lagrangian immersion in $(\R^{2n}, \omega)$ following the exposition of \cite[Section 2]{McDuffSalamon95} and \cite[Section 2.2]{EkholmEtnyreSullivan05a}. Let $\LGr_n$ denote the space of Lagrangian subspaces of $(\R^{2n}, \omega_{st})$. Define the \dfn{Maslov cycle} $\Sigma(n)$ to be the singular hypersurface of $\LGr_n$ of codimension $1$ consisting of all Lagrangian subspaces which intersect the vertical $\{0\} \times \R^n$ non-transversely. For a Lagrangian immersion $\phi:L \to (\R^{2n}, \omega)$ and a loop $\gamma$ in $L$, a global trivialization of $\R^{2n}$ identifies the tangent planes along $\phi(\gamma)$ with a loop $\Gamma$ in $\LGr_n$. Here, we use the canonical trivialization of $\R^{2n}$ if $\omega = \omega_{st}$, and the trivialization with respect to the splitting $T(\R \times \R^{2n-1}) = \langle \partial_t, \partial_z \rangle \oplus \xi_{st}$ if $\omega = d(e^t\alpha_{st})$. The \dfn{Maslov index} of $\phi$ is then the homomorphism $H_1(L) \to \Z$ that sends $[\gamma] \in H_1(L)$ to the intersection number of this loop $\Gamma$ with $\Sigma(n)$ in $\LGr_n$. We refer the reader to \cite[Section 2]{McDuffSalamon95} or \cite[Section 2.2]{EkholmEtnyreSullivan05a} for a specific method of calculating this intersection number using the signature of the crossing form. In dimension $n = 1$, we note that the Maslov cycle is a point on $\LGr_1 \cong \S^1$ corresponding to $\{0\} \times \R$. Thus, the Maslov index is simply the degree of the loop $\Gamma:\S^1 \to \S^1$.

By minimizing the Maslov index over all loops with non-negative Maslov indices, we obtain a Lagrangian isotopy invariant.

\begin{definition}
    The \dfn{Maslov number} of a Lagrangian immersion $L \subset (\R^{2n}, \omega)$ is the smallest non-negative integer in the image of the Maslov index of $L$.
\end{definition}

We say that a Lagrangian is Maslov-$m$ if its Maslov number is equal to $m$. We will determine the Maslov number of a decomposable exact Lagrangian filling $L \subset (\R^4,d(e^t\alpha_{st}))$ by constructing a Maslov potential on $L$. To define this concept, we first view the Maslov index of $L$ as a cohomology class in $H^1(L,\Z)$, which in turn corresponds to a homotopy class $[\alpha] \in  [L,\S^1]$ under the canonical isomorphism $H^1(L,\Z) \simeq [L,\S^1]$; here we use $[M, N]$ to denote the set of homotopy classes of maps from $M$ to $N$. Following the exposition of \cite[Section (11j)]{Seidel08}, we define a specific representative map $\alpha:L \to \S^1$ as follows. Let $\eta^2:\La_\C^n(\R^{2n})^{\otimes 2} \to \C$ be the isomorphism of complex lines defined by $\eta^2 = \Omega \otimes \Omega$, where $\Omega$ is the complex volume form on $\R^{2n}$ determined by a compatible complex structure. Since $L$ is Lagrangian, the \dfn{phase map} $\alpha:L \to \S^1$ given by
\begin{equation*}
    \alpha(p) := \frac{\eta^2(v_1,\ldots,v_n,v_1,\ldots,v_n)}{|\eta^2(v_1,\ldots,v_n,v_1,\ldots,v_n)|},
\end{equation*}
where $(v_1,\ldots,v_n)$ is a basis of $T_pL$, is well-defined independent of the choice of basis. For our applications, we will equip $\R^2$ with the standard complex structure and the symplectization $\R^4 = \R \times \R^3$ with the cylindrical complex structure \cite[Section 3.6]{EHK}. Then, the phase map of a Lagrangian immersion $\phi:\S^1 \to \R^2$ is given by the square of its complex unit tangent vector at each point. This map coincides with the loop $\Gamma:\S^1 \to \S^1 \simeq \LGr_1$ used to compute the Maslov index of $\S^1$ mentioned above, where the Maslov cycle is identified with $-1 \in \S^1$ in the codomain. 

\begin{definition}
    An \dfn{$m$-graded Maslov potential} of $L$ to be a map $\alpha^\sharp:L \to \R/m\Z$ that lifts the phase map of $L$.
\end{definition}

By the lifting criterion, the existence of an $m$-graded Maslov potential on $L$ is governed by the Maslov number of $L$.
\begin{lemma}
\label{lemma: MaslovPotentialCriterion}
    There exists an $m$-graded Maslov potential on $L$ if and only if $m$ divides the Maslov number of $L$. An $m$-graded Maslov potential on $L$, if it exists, is well-defined up to a global $\Z/m\Z$ shift on each connected component of $L$.
\end{lemma}

Noting that the Lagrangian projection of a Legendrian link is a Lagrangian immersion in $\R^2$, we extend the definitions above to Legendrian links by applying them to their Lagrangian projections. Several comments are in order to clarify these definitions and to connect them to other notions in the literature. First, a Legendrian knot has Maslov number equal to twice its rotation number \cite{Chekanov}. Thus, a Legendrian link has Maslov number equal to the greatest common divisor of $2\rot(\Lambda_i)$ over all its components $\Lambda_i$.

Second, if $m$ divides the Maslov number of $\La$, we can define a $\Z/m\Z$-grading of a Reeb chord of $\La$ using a Maslov potential on $\La$. Suppose $c$ is a Reeb chord from $c^-$ to $c^+$. Locally near $c^+$ and $c^-$, $\La$ is the $1$-jet of Morse functions $f_u$ and $f_l$, respectively, defined on a neighborhood of $\pi_x(c)$. Moreover, $\pi_x(c)$ is a critical point of the difference function $f_{ul}:= f_u - f_l$. Let $\ind_{f_{ul}}(\pi_x(c))$ denote the Morse index of $f_{ul}$ at $\pi_x(c)$. Given a $\Z/m\Z$-graded Maslov potential $\alpha^{\sharp}$, we define the $\Z/m\Z$-grading of $c$ by
\begin{equation*}
    |c|_{\alpha^\sharp} := \lfloor \alpha^\sharp(c^+) \rfloor - \lfloor \alpha^\sharp(c^-) \rfloor + \ind_{f_{ul}}(\pi_x(c))-1 \pmod m.
\end{equation*}
By Lemma~\ref{lemma: MaslovPotentialCriterion}, this grading is independent of the choice of Maslov potential when $\La$ is a knot. In fact, it follows from \cite[Lemma 3.4]{EkholmEtnyreSullivan05a} that it equals the Conley-Zehnder index of $c$ minus $1$ modulo $m$, where the Conley-Zehnder index of $c$ is the Maslov index of a path from $c^+$ to $c^-$ in $\La$ followed by the Reeb flow from $c^-$ to $c^+$; we refer the reader to \cite[Section 2.2]{EkholmEtnyreSullivan05a} for the precise definition. When $\La$ is a link, the grading of $c$ depends on the choice of $\alpha^\sharp$, although the difference in the gradings of two Reeb chords does not. We omit the subscript from the notation when the grading is independent of $\alpha^\sharp$ or when the choice is clear from the context.

Finally, phase maps and Maslov potentials of a Legendrian link are conveniently studied via their pullbacks to the front projection. This perspective facilitates a combinatorial definition of a Maslov potential on an oriented Legendrian link $\La$, which streamlines the computation of Reeb chord gradings \cite{Ng03}. Recall that the cusps of $\Pi_{xz}(\La)$ correspond to the vertical tangencies of $\Pi_{xy}(\La)$. Consequently, the pullback of the phase map $\alpha: \Pi_{xz}(\La) \to \S^1$ intersects the Maslov cycle $-1 \in \S^1$ positively at cusps oriented upwards and negatively at cusps oriented downwards. In particular, $\lfloor \alpha^\sharp \rfloor$ induces a locally constant function on the strands of $\Pi_{xz}(\La)$ with values in $\Z/m\Z$, such that near a cusp point, its value on the upper sheet is exactly $1$ more than that on the lower sheet modulo $m$. Motivated by this observation, we define a \dfn{combinatorial Maslov potential} on $\La$ to be a locally constant function
\begin{equation*}
    \mu: \Pi_{xz}(\La)\backslash \{\text{cusps}\} \to \Z/m\Z,
\end{equation*}
that satisfies this exact jump condition at the cusp points. This formulation allows us to combinatorially define the $\Z/m\Z$-grading of Reeb chords by
\begin{equation*}
    |c| = \mu(\Pi_{xz}(c^+)) - \mu(\Pi_{xz}(c^-)) + \ind_{f_{ul}}(\pi_x(c))-1 \pmod m.
\end{equation*}
without requiring an a priori Maslov potential $\alpha^\sharp$.\footnote{When $\La$ is obtained as Ng's resolution of a front \cite{Ng03}, every Reeb chord of $\La$ corresponds to a \emph{maximum} of a difference function. Thus, the $\Z/m\Z$-grading of a Reeb chord $c$ associated to a $\Z/m\Z$-graded Maslov potential is simply $|c| = \mu(c^+) - \mu(c^-)$.}

Let $L \subset (\R^4,\omega_{st})$ be an exact Lagrangian cobordism from the Legendrian link $\La_-$ to the Legendrian link $\La_+$ in $(\R^3,\xi_{st})$. We now connect the Maslov potential of $L$ to the Maslov potentials of $\La_-$ and $\La_+$. Recall that $(\R^4,\omega_{st})$ is equipped with the cylindrical complex structure and the associated complex volume form is $\Omega_{cyl} = (dt+idz)\wedge (dx+idy)$. Denote by $\Omega_{st} = dx+idy$ the standard complex volume form on $\R^2$.

\begin{proposition}
\label{proposition: maslov-ext}
    An $m$-graded Maslov potential on $L$ restricts to $m$-graded Maslov potentials on $\La_\pm$. Conversely, if $\La_+$ (resp. $\La_-$) is a knot and $m$ divides the Maslov number of $L$, then an $m$-graded Maslov potential on $\La_+$ (resp. $\La_-$) extends to $L$.
\end{proposition}
\begin{proof}
    A direct computation using the complex volume form $\Omega_{cyl}$ shows that the phase map of $L$ evaluated at a point $(t,p)$ in the cylindrical ends $\R \times \La_{\pm}$ reduces to the phase map of $\La_{\pm}$ evaluated at $p$. Thus, the restriction of a Maslov potential on $L$ to $\{t\} \times \La_{\pm}$ is a Maslov potential on $\La_{\pm}$.

    Next, for $* = +$ or $-$, assume that $\La_*$ is a knot equipped with an $m$-graded Maslov potential $\alpha^\sharp_*$. If $m$ divides the Maslov number of $L$, it follows from Lemma~\ref{lemma: MaslovPotentialCriterion} and above that $L$ admits an $m$-graded Maslov potential $\alpha_L^\sharp$ and it restricts to a Maslov potential on $\La_*$. Applying Lemma~\ref{lemma: MaslovPotentialCriterion} again, we obtain a Maslov potential on $L$ that extends $\alpha^\sharp_*$ after a global $\Z/m\Z$ shift of $\alpha_L^\sharp$.
\end{proof}

In general, an $m$-graded Maslov potential on a \emph{link} $\La_+$ or $\La_-$ need not extend to a Maslov potential of $L$ even if $m$ divides the Maslov number of $L$. A standard algebraic topology argument using the long exact sequence of the pair $(L,\La_{\pm})$ shows that extending arbitrary potentials is obstructed unless the restriction map $i^*:H^0(L;\Z/m\Z) \to H^0(\La_{\pm};\Z/m\Z)$ is surjective. When $L$ is an elementary cobordism, however, we can constructively determine which Maslov potentials of $\La_+$ extend to $L$. The following proposition is key to determining the Maslov number of a decomposable Lagrangian cobordism and provides a combinatorial method for computing it.

\begin{proposition}
\label{proposition: DecomposableMaslovExtension}
    If $L$ is a minimum or isotopy cobordism, then any $m$-graded Maslov potential on $\La_+$ extends to $L$. If $L$ is a saddle cobordism obtained by resolving a contractible Reeb chord $c$, then an $m$-graded Maslov potential $\alpha_+^\sharp$ on $\La_+$ extends to $L$ if and only if $|c|_{\alpha_+^\sharp} = 0$. 
\end{proposition}
\begin{proof}
    If $L$ is a minimum cobordism of $\La_+ = \La_- \cup U$, then $L$ is a union of a cylinder over $\La_-$ and a disk that bounds the unknot $U$. By the proof of Proposition~\ref{proposition: maslov-ext}, any $m$-graded Maslov potential $\alpha_+^\sharp$ on $\La_+$ extends uniquely to a potential $\alpha_L^\sharp$ on $L$ with $\alpha_L^\sharp(t,p) = \alpha_+^\sharp(p)$ for $p \in \La_-$.

    If $L$ is an isotopy cobordism induced by a Legendrian isotopy $\tau$ of $\La_+$, by \cite[Lemma 6.1]{EHK}, we can parametrize $L$ by $(t, x(t,s), y(t,s),z(t,s)+H(t,s))$ where $t$ is the time coordinate for the isotopy and $H$ is an arbitrarily small contact Hamiltonian $H$ in the $C^1$-metric defined by $H(t,s) = z_t-yx_t$.\footnote{While \cite[Lemma 6.1]{EHK} only uses a $C^0$-small $H$, the $C^1$-small property can be achieved by scaling the time coordinate of the isotopy with a factor $\frac{1}{T}$ for $T >>0$.} Using the tangent vectors $v_t,v_s \in T_{L(t,s)}L$ with respect to this parametrization, we have
    \begin{equation*}
        \Omega_{cyl}(v_t,v_s) = (x_s + iy_s)\left(1+i(H+H_t)-iH_s\frac{x_t+iy_t}{x_s+iy_s} \right).
    \end{equation*}
    Since $H$ is $C^1$-small, it follows that $\Omega_{cyl}(v_t,v_s)$ is $C^0$-close to $\Omega_{st}(x_s,y_s)$ and the phase map of $L$ is $C^0$-close to the phase map of $\La_t := \tau_t(\La_+)$. By a $C^0$-small perturbation of the pushforward potential $\alpha_+^\sharp \circ \tau_t^{-1}$ on $\La_t$ for each $t$, we obtain a potential on $L$ as desired.

    If $L$ is a saddle cobordism obtained by resolving a contractible Reeb chord $c$ of $\La_+$, then $L$ is cylindrical over a region of $\La_+$ outside of a neighborhood $U$ of $c$. The subset of $L$ in $\R \times U$ is part of a Lagrangian disk whose boundary consists of the two strands in $\La_+$ and two cusps in $\La_-$ in the front model, connected by four vertical lines in the $t$-direction at their endpoints; see Figure~\ref{fig: Saddle} (left). Note that $c$ corresponds to a minimum of $f_{ul}$, and the value of $\lfloor \alpha^\sharp_+ \rfloor$ at the endpoints of the upper strand of Figure~\ref{fig: Saddle} (left) is $|c|_{\alpha^\sharp_+}+1$ greater than its value at the endpoints of the lower strand. 
    
    By the proof of Proposition~\ref{proposition: maslov-ext}, the potential $\alpha_+^\sharp$ extends uniquely to the cylindrical part of $L$. In particular, this extension restricts to $\La_-$ in the cylindrical end away from the two cusps, denoted by $\alpha_-^\sharp$. Moreover, the values of $\alpha_-^\sharp$ at the endpoints of the two cusps are the same as the values of $\alpha_+^\sharp$ at the end points of the two strands in $\La_+$ connected by $c$. By the jump condition at a cusp, we know that $\alpha_-^\sharp$ extends over the two cusps if and only if $\lfloor \alpha_-^\sharp \rfloor$ on the upper sheet of the cusp is $1$ more than its value on the lower sheet of the cusp. This condition is thus equivalent to $|c|_{\alpha^\sharp_+} + 1= 1$. By Proposition~\ref{proposition: maslov-ext}, we conclude that $\alpha^\sharp_+$ extends to $L$ only if $|c|_{\alpha^\sharp_+} = 0$.

    If $|c|_{\alpha^\sharp_+} = 0$ so that $\alpha_-^\sharp$ extends to an $m$-graded potential on $\La_-$, we obtain a well-defined $m$-graded potential on the boundary of the Lagrangian disk attached to the cylindrical part of $L$. Since a Lagrangian disk has Maslov number $0$ and one boundary component, it follows from the proof of Proposition~\ref{proposition: maslov-ext} that the Maslov potential extends over the disk and thus, globally.
\end{proof}

We are now in position to prove the main theorem of this section. Suppose $L = L_1 \circ L_2 \circ \cdots \circ L_n$ is a concatenation of elementary cobordisms $L_i$ from $\La_{i+1}$ to $\La_i$, where $\La_+ := \La_1$ and $\La_- := \La_{n+1}$. Suppose moreover $k$ of these are saddle cobordisms $L_{i_1},\ldots,L_{i_k}$, obtained by resolving a Reeb chord $c_j \in \La_{i_j}$ for $j = 1,\ldots,k$. For $i < j$, we denote by $L_{[i,j]}$ the contatenation of cobordisms $L_i \circ L_{i+1} \circ \cdots \circ L_j$.

Let $m_1 := 2r(\La_+)$ denote the Maslov number of $\La_+$, where $r(\La_+)$ is the greatest common divisor of the rotation numbers of the components of $\La_+$. By Proposition~\ref{proposition: DecomposableMaslovExtension}, a $\Z/m_1\Z$-graded Maslov potential on $\La_+$ extends over $L_{[1,i_1-1]}$ to a $\Z/m_1\Z$-graded potential on $\La_{i_1}$. Moreover, if $c_1$ is a Reeb chord between two connected components of $L_{[1,i_1-1]}$, we shift the potential on one component so that the grading of $c_1$ is $0$. Let $|c_1|$ denote the grading of $c_1$ with respect to this potential and let $m_2 := \gcd(m_1,|c_1|)$. 

Composing the $\Z/m_1\Z$-graded Maslov potential on $L_{[1,i_1-1]}$ with the covering map $\R/m_1\Z \to \R/m_2\Z$ yields a $\Z/m_2\Z$-graded Maslov potential with respect to which $c_1$ has grading $0$. By Proposition~\ref{proposition: DecomposableMaslovExtension}, it extends over $L_{[1,i_2-1]}$ to a $\Z/m_2\Z$-graded potential on $\La_{i_2}$. Again, we shift the potential if $c_2$ connects two connected components of $L_{[1,i_2-1]}$. Let $|c_2|$ denote the grading of $c_2$ with respect to this potential. 

Proceeding inductively, we define $|c_j|$ to be the grading of $c_j$ induced by the $\Z/m_j\Z$-graded Maslov potential on $L_{[1,i_j-1]}$, where $m_j = \gcd(2r(\Lambda_+),|c_1|,\ldots,|c_{j-1}|)$.

\begin{theorem}
\label{thm: MaslovNumber}
    The Maslov number of $L$ is $m(L) = \gcd(2r(\Lambda_+),|c_1|,\ldots,|c_k|)$.
\end{theorem}
\begin{proof}
     By Lemma~\ref{lemma: MaslovPotentialCriterion}, the construction of the potential ensures that the Maslov number $m(L)$ is a multiple of $\gcd(2r(\Lambda_+),|c_1|,\ldots,|c_k|)$. To rule out strict multiples, we proceed by induction. First, since the boundary link $\La_+$ generates $H_1(L_{[1,i_1-1]})$, the Maslov number of $L_{[1,i_1-1]}$ is exactly $m_1 = 2r(\La_+)$. 

     If $c_1$ connects two different components of $L_{[1,i_1-1]}$, then we know by construction that the extended Lagrangian $L_{[1,i_1]}$ retains the same Maslov number $m_2 = m_1$. If $c_1$ has its endpoints on the same component of $L_{[1,i_1-1]}$, assume to the contrary that the Maslov number $m$ of $L_{[1,i_1]}$ is a strict multiple of $\gcd(m_1,|c_1|)$. Since $m$ must also divide $m_1$, composing the prescribed $\Z/m_1\Z$-graded potential on $L_{[1,i_1-1]}$ with the covering map $\R/m_1\Z \to \R/m\Z$ yields a $\Z/m\Z$-graded potential. It follows from our assumption that the grading of $c_1$ with respect to this new potential, $|c_1| \mod m$, is strictly nonzero. Since this $\Z/m\Z$-potential is well-defined up to a global shift on the component $c_1$ belongs to, Proposition~\ref{proposition: DecomposableMaslovExtension} dictates that no $\Z/m\Z$-graded Maslov potential on $L_{[1,i_1-1]}$ can extend to $L_{[1,i_1]}$, which is a contradiction. 
     
     We have thus proven that the Maslov number of $L_{[1,i_1]}$ is $m_2 = \gcd(m_1,|c_1|)$. Since a basis of $H_1(L_{[1,i_1]})$ generates $H_1(L_{[1,i_2-1]})$, the Lagrangian $L_{[1,i_2-1]}$ retains the Maslov number $m_2$. The result then follows by induction.
\end{proof}

\begin{remark}
   Let $L = L_1 \circ L_2$ be a concatenation of a decomposable cobordism $L_1$ with a saddle cobordism $L_2$ obtained by pinching a Reeb chord $c$ of $\La_2$. A local computation shows that if $c$ has nonzero grading with respect to a $\Z/m\Z$-graded Maslov potential on $L_1$, then $L$ acquires either an additional boundary component or a new handle with Maslov index $|c| \pmod m$. Consequently, Theorem~\ref{thm: MaslovNumber} can alternatively be proven by observing that the Maslov number of a decomposable cobordism $L$ is the greatest common divisor of the Maslov indices of a basis for $H_1(L)$, represented by the intermediate Legendrian links and the genera of $L$.
\end{remark}

We briefly describe a combinatorial recipe to compute the Maslov number of a decomposable cobordism $L$ from $\La_-$ to $\La_+$. Following the proof of Proposition~\ref{proposition: DecomposableMaslovExtension}, to compute the gradings of the Reeb chords resolved in $L$, it suffices to keep track of the changes to a combinatorial Maslov potential on $\La_+$ over $L$. Local computations show that such changes under the Reidemeister moves and the saddle cobordism in the front diagram are as depicted in Figure~\ref{fig: Reidemeister} and Figure~\ref{fig: Saddle}. For a saddle cobordism, note that the potential must be adjusted by the grading of the pinched chord on the top. Using this diagrammatic calculus, the extension of a combinatorial Maslov potential on $\La_+$ over $L$ then follows the same procedure as described before Theorem~\ref{thm: MaslovNumber}. 

\begin{center}
		\begin{figure}[h!]{ \includegraphics[width=\textwidth]{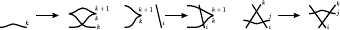}}\caption{Combinatorial rules for computing the change of a Maslov potential under the Reidemeister moves.}\label{fig: Reidemeister}\end{figure}
\end{center}

\begin{center}
		\begin{figure}[h!]{ \includegraphics[width=0.4\textwidth]{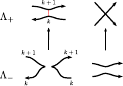}}\caption{Saddle cobordism in front (left) and Lagrangian projection (right). The front diagram is labeled with the change of a Maslov potential.}\label{fig: Saddle}\end{figure}
\end{center}

As a brief application of Theorem~\ref{thm: MaslovNumber}, we conclude this section by showing that fillings of a fixed Legendrian can have different Maslov numbers.

\begin{proposition}\label{}
    There exist Legendrians $\La_k$ that admit an embedded Maslov-0 filling and an embedded Maslov-$k$ filling for all non-negative $k\in 2\Z$.
\end{proposition}

\begin{proof}
    Let $\La_k$, $k\in 2\Z$, $k\geq 2$ denote the family of Legendrians defined in \cite[Section 2.3]{LipshitzNg}. The gradings of the Reeb chords of $\La_k$ are given in Subsection 2.3.1 of loc.~cit.\ and we adopt the notation therein. We construct fillings of $\La_k$ by resolving certain crossings via a composition of a Reidemeister II move and the cobordism depicted in Figure~\ref{fig: D4minus}; see Figure~\ref{fig: preparatory_isotopy} for the preparatory Reidemeister II move. To construct a Maslov-0 filling of $\La_k$, resolve the crossings labeled $a_1$ and $a_2$. The result is isotopic to a max-tb unknot, and, as $|a_1|=|a_2|=0$, we obtain a Maslov-0 filling. To produce a Maslov-$k$ filling, we resolve the crossings $a_4$ and $a_7$ to again arrive at a max-tb unknot. Since $|a_1|=|a_2|=k$, Theorem~\ref{thm: MaslovNumber} implies that the Maslov number of the resulting filling is $k$. Note that $k\in 2\Z$ implies that the crossing resolutions in the constructions above represent orientable 1-handle attachments. 
\end{proof}

\begin{center}
		\begin{figure}[h!]{ \includegraphics[width=0.4\textwidth]{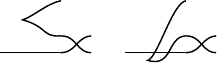}}\caption{Preparatory isotopy we compose with the 1-handle attachment in Figure~\ref{fig: D4minus} to resolve a crossing in the front diagram.}\label{fig: preparatory_isotopy}\end{figure}
\end{center}

\subsection{Legendrian weaves}
\label{ssec:weave}

In this section, we describe Legendrian weaves, a geometric construction of Casals and Zaslow that can be used to combinatorially represent Legendrian surfaces $\w$ in the 1-jet space $J^1\D^2=T^*\D^2\times \R_z$ 
by the singularities of their front projection in $\D^2\times \R_z$. As explained below, the motivation for this approach is that the Lagrangian projection of these surfaces are exact Lagrangian surfaces in $T^*\D^2$, thus providing an alternative way of constructing fillings of certain Legendrian links. In this paper, this material is used primarily in the construction of fillings of the $(2,6)$-cables of our Legendrian representatives of $m(7_2)$; the reader interested in our other results may safely skip this section.

Let $\beta\in Br_N^+$ be a positive braid. We denote by $\La(\beta)$ the $(-1)$-closure of $\beta$; see Figure~\ref{fig: -1closure} for a schematic of the front projection. The contact geometric description of an embedded exact Lagrangian filling of $\La(\beta)$ obtained from a Legendrian weave is as follows. We first construct a Legendrian surface $\w$ in $J^1\D^2=T^*\D^2\times \R_z$ by specifying the (singularities of the) front projection of $\w$. We equip $T^*\D^2$ with the symplectic form $d(e^r\alpha)$ where $\ker(\alpha)=\ker(dy_1-y_2d\theta)$ is the standard contact structure on $J^1(\partial \D^2)$ and $r$ is the radial coordinate. This choice of symplectic form ensures that the flow of $e^r\alpha$ is transverse to $J^1\S^1\cong \R^2\times \partial \D^2$ thought of as the cotangent fibers along the boundary of the 0-section. The Lagrangian projection of $\w$ is then a Lagrangian surface in $(T^*\D^2, d(e^r\alpha))$. Moreover, since $\w\subseteq (J^1\D^2, \ker (dz-e^r\alpha))$ is Legendrian, we immediately obtain the function $z:\Pi_{xy}(\w)\to \R$ satisfying $dz=e^r\alpha|_{\Pi_{xy}(\w)}$, proving that $\Pi_{xy}(\w)$ is exact. 

The boundary of $\Pi_{xy}(\w)$ is taken to be a positive braid $\beta$ in $J^1\S^1$ so that we regard it as a Legendrian link in a contact neighborhood of $\partial \D^2$. As the 0-section of $J^1\S^1$ is Legendrian isotopic to a max-tb standard Legendrian unknot, we can take $\partial\Pi_{xy}(\w)$ to equivalently be the standard satellite of the standard Legendrian unknot. The result of satelliting the positive braid $\beta$ in $J^1\S^1$ about the standard Legendrian unknot in $\R^3$ is the $(-1)$-framed closure of $\beta$ in standard contact $\R^3$; we will use this identification to construct fillings of Legendrian links in $\R^3$.   

The double points of the Lagrangian projection of a weave surface $\w$ correspond to its Reeb chords, which are critical points of the height-difference functions between pairs of sheets of $\w$. Consequently, if $\w$ has no Reeb chords, then $\Pi_{xy}(\w)$ is an embedded exact Lagrangian filling of $\partial(\w)$. We denote such a filling associated to the weave $\w$ by $L_\w$.

\subsubsection{$N$-Graphs and singularities of fronts} To construct a Legendrian weave surface $\w$ in $J^1\D^2,$ we combinatorially encode the singularities of its front projection in a colored graph. Local models for these singularities of fronts are given by Arnol'd \cite[Section 3.2]{ArnoldSing}; the three singularities that appear in our construction are pictured in Figure~\ref{fig: wavefronts} and called the elementary Legendrian cobordisms.

	\begin{center}
		\begin{figure}[h!]{ \includegraphics[width=.8\textwidth]{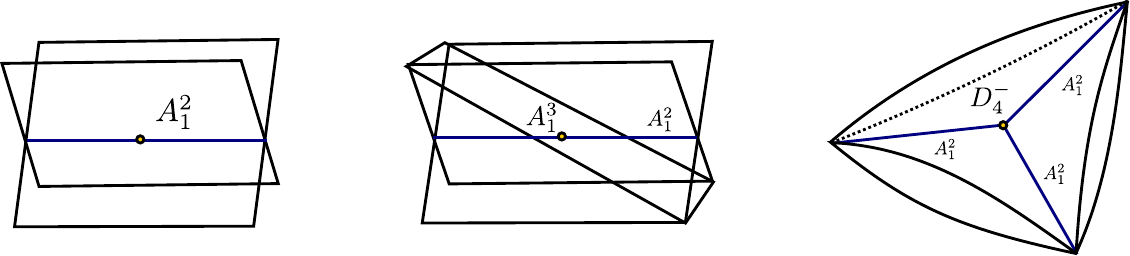}}\caption{Singularities of front projections of Legendrian surfaces. Labels correspond to notation used by Arnol'd in his classification.}
			\label{fig: wavefronts}\end{figure}
	\end{center}

We construct $\w$ so that the boundary of our singular surface $\Pi_{xz}(\w)$ is the front projection of an $N$-stranded positive braid $\beta$. $\Pi(\w)$ can therefore be pictured as a collection of $N$ sheets away from its singularities. We describe the behavior at the singularities as follows:

\begin{enumerate}
    \item The \dfn{$A_1^2$ singularity} occurs when two sheets in the front projection intersect. This singularity can be thought of as the trace of a constant Legendrian isotopy in the neighborhood of a crossing in the front projection. 
    \item The \dfn{$A_1^3$ singularity} occurs when a third sheet passes through an $A_1^2$ singularity. This singularity can be thought of as the trace of a Reidemeister III move in the front projection.
    \item A \dfn{$D_4^-$ singularity} occurs when three $A_1^2$ singularities meet at a single point. This singularity can be thought of as the trace of a 1-handle attachment in the front projection. 
\end{enumerate}

Having identified the singularities of fronts of a Legendrian weave surface, we encode them by a colored graph $\Gamma\subseteq \D^2$. 

\begin{definition}
    An \dfn{$N$-graph} $\Gamma$ is a planar graph with edges labeled by positive Artin generators of the braid group $\Br_N$. We require that any pair of edges labeled $\sigma_i$ and $\sigma_{i+1}$ can only meet at a hexavalent vertex with alternating labels, while any edges labeled $\sigma_i$ can only meet at trivalent vertices. Edges labeled by $\sigma_i$ and $\sigma_j$ for $|i-j|\geq 2$ can only meet at tetravalent vertices with alternating labels.
\end{definition}

To obtain a Legendrian weave $\w(\Gamma)\subseteq (J^1\D^2,\xi_{st})$ from an $N$-graph $\Gamma$, we glue together the local germs of singularities according to the edges of $\Gamma$. First, consider $N$ horizontal sheets $\D^2\times \{1\}\sqcup \D^2\times \{2\}\sqcup \dots \sqcup \D^2\times \{N\}\subseteq \D^2\times \R$ and an $N$-graph $\Gamma\subseteq \D^2\times \{0\}$. We construct the associated Legendrian weave $\w(\Gamma)$ as follows  \cite[Section 2.3]{CasalsZaslow}.

	\begin{itemize}
		\item Above each edge labeled $\sigma_i$, insert an $A_1^2$ crossing between the $\D^2\times\{i\}$ and $\D^2\times \{i+1\}$ sheets so that the projection of the $A_1^2$ singular locus under $\Pi:\D^2\times \R \to \D^2\times \{0\}$ agrees with the edge labeled $\sigma_i$.   
		\item At each trivalent vertex $v$ involving three edges labeled by $\sigma_i$, insert a $D_4^-$ singularity between the sheets $\D^2\times \{i\}$ and $\D^2\times\{i+1\}$ in such a way that the projection of the $D_4^-$ singular locus agrees with $v$ and the projection of the $A_2^1$ crossings agree with the edges incident to $v$.
		\item At each hexavalent vertex $v$ involving edges labeled by $\sigma_i$ and $\sigma_{i+1}$, insert an $A_1^3$ singularity along the three sheets in such a way that the origin of the $A_1^3$ singular locus agrees with $v$ and the $A_1^2$ crossings agree with the edges incident to $v$.
	\end{itemize} 
	\begin{center}
		\begin{figure}[h!]{ \includegraphics[width=\textwidth]{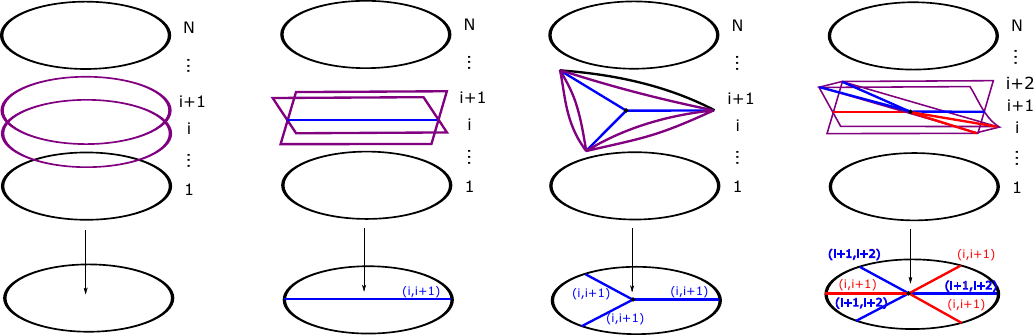}}\caption{The weaving of singularities of fronts along the edges of the $N$-graph (courtesy of Roger Casals and Eric Zaslow, used with permission). Gluing these local models according to the $N$-graph $\Gamma$ yields the weave $\w(\Gamma)$.}	\label{fig:Weaving}\end{figure}
	\end{center}
	
 Note that the tetravalent vertex does not correspond to a distinct singularity of the front projection, as it represents a pair of disjoint $A_1^2$ singularities. If we take an open cover $\{U_i\}_{i=1}^m$ of $\D^2\times \{0\}$ by open disks, refined so that any disk contains at most one of these three features pictured in Figure~\ref{fig:Weaving}, we can glue together the resulting fronts according to the intersection of edges along the boundary of our disks. Specifically, if $U_i\cap U_j$ is nonempty, then we define $\Pi(\w(U_1\cup U_2))$ to be the front resulting from considering the union of fronts $\Pi(\w(U_1))\cup\Pi(\w(U_j))$ in $(U_1\cup U_2)\times \R$.

	\begin{definition}\label{def: weave}
	    The \dfn{Legendrian weave} $\w(\Gamma)\subseteq (J^1\D^2, \xi_{st})$ is the Legendrian lift of the front $\Pi(\w(\cup_{i=1}^m U_i))$ given by gluing the local fronts of singularities together according to the $N$-graph $\Gamma$.
	\end{definition}

Below we will often denote a weave by $\w$ unless we wish to emphasize its relationship with the underlying $N$-graph $\Gamma$, in which case we will retain the notation $\w(\Gamma)$.
\subsubsection{Homology of weaves}	
\label{sssec:weave-homology}

Topologically, a Legendrian weave $\w$ is an $N$-fold branched cover of $D^2$ with simple branch points corresponding to the trivalent vertices of $\Gamma$. The $N$-graph description of $\w$ provides a combinatorial way of identifying particular cycles in $H_1(L_\w;\Z) \simeq H_1(\w;\Z)$ known as $\L$-compressing cycles.

\begin{definition}\label{def: L-compressing cycle}
    An \dfn{$\L$-compressing cycle} $\gamma\in H_1(L)$ of an exact Lagrangian filling $L$ is a homology cycle that bounds an embedded Lagrangian disk in the complement of $L$. An \dfn{$\L$-compressing system} of $L$ is a basis of $\L$-compressing cycles of $H_1(L)$.
\end{definition}

As we will describe in Section~\ref{ssec:cluster}, $\mathbb{L}$-compressing cycles play a crucial role in defining the cluster seed associated to an exact Lagrangian filling of $\La$ in its augmentation variety. For a general filling, determining the existence of an $\L$-compressing cycle -- let alone explicitly identifying one -- remains an open problem. However, when the filling arises from a weave, these cycles can be identified explicitly in the $N$-graph of the weave as follows. First, consider an edge connecting two trivalent vertices in the $N$-graph of a weave. An examination of its local model as in \cite[Figure 11]{CasalsZaslow} shows that each such edge corresponds to a $1$-cycle, and we refer to it as a \dfn{short {\sf I}-cycle}. Similarly, any three edges of the same color that connect a single hexavalent vertex to three trivalent vertices correspond to another 1-cycle, which we refer to as a \dfn{short {\sf Y}-cycle}. See Figure \ref{fig:ShortCycles} for examples of a short {\sf I}-cycle (left) and a short {\sf Y}-cycle (right). We can also consider a sequence of edges starting and ending at trivalent vertices and passing directly through any number of hexavalent vertices, as pictured in Figure \ref{fig:longI-cycle}. Such a cycle is referred to as a \dfn{long {\sf I}-cycle}. Finally, we can combine any number of {\sf I}-cycles and short {\sf Y}-cycles to describe a wide class of 1-cycles as trees with leaves on trivalent vertices and edges passing directly through or branching at hexavalent vertices; we refer to such cycles as \dfn{{\sf Y}-trees}.

 	\begin{center}
		\begin{figure}[h!]{ \includegraphics[width=.45\textwidth]{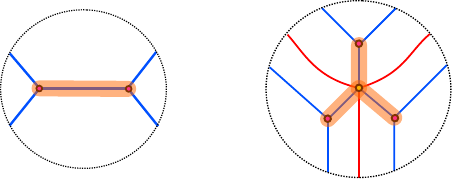}}\caption{Local models of a short {\sf I}-cycle (left) and a short {\sf Y}-cycle (right) depicted by the edges highlighted in orange.
			}
		\label{fig:ShortCycles}\end{figure}
	\end{center}

	\begin{center}
		\begin{figure}[h!]{ \includegraphics[width=.6\textwidth]{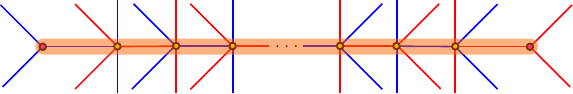}}\caption{A long {\sf I}-cycle depicted by the edges highlighted in orange.}
			\label{fig:longI-cycle}\end{figure}
	\end{center}

For a pair of oriented 1-cycles $\gamma_1, \gamma_2\in H_1(\w(\Gamma))$ intersecting transversely in $\w$, we refer to the signed count of the points of intersection of $\gamma_1$ and $\gamma_2$ as their algebraic intersection and denote it by $\langle \gamma_1, \gamma_2\rangle$. This intersection form governs the cluster structures arising from Lagrangian fillings and provides a powerful invariant for distinguishing them.

In Figure~\ref{fig: CycleIntersections}, we depict cycle intersections near a trivalent (left) and hexavalent (middle and right) vertex. Assuming that the cycles do not intersect at other vertices of the weave, the corresponding intersection quivers are illustrated below. The quiver vertices are color-coded to match the cycles they represent, with an arrow drawn from vertex $1$ to vertex $2$ if $\langle \gamma_1, \gamma_2\rangle = 1$. Note that the {\sf I}-cycle and {\sf Y}-cycle meeting at a hexavalent vertex as in the rightmost local picture do not intersect.  
   
	\begin{center}
		\begin{figure}[h!]{ \includegraphics[width=.5\textwidth]{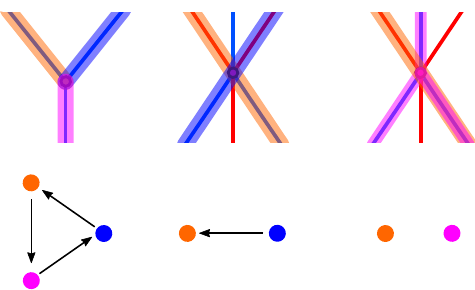}}\caption{Combinatorial rules for locally computing cycle intersections in an inductive weave.}\label{fig: CycleIntersections}\end{figure}
	\end{center}

\section{Legendrian invariants and cluster algebras}\label{sec: algebra_background}

Having established the central geometric objects of this paper, we now turn to the algebraic invariants used to study them, drawing on tools from Legendrian contact homology and cluster theory.

\subsection{The Chekanov-Eliashberg DGA} 
\label{ssec:ce-dga}

 In this section, we review a combinatorial description of the commutative Cheaknov-Eliashberg differential graded algebra (DGA) for Legendrian links in $\R^3$. By the functoriality of Legendrian contact homology, exact Lagrangian cobordisms induce maps of these DGAs, and we describe them explicitly for the elementary cobordisms. These algebraic structures lay the groundwork for Section~\ref{ssec:aug}, where we construct the graded augmentation variety of a Legendrian link, and for Section~\ref{ssec:cluster}, where we introduce the cluster-theoretic framework used to study these varieties. Our exposition serves primarily to set notation and context. For more depth on these materials, we refer the reader to \cite{ENsurvey,ENS,Ng03} for the Chekanov-Eliashberg DGA and its augmentations, and to \cite{EHK, YuPan, CasalsNg} for the DGA maps induced by exact Lagrangian cobordisms.

Let $\La$ be an oriented Legendrian link in $\R^3$, decorated with at least one base point on each component that avoids the Reeb chords. We describe the \emph{commutative} Chekanov-Eliashberg DGA $\SA(\La)$ over a field $\Bbbk$ using the Lagrangian projection $\Pi_{xy}(\La)$. Let $\SR$ denote the set of Reeb chords of $\La$ and $\ST = \{t_1,\ldots,t_l\}$ the set of base points on $\La$. Define the algebra $\SA(\La)$ to be the commutative algebra generated by $\SR \cup \{t_i,t_i^{-1}\ |\ t_i \in \ST\}$ modulo the relations $t_it_i^{-1} = t_i^{-1}t_i = 1$. With the gradings of the Reeb chords $|c|$ defined as in Section~\ref{ssec:maslov} and $|t_i| = 0$ for all $t_i \in \ST$, $\SA(\La)$ has a well-defined $\Z/2r(\La)\Z$ grading up to a global shift on each component of $\La$, where $r(\La)$ is the greatest common divisor of the rotation numbers of the components of $\La$.

The differential $\partial$ of $\SA(\La)$ counts immersed disks in the Lagrangian diagram $\pi_{xy}(\La)$, which constitute rigid pseudo-holomorphic disks in the symplectization $\R \times \R^3$ with boundary on the Lagrangian $\R \times \La$; see \cite{EkholmEtnyreSullivan05b,ENS} for more on the latter perspective. For the combinatorial description, we begin by labeling the four quadrants of every crossing in $\Pi_{xy}(\La)$ with a Reeb sign as in Figure~\ref{fig: DGASigns} (left). Specifically, we label a quadrant with a $+$ if traversing the boundary of the quadrant in the counterclockwise direction one moves from an understand to an overstrand; otherwise, we label it with a $-$. To define the differential over a field of characteristic not equal to $2$, we in addition shade the quadrants of every crossing as in Figure~\ref{fig: DGASigns} (right). 

	\begin{center}
		\begin{figure}[h!]{ \includegraphics[width=.7\textwidth]{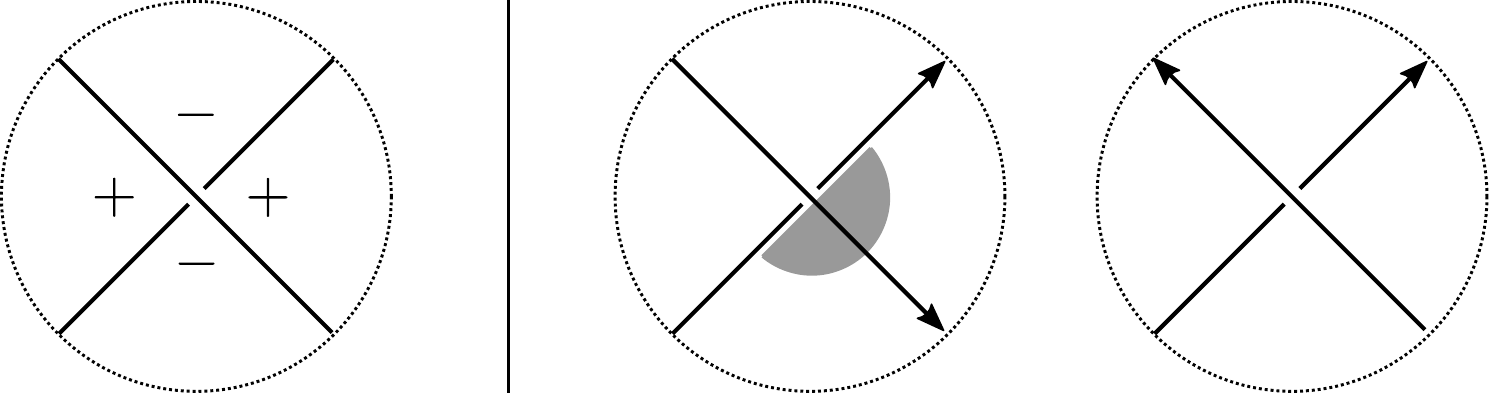}}\caption{Reeb signs (left) and shaded quadrants (right) at a crossing of $\Pi_{xy}(\La)$. 
        }\label{fig: DGASigns}\end{figure}
	\end{center}

To compute the differential of a generator $a \in \SR$, define $\Delta(a,\mathbf{b})$ to be the set of immersions $u: (D_{n+1},\partial D_{n+1}) \to (\R^2,\Pi_{xy}(\La))$ (up to smooth reparametrization) from a disk with $n+1$ number of boundary punctures $p,q_1,\ldots,q_n$ to $\R^2$ that extend continuously to the closed disk and satisfy the following asymptotics conditions:
\begin{itemize}
    \item $u$ sends a neighborhood of $p$ to a positive quadrant of the crossing $a$;
    \item for $i = 1,\ldots,n$, $u$ sends a neighborhood of $p_i$ to a negative quadrant of a crossing $b_i$.
\end{itemize}
Define $w(u)$ to be the word in $\SA$ obtained by traversing the image of $\partial D_{n+1}$ in the counterclockwise direction starting at $a$ and successively reading off the Reeb chords $b_i$ and base points along the way. In particular, we record $t_i$ or $t_i^{-1}$ depending on whether the orientations of $\partial D_{n+1}$ and $\La$ agree or not at the base point. Let $s(u)$ denote the number of shaded quadrants in the image of $u$. We define the differential on generators in $\SR$ by
\begin{equation*}
    \partial a = \sum_{u \in \Delta(a,\mathbf{b})} (-1)^{s(u)} w(u).
\end{equation*}
Finally, define the differential to vanish on the base point generators $t_i^{\pm 1}$ and extend it over all of $\SA$ by linearity and the signed Leibniz rule $\partial (ww') = (\partial w)w'+(-1)^{|w|}w(\partial w')$. It is well-known that $\partial$ has degree $-1$ and squares to zero, and that the DGA is a Legendrian isotopy invariant in the following sense.

\begin{theorem}[\cite{Chekanov}]
    $(\SA(\La),\partial)$ is a well-defined differential graded algebra and the stable tame isomorphism type of  $(\SA(\La),\partial)$ is an invariant of $\La$ under Legendrian isotopy and choice of base points.
\end{theorem}

Briefly, a stable tame isomorphism is a composition of elementary automorphisms after stabilizing the algebra with finitely many acyclic pairs of generators $e_1,e_2$ with the relation $\partial e_1 = e_2$. In particular, a stable tame isomorphism $\Phi:\SA(\La) \to \SA(\La')$ induces an isomorphism $H_*(\SA(\La)) \simeq H_*(\SA(\La'))$. We refer the reader to \cite[Theorem 3.3]{Chekanov} or \cite[Section 3.2]{ENsurvey} for a precise definition.

By \cite{EHK} and \cite{YuPan}, an exact Lagrangian cobordism $L$ from $\La_-$ to $\La_+$ induces a DGA map $\Phi_L: (\SA(\La_+)[H_1(L)],\partial) \to (\SA(\La_-)[H_1(L)],\partial)$, where $\partial$ vanishes on invertible elements in $H_1(L)$. Specifically, $\phi_L(a)$ counts rigid pseudo-holomorphic disks with boundary on $L$ that have a single boundary puncture asymptotic to the Reeb chord $a$ of $\La_+$, and several boundary punctures asymptotic to Reeb chords of $\La_-$. By the dimension formula of the moduli spaces of these disks \cite[Lemma 3.8]{EHK}, $\phi_L$ has degree $0$ in the grading $\Z/m(L)\Z$. Moreover, the assignment $L \mapsto \Phi_L$ is functorial with respect to the composition of cobordisms: for a composition of cobordisms $L = L_1 \circ L_2$, we have $\Phi_{L} =  \Phi_{L_2} \circ \Phi_{L_1}$ via the inclusion $H_1(L_1) \hookrightarrow H_1(L)$.  For our purpose of computing the DGA maps induced from decomposable cobordisms, we briefly review the DGA maps associated to the elementary cobordisms following \cite{EHK,YuPan,CasalsNg}.

\begin{itemize}
    \item \textbf{Isotopy cobordisms.} As discussed above, an isotopy cobordism $L$ from $\La_-$ to $\La_+$ induces a stable tame isomorphism $\Phi_L:\SA(\La_+) \to \SA(\La_-)$; see \cite[Section 4.1]{CasalsNg} for more details. 
    \item \textbf{Saddle cobordisms.} Let $L_a$ be a saddle cobordism from $\La_-$ to $\La_+$ obtained by resolving the Reeb chord $a$. Moreover, assume that $a$ is \dfn{simple}, meaning for any immersed disk $u$ in $\La_+$ with $k$ positive corners at $a$, a positive convex corner at some Reeb chord besides $a$, and all other corners (including concave corners) being negative, then $u$ must have at least $k$ concave corners. As explained in Section~\ref{ssec: filling}, all of the Reeb chords that we pinch in this paper are simple. 
    
    For the definition of the DGA map $\Phi_{L_a}$, we decorate $\La_-$ with two additional base points $s$ and $-s^{-1}$ of grading $0 \in \Z/m(L)\Z$ on the top and bottom stand of the place of resolving. Geometrically, the concatenation of the curves connecting $s$ and $-s^{-1}$ to the saddle point is a relative cycle in $H_1(L_a,\La_- \cup \La_+)$. As a result, we can use $s$ to represent a cycle in $H_1(L_1)$, viewed as the Poincaré dual of the $1$-cycle; see \cite[Section 3.1]{YuPan} and \cite[Section 3.5]{CasalsNg} for more details. Let $\Delta_a(b)$ denote the set of immersed disks in $\La_+$ with convex corners, comprising exactly two positive corners at $a$ and $b$, and possibly more negative corners. For $u \in \Delta_a(b)$, let $w_1(u)$ denote the word of negative corners and base points encountered as one traverses $\partial u$ from $b$ to $a$, and $w_2(u)$ the analogous word as one traverses $\partial u$ from $a$ to $b$. We define the map $\Phi_{L_a}:\SA(\La_+)[s^{\pm}] \to \SA(\La_-)[s^{\pm}]$ on the generators of $\SA(\La_+)[s^{\pm}]$ by
    \begin{align*}
        \Phi_{L_a}(a) &= s,\quad \Phi_{L_a}(t^{\pm 1}) = t^{\pm 1}, \\
        \Phi_{L_a}(b) &= b + \sum_{u \in \Delta_a(b)} (-1)^{s(u)}w_1(u)s^{-1}w_2(u),\quad \text{for any Reeb chord}\ b \not = a.
    \end{align*}
    \item \textbf{Minimum cobordisms.} Finally, let $L$ be the minimum cobordism from $\La_-$ to $\La_+ = \La_- \cup U$, where $U$ is the max-tb Legendrian unknot decorated with base points $s_1,\ldots,s_k$. Then, $\Phi_L$ is the unique DGA map that sends the Reeb chords on $U$ to $0$, the product of base points $s_1\cdots s_k$ to $-1$, and acts as the identity on all other generators. Note that in the case when $\La_- = \emptyset$, we obtain a DGA map $\Phi_L:(\SA(U),\partial) \to (\Bbbk[s_1^{\pm},\ldots,s_k^{\pm}]/(s_1\cdots s_k = -1),0)$.
\end{itemize}

Hence, for a decomposable filling of an link $\La$ equipped with exactly one base point $t_i \in \ST$ on each of its $l$ components, the cobordism functor induces an explicit DGA map
\begin{equation*}
    \epsilon_L:(\SA(\La),\partial) \to (\Bbbk[t_1^{\pm 1},\ldots,t_l^{\pm 1},s_1^{\pm 1},\ldots,s_r^{\pm 1}]/(w_1 = \cdots = w_k = -1),0),
\end{equation*}
through the composition of the elementary maps, where $r$ is the number of saddle cobordisms applied and $w_i$ is a word of base points lying on the same component of an unknot capped off by a minimum. Note that by \cite[Proposition 3.13]{CasalsNg}, the ring $R$ on the right hand side is isomorphic to $\Bbbk[H_1(L\backslash \ST, \La\backslash \ST;\Z)] \simeq \Bbbk[H_1(L)\oplus \Z^{l-1}] \simeq (\Bbbk^*)^{2g+2l-2}$, where $g$ is the genus of $L$ and $l$ is the number of components of $\La$. This confirms that the map is consistent with the definition of the cobordism functor. With a choice of a $\Bbbk$-algebra homomorphism $\eta: R \to \Bbbk$, we obtain a DGA map $\eta \circ \epsilon_L: (\SA(\La),\partial) \to (\Bbbk,0)$ known as an augmentation of $\SA(\La)$, which we discuss in more detail in the next section.

\subsection{Augmentations} 
\label{ssec:aug}

The Chekanov-Eliashberg DGA and its homology are often infinite-dimensional and thus difficult to use directly for practical computations. A fundamental tool used to circumvent this issue is an augmentation. Let $m$ be a divisor of $2\rot(\La)$. A $(\Bbbk$-valued) \dfn{$m$-graded augmentation} of a decorated Legendrian link $\La$ is a DGA map
\begin{equation*}
    \epsilon: (\SA(\La),\partial) \to (\Bbbk,0)
\end{equation*}
where we think of $(\SA(\La),\partial)$ as inheriting a $\Z/m\Z$ grading from the $\Z/2(r)\Z$ grading described in the previous subsection. In particular, the algebraically defining equation for $\epsilon$ is $\epsilon \circ \partial = 0$. By \cite[Theorem 1.1]{Leverson}, it must also send the product of base points $t_1,\ldots,t_k$ on the same component of $\La$ to $(-1)^k$. 

By \cite{Fuchs03,FuchsIshkhanov,Sabloff05}, note that augmentations do not always exist for an arbitrary Legendrian link. However, for an exact Lagrangian fillable link $\La$, we see from above that the DGA map $\epsilon_L$ induced from a filling $L$ gives rise to a family of augmentations of $\La$ parametrized by the local systems $\eta: \Bbbk[H_1(L\backslash \ST, \La\backslash \ST;\Z)] \simeq R \to \Bbbk$. On the other hand, we note that not every augmentation of $\La$ is induced from an embedded filling and that there exist non-fillable Legendrian links that admit augmentations; see e.g. \cite{GaoRutherford, PanRutherford}.

Augmentations have historically been used to define invariants of Legendrian links and their exact Lagrangian fillings. To introduce the ones relevant to our work, we begin by defining a notion of equivalence of augmentations.
\begin{definition}
\label{dfn:dg-htpy}
    Two $m$-graded augmentations $\epsilon_1,\epsilon_2:(\SA(\La),\partial) \to (\Bbbk,0)$ are \dfn{dg-algebra homotopic} if there exists a $\Bbbk$-linear map $h:(\SA(\La),\partial) \to \Bbbk$ of degree $+1$ such that
    \begin{enumerate}
        \item $h(xy) = h(x)\epsilon_2(y) + (-1)^{|x|}\epsilon_1(x)h(y)$ for all $x,y \in \SA(\La)$; 
        \item $\epsilon_1(a)-\epsilon_2(a) = h(\partial a)$ for any Reeb chord $a$ of $\La$.
    \end{enumerate}
\end{definition}

The main algebraic invariant studied in this paper is the augmentation variety of $\La$, which we define as follows. First, the \dfn{naïve $m$-graded augmentation variety} $\widetilde{\Aug}_m(\La)$ of $\La$ over $\Bbbk$ is the set of $\Bbbk$-valued $m$-graded augmentations of $\La$ with a fixed decoration of base points. Assuming $\Bbbk$ is algebraically closed, $\widetilde{\Aug}_m(\La)$ is an affine subvariety of $\Bbbk^n$ cut out by the system of algebraic equations $\epsilon \circ \partial = 0$, where $n$ is the number degree $0$ generators in $\SA(\La)$. By \cite[Corollary 3.15]{Ng03}, $\widetilde{\Aug}_m(\La)$ is a Legendrian invariant up to multiplying by affine space factors. The augmentation variety is an enhancement of the naïve augmentation variety.

\begin{definition}
    \label{dfn:aug-vareity}
    The \dfn{$m$-graded augmentation variety} $\Aug_m(\La)$ of a Legendrian link $\La$ over $\Bbbk$ is the quotient space $\widetilde{\Aug}_m(\La)/\sim$, where two augmentations are considered to be equivalent if they are dg-algebra homotopic.
\end{definition}

As observed in e.g.\ \cite[Remark 2.9 and Proposition 2.10]{CSHW2023}, the augmentation variety as a set is an invariant up to Legendrian isotopy of $\La$ and change of base points; see also \cite[Corollary 3.15]{Ng03}. In general, it is not certain whether $\Aug_m(\La)$ is still an affine variety. However, it has been established for some specific cases that they are, including positive braid Legendrian links \cite[Corollary 2.14]{GSW}, and certain max-tb representatives of Legendrian $2$-bridge links \cite[Proposition 3.19]{CSHW2023} and Legendrian twist knots \cite{GaoRutherford}, among others. Similarly, we will prove in Proposition \ref{proposition: twist-cluster} that the $m$-graded augmentation variety of certain max-tb representatives of double twist knots $\La$ are affine, where $m$ is defined as in the statement of Theorem~\ref{thm: main}.

We end this section by relating the augmentation variety of $\La$ to its fillings. Recall that a filling $L$ of link $\La$ decorated with base points in $\ST$ induces a $(\Bbbk^\ast)^{2g+2l-2}$-family of augmentations. Moreover, we know by \cite[Proposition 2.6]{GaoRutherford} that these augmentations are pairwise DGA non-homotopic so that a filling induces an injective map of the algebraic torus $\Bbbk^{2g+2l-2}$ into the augmentation variety $\Aug_m(\La)$. It is a standard fact from \cite[Lemma 3.13]{EHK} that if $L,L'$ are two Hamiltonian isotopic exact Lagrangian filling of a decorated Legendrian link $\La$, then there is an invertible map $\phi:\Bbbk[H_1(L'\backslash \ST, \La\backslash \ST;\Z)] \to \Bbbk[H_1(L\backslash \ST, \La\backslash \ST;\Z)]$ such that $\eta \circ \epsilon_L$ and $\eta \circ \phi \circ \epsilon_{L'}$ are dg-algebra homotopic for any $\eta:\Bbbk[H_1(L\backslash \ST, \La\backslash \ST;\Z)] \to \Bbbk$. Therefore, we can use the algebaric embeddings $(\Bbbk^\ast)^{2g+2l-2} \hookrightarrow \Aug_m(\La)$ to distinguish the fillings of $\La$.

\subsection{Rulings of Legendrian links}
In order to distinguish Legendrian representatives of $\La[w,c]$ defined in Section~\ref{ssec: max-tb_reps}, we require the notion of a normal ruling. Normal rulings of Legendrian links are combinatorial invariants associated to their front projections. They were first introduced by Chekanov and Pushkar in  \cite{ChekanovPushkar05}, and separately by Fuchs in \cite{Fuchs03}. We review the basic definitions and results here. 

Given a front projection $\Pi_{xz}(\La)$ of a Legendrian link $\La$, a normal ruling is a decomposition of the front into pairs of paths, where two paired paths begin at the same left cusp and end at the same right cusp. Rulings paths are required to satisfy the following conditions: 
\begin{itemize}
    \item all paths are smooth except at cusps and possibly crossings;
    \item two paired paths do not intersect in the front diagram except at their shared endpoints;
    \item any two paths intersect in at most a set of crossings and cusps
    \item at crossings where two paths intersect and one path lies entirely above the other, the two paths and their pairs must be arranged as in Figure~\ref{fig: RulingDef}. Such a crossing is referred to as a \dfn{switch}.
\end{itemize}

\begin{center}
		\begin{figure}[h!]{ \includegraphics[width=.7\textwidth]{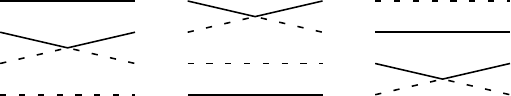}}\caption{Allowed behavior at switched crossings in a normal ruling. Solid lines represent one pair of paths, while dotted lines represent a second pair of paths. 
        }\label{fig: RulingDef}\end{figure}
	\end{center}

An appropriate count of normal rulings yields an invariant of Legendrian links. This count can be refined in the following way:

\begin{definition}
    An $m$-\dfn{graded normal ruling} is a normal ruling such that every switch occurs at a crossing $c$ with grading $|c|$ dividing $m$. 
\end{definition}

Note that in order for there to exist an $m$-graded ruling of a Legendrian knot $\La$, $m$ must divide $2\rot(\La)$. 
Counting rulings provides a Legendrian isotopy invariant:

\begin{theorem}[Theorem 2.5, \cite{ChekanovPushkar05}]\label{thm: ruling_invariant}
    Given a Legendrian knot $\La$ with rotation number $\rot(\La)$, the number of $m$-graded rulings is a Legendrian isotopy invariant of $\La$ for any $m$ dividing $\rot(\La)$.
\end{theorem}

Rulings also encode information about the existence of decomposable Lagrangian fillings. In particular, a decomposable Lagrangian filling $L$ of $\La$ induces a \dfn{canonical ruling} $\rho_L$ on $\La$ by \cite[Lemma 2]{atiponrat2015}. The Euler characteristic of a ruling is defined in \cite{ChekanovPushkar05} to be the quantity $\chi(\rho)=c(\La)-s(\rho)$ where $c(\La)$ is the number of left (or right) cusps of $\La$ and $s(\rho)$ denotes the number of switches of $\rho$. The Euler characteristic of a canonical ruling is determined by $\La$.

\begin{proposition}[Proposition 3.11, \cite{CCPRSY}]\label{proposition: canonical_ruling_obstruction}
    If $L$ is a decomposable Lagrangian filling of $\La$ with canonical ruling $\rho_L$, then $\chi(\rho_L)=\chi(L).$
\end{proposition}

Since \cite{Chantraine10} implies that for any filling $L$, we have the equality $\chi(L)=-\tb(\La)$, we can use Proposition~\ref{proposition: canonical_ruling_obstruction} to obstruct the existence of a decomposable filling if we know that no rulings with the correct number of switches exist.

\subsection{Cluster Varieties}
\label{ssec:cluster}

We review the basic terminology in cluster theory as it relates to the augmentation varieties of Legendrian links. For our purpose, it suffices to describe the global geometric structure and the combinatorial data that governs it. We refer the reader to \cite{FominZelevinsky_ClusterI} and \cite{FominZelevinsky_ClusterII} for a comprehensive treatment of these materials. First, the basic building block of a cluster variety is a cluster seed.

\begin{definition}
    A \dfn{cluster seed} $\s = (T_\s,\BA_\s,Q_\s)$ consists of a split algebraic torus $T_s \simeq (\Bbbk^\ast)^n$, a basis of coordinates $\BA_\s = (A_1,\ldots,A_n)$ called \dfn{cluster variables}, and a quiver $Q_\s$ with no $1$-cycles or oriented $2$-cycles whose vertices are in bijection with the cluster variables. The cluster variables and their corresponding vertices in $Q_\s$ are partitioned into two sets: \dfn{mutable} and \dfn{frozen}.
\end{definition}

To construct a global space, cluster seeds are glued together via an involutive operation called \dfn{cluster mutation}. Let $\s$ be a cluster seed and $k \in Q_\s$ a mutable vertex. The cluster mutation of $\s$ in the direction of $k$ produces a new cluster seed $\s' = (T_{\s'},\{A_1',\ldots,A_n'\},Q_{\s'})$ along with a birational map $\mu_k:T_\s \to T_{\s'}$. The cluster seed $\s'$ is determined as follows:
\begin{itemize}
    \item \textbf{Quiver mutation}: The quiver $Q_{\s'}$ is obtained from $Q_{\s}$ by adding a new arrow $i \to j$ for every pair of arrows $i \to k$ and $k \to j$; reversing all arrows incident to $k$; and deleting arrows in opposite directions in pairs until no $2$-cycles remain.
    \item \textbf{Coordinate transformation}: The pullback map $\mu_k^*$ is given by
    \begin{equation*}
        \mu_k^*(A_i') = \begin{cases}
            A_k^{-1}\left(\prod_j A_j^{[\epsilon_{kj}]_+} + \prod_j^{[-\epsilon_{kj}]_+} A_j\right) & \text{if}\ i = k, \\
            A_i & \text{if}\ i \not = k,
        \end{cases}
    \end{equation*}
    where $\epsilon_{kj} = \# (i \to j) - \# (j \to i)$ and $[x]_+:= \max\{x,0\}$.
\end{itemize}

Two cluster seeds are said to be \dfn{mutation equivalent} if they are related by a sequence of cluster mutations. A priori, arbitrary compositions of these rational maps could yield increasingly complicated rational functions. However, cluster mutations exhibit an algebraic rigidity that ensures the charts glue together well, as established by the following theorem.

\begin{theorem}[{Laurent phenomenon \cite[Theorem 3.1]{FominZelevinsky_ClusterI}}] 
    If $\s$ and $\s'$ are mutation equivalent, the composition of their birational transition maps expresses any cluster variables in $\s'$ as a Laurent polynomial in the cluster variables of $\s$.
\end{theorem}

\begin{definition}
    A cluster variety $\FA$ is an affine variety with an atlas of mutation equivalent cluster seeds $\mathbf{S} = \{\s = (T_\s, \BA_\s, Q_\s)\}$ satisfying
    \begin{itemize}
        \item Each cluster seed $\s$ yields an open embedding $T_\s \hookrightarrow \FA$ such that the union of these \dfn{cluster charts}, $\bigcup_{\s \in \mathbf{S}} T_\s$, covers $\FA$ up to codimension $2$;
        \item The atlas $\mathbf{S}$ is closed under cluster mutation at any mutable vertex;
        \item The transition maps between any charts $T_\s$ and $T_{\s'}$ in $\FA$ are birationally equivalent to the composition of their cluster mutation maps.
    \end{itemize}
    The \dfn{cluster algebra} is the subalgebra of the coordinate ring $\SO(\FA)$ generated by all cluster variables and the inverses of all frozen variables .
\end{definition}

By the Laurent phenomenon, every cluster variable in every cluster seed extends to a globally defined regular function on the entire variety $\FA$. Since frozen variables are invariant under mutations, they define nowhere-vanishing $\Bbbk^\ast$-valued functions on $\FA$, making them units in $\SO(\FA)$. On the other hand, if $A_i$ is a mutable cluster variable on $\FA$, we define the \dfn{cluster localization} of $\FA$ at $A_i$ to be the non-vanishing locus $\{A_i \not = 0\}$ in $\FA$, with coordinate ring equal to the localization of $\SO(\FA)$ at $A_i$. Finally, since cluster mutation acts transitively, the cluster structure of $\FA$ is uniquely determined by specifying a single cluster seed. We call such a seed an \dfn{initial cluster seed} and its cluster variable the \dfn{initial cluster variable}.

\subsection{Braid Varieties}\label{ssec: braid_vties}
Recent developments in contact geometry \cite{GSW,CasalsWeng,CSHW2023} have demonstrated that for a broad class of exact Lagrangian fillable Legendrian links, including positive braid links and links realized by certain $(-1)$-closure of positive braids, their augmentation variety admits a cluster structure whose cluster seeds are induced by exact Lagrangian fillings. Specifically, if we denote by $g$ the genus of a filling $L$, and by $l$ the number of boundary components, then the cluster charts are the algebraic embeddings $(\Bbbk^\ast)^{2g+2l-2} \hookrightarrow \Aug_m(\La)$  induced by the DGA maps associated to the fillings and the quivers are the intersection quivers of the $\L$-compressing systems of the fillings.\footnote{\cite{CasalsWeng} uses the microlocal theory of sheaves to produce embeddings of algebraic tori into a sheaf-theoretic analog of the augmentation variety.}  A cluster mutation of a cluster seed in $\Aug_m(\La)$ then corresponds geometrically to a Lagrangian disk surgery of the filling along an $\L$-compressing disk that, in general, produces a new filling. We refer the reader to \cite{Polterovich_Surgery} for the precise detail of this surgery operation. In particular, \cite[Proposition 2.5.1]{Yau} shows that this operation preserves the Maslov number of a Lagrangian filling in the 4-ball. Often, the fillings corresponding to the cluster seeds are constructed using the Legendrian weaves introduced in Section~\ref{ssec:weave}. In this setting, the Lagrangian disk surgery can be executed combinatorially as a Legendrian mutation of a weave along an $I$-cycle using $N$-graph calculus \cite[Theorem 4.21]{CasalsZaslow}.

The geometric framework described above finds its most general realization to date in the study of braid varieties. Given a positive braid $\beta$, the \dfn{braid variety} $X(\beta)$ is defined as a space of flags with certain transversality conditions determined by $\beta$; see \cite[Section 3.3]{CGGLSS} for a precise definition. If we denote by $\Delta$ the half twist element of the braid group, then the braid variety of $\beta$ corresponds to the augmentation variety of the $(-1)$-closure $\La(\beta\Delta)$ in the following sense. 

\begin{proposition}[Theorem 2.39, \cite{ACSHLLW}]
    For $\beta \in \Br_n^+$, there is an isomorphism of algebraic varieties $X(\beta)\cong \Aug(\La(\beta\Delta))$.
\end{proposition}

See also \cite[Theorem 2.39]{CGGS1} for the original description of the isomorphism in the case where $\beta\Delta=\beta'\Delta^2$ for $\beta'$ a positive braid. Note that $X(\beta)\cong \Aug(\La(\beta\Delta))=\emptyset$ if and only if the \emph{Demazure product} $\delta(\beta)$ of the positive braid $\beta\in \Br_n^+$ is less than the longest word $w_0\in S_n$ in Bruhat order;  see e.g. \cite[Section 4.4]{CGGS1} for a definition of the Demazure product. 

For braids satisfying the condition that $\delta(\beta)=w_0$, \cite[Section 4.3]{CGGLSS} gives a recipe for constructing a filling for $\La(\beta\Delta)$ as a Legendrian weave $\w_\beta$. The authors also give a combinatorial rule for identifying a set of homology cycles, referred to as \emph{Lusztig cycles}, 
in $H_1(\w_\beta, \La(\beta\Delta))$ whose algebraic intersections form a quiver. The Lusztig cycles that form absolute cycles in $H_1(\w_\beta, \La(\beta\Delta))$ correspond to mutable vertices, while relative cycles correspond to frozen vertices; see Section~\ref{sssec:weave-homology} for details on how to compute this quiver from a given set of cycles. The set of Lusztig cycles is in bijection with the set of trivalent vertices $\mathfrak{W}_3$ of $\w_\beta$, and following the notation of \cite{CGGLSS} we index a Lusztig cycle as $\gamma_v$ for a vertex $v\in \mathfrak{W}_3$. If $\beta=\Delta\beta'$ for some $\beta'\in \Br_n^+$, the mutable Lusztig cycles can be shown to be $\L$-compressing. It is expected that this property holds in full generality. 

From the data of the weave $\w_\beta$ and its Lusztig cycles, the authors extract regular functions $\{A_{\gamma_v}\}_{v\in \mathfrak{W}_3}$. These functions produce a cluster structure on braid varieties:

\begin{theorem}[Theorem 1.1, \cite{CGGLSS}]
  For $\beta\in \Br_n^+$ with $\delta(\beta)=w_0$, the coordinate ring $\C[X(\beta)]$ is a cluster algebra with initial seed given by the data of the intersection quiver of $\w_\beta$ together with the regular functions $\{A_{\gamma_v}\}_{v\in \mathfrak{W}_3}$. 
\end{theorem}

Following \cite{CasalsWeng}, these cluster structures are defined up to quasi-cluster equivalence. In particular, when $\w_\beta$ does not admit a basis of $\L$-compressing cycles, arbitrarily completing the collection of $\L$-compressing cycles to a basis represents a choice of quasi-cluster structure, and different choices made in completing this collection to a basis lead to quasi-equivalent quasi-cluster structures; see~\cite[Section 3.7]{CasalsWeng} for further details. In all of the cases considered in this paper, our fillings do not admit an absolute homology basis of $\L$-compressing cycles and the choice of basis follows the conventions of \cite{CGGLSS}. We generally suppress discussion of quasi-cluster equivalence otherwise for ease of exposition.

The functions $\{A_{\gamma_v}\}_{v\in \mathfrak{W}_3}$ are obtained as products of auxiliary functions $\{s_v\}_{v\in \mathfrak{W}_3}$ computed using what the authors call ``framed weaves'' \cite[Section 5.1]{CGGLSS}. These auxiliary functions can be identified with the homology variables of the pinching map defined in Section~\ref{ssec:ce-dga}, and are generally not regular functions on $X(\beta)$. The product defining  $A_{\gamma_v}$ is determined by the intersection of Lusztig cycles in a way that combinatorially represents a change of basis between the homology cycles recorded by the pinch map and the Lusztig cycles. We give a precise illustration of this process in Section~\ref{ssec: cluster_coordinates}.

\section{Legendrian double twist knots}\label{sec: DoubleTwistKnots}

In this section, we distinguish different max-tb Legendrian representatives $\La[w, c]$ of the double twist knots $K_{b, c}$ following the methods of \cite{EtnyreNgVertesi13}. We begin by characterizing when $\La[w, c]$ and $\La[w', c]$ are isotopic for two words $w, w'$ of even length. We then use \cite{OzsvathStipsicz} and ruling invariants of $\La[w, c]$ to distinguish the remaining Legendrian representatives. We conclude the section by giving an inductive procedure for showing that $\La[w, c]$ is Legendrian isotopic to the $(-1)$-closure of a positive braid when $w$ has equal numbers of $S$ and $Z$ tangles.

\subsection{Max-tb representatives of Legendrian double twist knots}\label{ssec: max-tb_reps}

Recall from Section~\ref{sec: intro} that we denote by $K_{b, c}$ the double twist knot pictured in Figure~\ref{fig: DoubleTwist} (left). In \cite[Theorem 1.1]{EtnyreNgVertesi13}, the authors classify all Legendrian representatives of twist knots $K_{b, 2}$.\footnote{The parameter $b$ here is equal to $m-1$ appearing in \cite{EtnyreNgVertesi13} in the case where $m$ is negative.} As part of this classification, they show that when $b\geq -3$, $K_{b, 2}$ is Legendrian simple. Moreover, $K_{b, c}$ has maximal Thurston-Bennequin number at most $-3$ when $b\leq -4$ is even. As discussed in Section~\ref{ssec:contact}, \cite[Theorem 1.3]{Chantraine10} implies that such twist knots do not admit any fillings. Therefore, we restrict to $b\leq -5$ odd. 

We produce a max-tb representative of $K_{b, c}$ by filling in the box in the front diagram pictured in Figure~\ref{fig: DoubleTwist} (right) by $-b-1$ tangles of type $Z$ (top right) or $S$ (bottom right). We record the collection of tangles in the word $w$ and denote the resulting Legendrian by $\La[w, c]$ where $c$ is the number of positive crossings in the upper right region of the front diagram. If we orient our knot, then we can further separate the $Z$ and $S$ tangles into oriented Legendrian tangles $Z^+, Z^-, S^+,$ and $S^-$, as pictured in Figure~\ref{fig: Z+Z-S+S-}. Following  \cite[Section 3]{EtnyreNgVertesi13}, we denote the number of $Z^+, Z^-, S^+$, and $S^-$ tangles in a word $w$ by $z^+, z^-, s^+$, and $s^-$, respectively. Note that the letters of $w$ must alternate in sign. Consequently, for $b = -2d-1$, we must have $z^+ + s^+=d$ and $z^- + s^-=d$.

\begin{center}
    \begin{figure}[h!]{ \includegraphics[width=.5\textwidth]{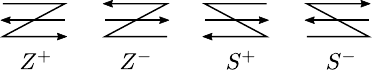}}\caption{From left to right: Legendrian tangles $Z^+$, $Z^-$, $S^+$, and $S^-$. }
    \label{fig: Z+Z-S+S-}\end{figure}
\end{center}

 We start by giving conditions on words $w$ and $w'$ that yield Legendrian isotopic knots following \cite[Lemma 3.4]{EtnyreNgVertesi13}. 

\begin{center}
    \begin{figure}[h!]{ \includegraphics[width=\textwidth]{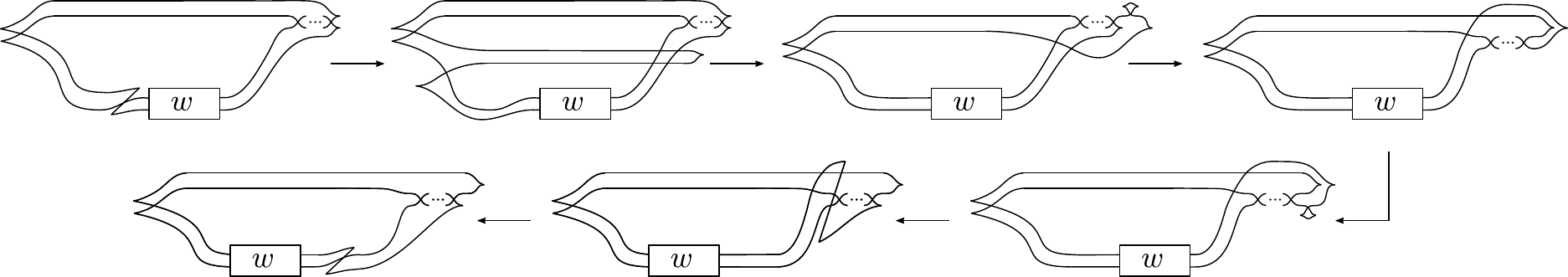}}\caption{Legendrian isotopy of double twist knots $\La[Zw, c]$ and $\La[wZ, c]$.}
    \label{fig: Zw_wZ}\end{figure}
\end{center} 

\begin{lemma}\label{lemma: ZSS=SSZ}
    Let $w$ and $w'$ be two words of length $2d$ with the same $z^+, z^-, s^+, s^-$. Then $\La[w, c]\cong \La[w', c]$. 
\end{lemma}

\begin{proof}
We first show that for $w$ a word of length $2d-1$, $\La[Z^{\pm}w, c]\cong \La[wZ^{\pm}, c]$ and $\La[S^{\pm}w, c]\cong \La[wS^\pm, c]$; see Figure~\ref{fig: Zw_wZ} for the required Legendrian isotopy in the $Z$ case. Reflecting each figure about the vertical axis yields the isotopy for the $S$ case.

Following the proof of \cite[Lemma 3.4]{EtnyreNgVertesi13} or Figure 12 in loc. cit., the tangles $ZZS$ and $SZZ$ are Legendrian isotopic, as are $ZSS$ and $SSZ$. Since the letters in any word $w$ must alternate in sign, this local computation allows us to swap any two letters with the same sign separated by a single letter. Together with the isotopies from Figure~\ref{fig: Zw_wZ}, this implies that any two words with the same number of $Z^+, Z^-, S^+,$ and $S^-$ yield Legendrian isotopic knots.
\end{proof}

Given a word $w'$ in the letters $S$ and $Z$, we denote by ${z^\pm}'$ the number of $Z^\pm$ in $w'$. As a final equivalence relation, we have that simultaneously swapping all $Z^\pm$ with $S^\pm$ yields Legendrian isotopic knots, as is described in the following lemma.

\begin{lemma}[Lemma 3.5, \cite{EtnyreNgVertesi13}]\label{lemma: reflection}
   Let $w$ and $w'$ be words of length $2d$. If $(z^+, z^-) =(d-{z^+}', d-{z^-}')$, then $\La[w, c]\cong \La[w', c]$.
\end{lemma}

\begin{proof}
   The reflection $x\mapsto -x$ in the $xz$-plane lifts to a contactomorphism $(x, y, z)\mapsto (-x, -y, z)$ of $(\R^3, \xi_{st})$. This contactomorphism induces a Legendrian isotopy, as can be seen by its action of rotation by $180^\circ$ in the Lagrangian projection. If we perform the isotopy pictured in Figure~\ref{fig: reflection_front} prior to this reflection, then we see the reflection simply swaps $Z^\pm$ with $S^\pm$, producing the desired isotopy between $\La[w, c]$ and $\La[w', c]$. 
\end{proof}

\begin{center}
    \begin{figure}[h!]{ \includegraphics[width=\textwidth]{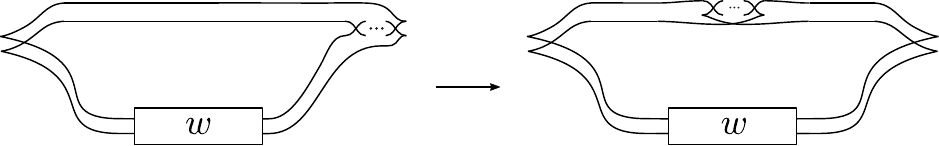}}\caption{Legendrian isotopy transforming $\La[w, c]$ into a front that is fixed by reflection across the $z$-axis except for possibly in the $w$ region.}
    \label{fig: reflection_front}\end{figure}
\end{center} 

We are now prepared to describe the Legendrian isotopy classes of $\La[w, c]$.

\begin{theorem}\label{thm: z+z-words}
Let $w$ and $w'$ be two words of length $2d$. The Legendrian double twist knots $\La[w, c]$ and $\La[w', c]$ are Legendrian isotopic if and only if $(z^+, z^-)=({z^+}', {z^-}')$ or $(z^+, z^-)=(d-{z^+}', d-{z^-}')$
\end{theorem}

\begin{proof}

To distinguish $\La[w, c]$ from $\La[w', c]$ we have two complementary tools: the count of $m$-graded rulings, and the $\mathcal{L}$-invariant of the transverse push-off computed in \cite{OzsvathStipsicz}. We use ruling invariants to distinguish Legendrian representatives with different $z^+$ and the $\mathcal{L}$-invariant to distinguish representatives with the same $z^+$ but different $z^-$.

By inspection, we see that the $c$ positive crossings in the front diagram of $\La[w, c]$ have grading $2(d-(z^++z^-))$ and all other crossings are 0-graded. Following the construction of rulings of $\La[w, 2]$ from \cite[Proposition 3.6]{EtnyreNgVertesi13}, we see that switching at all of the 0-graded crossings yields a 0-graded ruling of $\La[w, c]$. If, in addition, we switch at each of the $c$ remaining crossings, then we obtain a $2(d-(z^++z^-))$-graded ruling.\footnote{We can obtain additional $2(d-(z^++z^-))$-graded rulings from switching at a combination of the remaining crossings, but we don't require a precise count for our proof.} Therefore, if $z^++z^-$ is not equal to either $ {z^+}'+{z^-}'$ or $d-({z^+}'+{z^-}')$, the number of $2(d-(z^++z^-))$-graded rulings of $\La[w, c]$ and $\La[w', c]$ differ, and $\La[w, c]\not\cong\La[w',c]$ by Theorem~\ref{thm: ruling_invariant}.

Now consider words $w$ and $w'$ with $z^++z^-={z^+}'+{z^-}'$. \cite[Theorem 5.8]{OzsvathStipsicz} implies that the transverse push-offs of $\La[(ZZ)^{z^+}(SS)^{d-{z^+}}, c]$ and $\La[(ZZ)^{{z^+}'}(SS)^{2d-{z^+}'}, c]$ are not transversely isotopic if $z^+\neq {z^+}'$ for $1\leq z^+\leq d$. Since transverse push-offs are invariant under negative stabilization of the underlying Legendrian, this implies that no negative stabilizations of $\La[(ZZ)^{z^+}(SS)^{2d-{z^+}}, c]$ and $\La[(ZZ)^{{z^+}'}(SS)^{2d-{z^+}'}, c]$ are Legendrian isotopic. A local computation shows that the tangles $Z^-$ and $S^-$ become Legendrian isotopic after a single negative stabilization; see e.g. \cite[Figure 14]{EtnyreNgVertesi13}.
\end{proof}

Since the Legendrian isotopy class of $\La[w, c]$ depends only on $z^+$, $z^-$, and $c$, we may now assume without loss of generality that $w$ is of the form $(ZS)^{z^+}(SZ)^{z^-}(SS)^{\frac{m}{2}}$ for non-negative integers $z^+, z^-,$ and $m=2(d-(z^++z^-))$.

\subsection{Braid representatives of double twist knots}\label{ssec: braidification}

In this subsection, we give an inductive procedure for modifying the front projection of our twist knots in order to be able to make use of existing tools for constructing and distinguishing fillings of Legendrian knots. Given a word $w=(ZS)^{z^+}(SZ)^{z^-}$, we will produce a positive braid $\beta$ so that the $(-1)$-closure of $\beta$ is Legendrian isotopic to $\La[w, c]$. To set notation, for $\beta \in \Br_n$, we denote by $\Sh(\beta)$ the image of the shift homomorphism $\Sh:\Br_n \to \Br_{n+1}$ defined by $\Sh(\sigma_i) = \sigma_{i+1}$. The following is a restatement of Theorem~\ref{thm: braidification} from the introduction:

\begin{theorem}\label{thm: twist_braids}
    Let $w=(ZS)^{z^+}(SZ)^{z^-}$ for non-negative integers ${z^+}, {z^-}$ with ${z^+}+{z^-}\geq 1$. Then there exists a positive braid $\beta\in \Br_{2+z^++z^-}^+$ such that $\La[w, c]\cong \La(\beta)$. 
\end{theorem}

\begin{center}
    \begin{figure}[h!]{ \includegraphics[width=.7\textwidth]{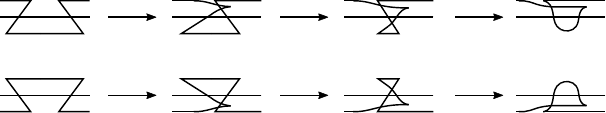}}\caption{Isotopies performed on $ZS$ (top) or $SZ$ (bottom) tangles to produce the $ZS$ or $SZ$ loops we use in the proof of Theorem~\ref{thm: twist_braids}.}
    \label{fig: SZ_blocks}\end{figure}
\end{center} 

\begin{center}
    \begin{figure}[h!]{ \includegraphics[width=\textwidth]{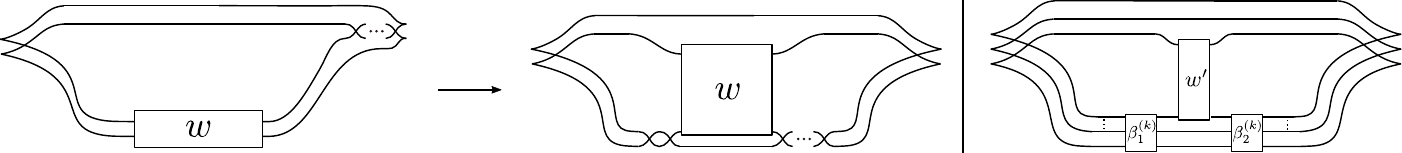}}\caption{On the left, a preparatory isotopy performed prior to the induction argument used in the proof of Theorem~\ref{thm: twist_braids}. The boxed region represents a word $w=(ZS)^{z^+}(SZ)^{z^-}$ of length $-b-1$. On the right, the front obtained after the $k$th step of the induction argument.} 
    \label{fig: base_case}\end{figure}
\end{center} 

\begin{proof}

We prove the statement by first performing an isotopy to produce a front diagram that is easier to manipulate systematically. We then inductively simplify paired $ZS$ and $SZ$ tangles to produce $\La_\beta$. 

We begin by performing the isotopies pictured in Figure~\ref{fig: SZ_blocks} to replace each of the paired $ZS$ or $SZ$ tangles by the corresponding tangle depicted on the right of the figure. We refer to these as $ZS$ or $SZ$ loops corresponding to which of the paired tangles they originated from. We then modify our front following the sequence of isotopy pictured in Figure~\ref{fig: base_case} (left), fixing the $ZS$ and $SZ$ loops. The resulting front provides the base case of our induction argument, and we set $\beta_1^{(1)}=\sigma_1^2$ and $\beta_2^{(1)}=\sigma_1^3$.

\begin{center}
    \begin{figure}[h!]{ \includegraphics[width=.8\textwidth]{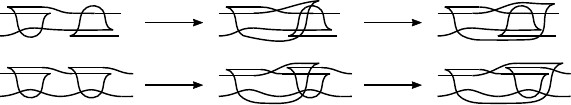}}\caption{Isotopies of $ZS$ loops and $SZ$ loops. Vertically reflecting the sequence of fronts in the top (respectively, bottom) row results in an isotopy describing how an $ZS$ loop interacts with an $SZ$ loop (resp. a $ZS$ loop).}
    \label{fig: SZZSCommuting}\end{figure}
\end{center}

We assume inductively that the tangles represented by the boxes labeled $\beta_{1}^{(k)}$ and $\beta_{2}^{(k)}$ are $(k+1)$-stranded positive braids obtained after simplifying $k-1$ of the $ZS$ or $SZ$ loops. We wish to conclude that we can simplify an additional $ZS$ or $SZ$ loop and produce positive braids $\beta_1^{(k+1)}$ and $\beta_2^{(k+1)}$ so that the front diagram is as pictured in Figure~\ref{fig: base_case} (right). Figure~\ref{fig: SZZSCommuting} depicts the process of sliding strands from the $k$th loop across its adjacent $SZ$ or $ZS$ loops. If the $k$th loop in the word $w$ is an $ZS$ loop, then we obtain the first front of Figure~\ref{fig: ZSinduction} in the end. Performing the sequence of isotopies pictured in the rest of Figure~\ref{fig: ZSinduction} produces a front with one fewer $ZS$ loop and $(k+2)$-stranded braids
\begin{equation*}
\beta_1^{(k+1)}=\sigma_1\dots\sigma_{k+1}\beta_1^{(k)}\sigma_{k+1}  \text{ and } \beta_2^{(k+1)}=\sigma_{k+1}\beta_2^{(k)}\sigma_{k+1}\dots \sigma_2.
\end{equation*}    

\begin{center}
    \begin{figure}[h!]{ \includegraphics[width=\textwidth]{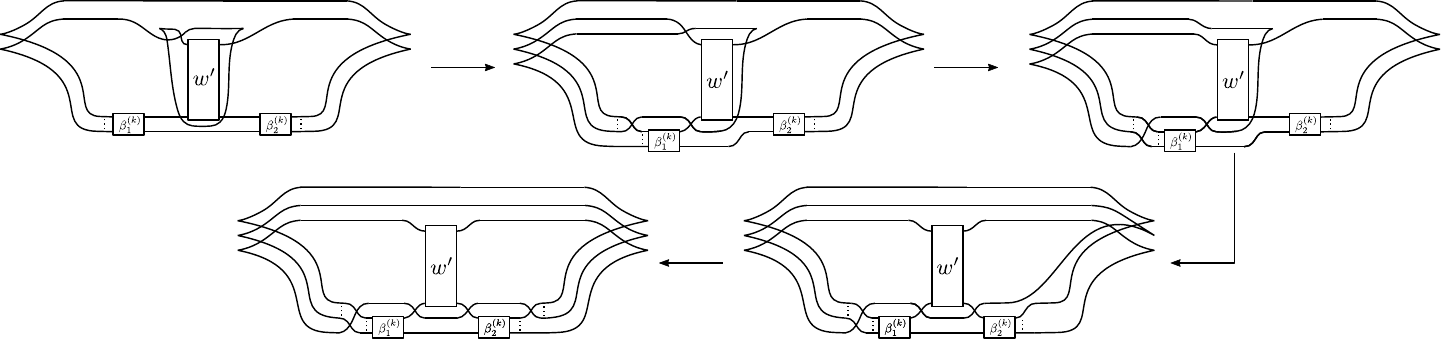}}\caption{Sequence of Legendrian isotopies depicting the induction step in the proof of Theorem~\ref{thm: twist_braids} with a $ZS$ block. The box labeled $w'$ represents the remaining $SZ$ and $ZS$ blocks.}
    \label{fig: ZSinduction}\end{figure}
\end{center}

Similarly, if the $k$th loop in the word $w$ is an $SZ$ loop, then we obtain the first front in Figure~\ref{fig: SZ_induction} by repeatedly applying the isotopies depicted by vertically reflecting those of Figure~\ref{fig: SZZSCommuting}. Performing the remaining sequence of isotopies from Figure~\ref{fig: SZ_induction} then produces a front with one fewer $SZ$ loop and $(k+2)$-stranded braids
\begin{equation*}
\beta_1^{(k+1)}=\sigma_1\Sh(\beta_1^{(k)})\sigma_{1}\dots\sigma_{k+1}\dots\sigma_1  \text{ and } \beta_2^{(k+1)}=\Sh(\beta_2^{(k)})\sigma_1.
\end{equation*}
\begin{center}
    \begin{figure}[h!]{ \includegraphics[width=\textwidth]{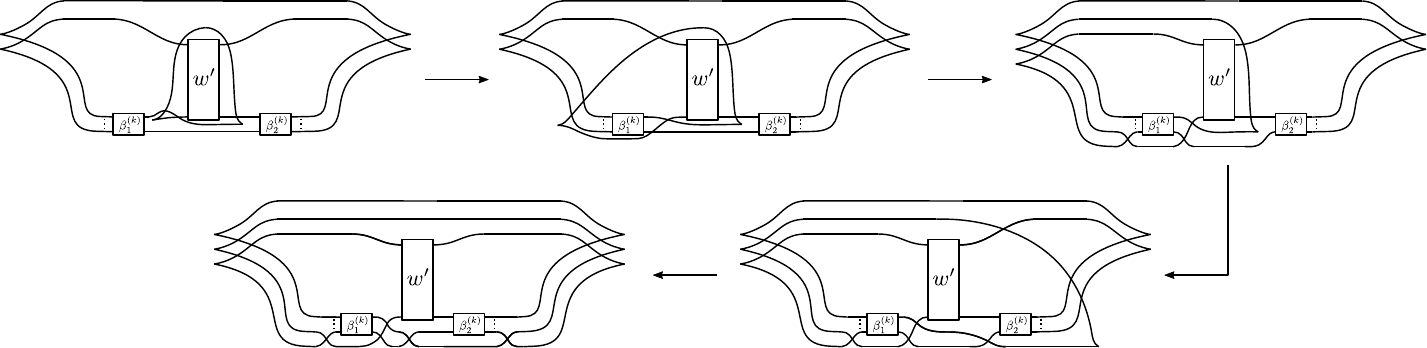}}\caption{Sequence of Legendrian isotopies depicting the induction step in the proof of Theorem~\ref{thm: twist_braids} with an $SZ$ block. The box labeled $w'$ represents the remaining $SZ$ and $ZS$ blocks.}
    \label{fig: SZ_induction}\end{figure}
\end{center} 
\end{proof}

We observe that our inductive argument in the proof above does not depend on the choice of braid in our base case. Given $\beta_1, \beta_2\in \Br_n^+$, let $\La[w, \beta_1, \beta_2]$ be the Legendrian link of the form pictured in Figure~\ref{fig: base_case} (right) replacing $w'$ with $w$ and $\beta_1^{(k)}, \beta_2^{(k)}$ by $\beta_1$ and $\beta_2$. From the proof of Theorem~\ref{thm: braidification}, we obtain the following corollary:

\begin{corollary}\label{cor: General_braidification}
    Let $\beta_1, \beta_2\in \Br_n^+$, $w=(ZS)^{z^+}(SZ)^{z^-}$ for non-negative integers ${z^+}, {z^-}$. Then there exists a positive braid $\beta\in \Br_{n+{z^+}+{z^-}}^+$ such that $\La[w, \beta_1, \beta_2]\cong \La(\beta)$. 
\end{corollary}

\begin{proof}
    The proof is identical to that of Theorem~\ref{thm: braidification}, replacing the initial base case front with the front pictured in Figure~\ref{fig: base_case} (right).
\end{proof}

We now verify that the braids we produce from $\La[(ZS)^{z^+}(SZ)^{z^-}, c]$ satisfy the hypotheses of \cite[Theorem 1.1]{CGGLSS}. Specifically, if $\beta$ is a braid obtained from the word $w=(ZS)^{z^+}(SZ)^{z^-}$, then we wish to show that $\La(\beta)\cong \La(\beta'\Delta_n)$ for $\beta'\in \Br_n^+$ satisfying the additional condition that its Demazure product $\delta(\beta)$ is equal to $w_0$. 

\begin{lemma}
    The braid $\beta\in \Br_{2+z^++z^-}$ obtained from $w=(ZS)^{z^+}(SZ)^{z^-}$ by the process described in Theorem~\ref{thm: twist_braids} is braid equivalent to $\beta'\Delta_{2+z^++z^-}$ for $\beta'\in \Br_{2+z^++z^-}$.  
\end{lemma}

\begin{proof}
We prove the statement by a induction argument similar to the one we use in the proof of Theorem~\ref{thm: twist_braids}. The base case where $w$ is empty is true: the braid we obtain from the empty word is $\beta_1^{(1)}\beta_2^{(1)}=\sigma_1^5$. We assume that $\beta_1^{(k)}$ is braid equivalent to a braid containing a half twist $\Delta_{k+1}$ and consider $\beta_1^{(k+1)}$. 

We first examine the case where the $k$th loop in $w$ is a $ZS$ loop. Since $\Delta_{k+1}\sigma_i=\sigma_{k+1-i}\Delta_{k+1}$ in $\Br_{k+1}^+$, we can assume that $\beta_1^{(k)}$ begins with $\Delta_{k+1}$. Then $\sigma_1\dots\sigma_{k+1}\Delta_{k+1}=\Delta_{k+2}$, so $\beta_1^{(k+1)}$ contains a half twist. In the case where the $k$th loop in $w$ is an $SZ$ loop, we apply a braid equivalence to ensure that $\beta_1^{(k)}$ ends with $\delta_{k+1}$. The braid $\beta_1^{(k+1)}$ then contains $\Sh(\Delta_{k+1})\sigma_1\dots \sigma_{k+1}=\Delta_{k+2}$, as desired.
\end{proof}

In order to ensure that the hypotheses of \cite[Theorem 1.1]{CGGLSS} are satisfied by the braid we produce, it remains to show that the Demazure product is equal to $w_0$. Rather than checking this directly, we observe that when $\delta(\beta)$ is less than $w_0$ in Bruhat order, $X(\beta)=\emptyset$. Since $X(\beta)\cong \Aug(\beta\Delta)$ and Proposition~\ref{proposition: aug_variety} below implies that the augmentation variety is nonempty, it follows that any braid produced from Theorem~\ref{thm: twist_braids} is equivalent to one of the form $\beta\Delta$ with $\delta(\beta)=w_0$. As an immediate corollary, we have that $\Aug(\La[(ZS)^{z^+}(SZ)^{z^-}, c])$ admits a cluster structure. We note that the method we use in Section~\ref{sec: cluster-double-twist} to produce a cluster structure on $\Aug_m(\La[w,c])$ does not rely on the results of \cite{CGGLSS}, though it is possible show that the cluster structure we produce is equivalent to the structure obtained from \cite{CGGLSS}.

\subsection{Motivation for double twist knots}\label{ssec: motivation}

We conclude this section with a brief explanation for our choice of double twist knots as the main object of our study. Any knot type that is well-suited to our comparative study must satisfy two properties: multiple orientably fillable Legendrian representatives and computable Legendrian invariants.
Braid positive knots, the class of Legendrians for which cluster-theoretic methods are most established, are expected to be Legendrian simple; see e.g.\ \cite{MR2629769, CasalsLagSkel} for results in this direction. Moreover, many Legendrian non-simple knot types for which we have a complete classification -- such as negative torus links, classified in \cite{DaltonEtnyreTraynor} -- do not admit orientable exact Lagrangian fillings. Even when we do have Legendrian non-simple knots that are Lagrangian fillable, our methods of computation can become impractical rather quickly. For example, while positive cables of uniformly thick knot types are classified in \cite{CEMR25}, there does not yet exist a straightforward criteria for determining when such Legendrians correspond to cluster structures. We give a sufficient condition in Proposition~\ref{proposition: cluster_cables}, but even in cases satisfying this condition, the combinatorial methods of \cite{CGGLSS} used to identify  cluster structures quickly become unwieldy, making results difficult to obtain in any level of generality. We demonstrate this difficulty with an example in Section~\ref{sec: cable_computation}.  

This leaves us with twist knots as a natural candidate for our exploration. Legendrian twist knots were classified in \cite{EtnyreNgVertesi13}, and include an infinite family that admits orientable fillings. However, twist knots have slice genus one, constraining the cluster structures that one could associate to them and significantly limiting the scope of our exploration. We therefore expand the class of knots we consider to include double twist knots. As the double twist knots we consider have slice genus $g_4(K_{b, c})=\frac{c}{2}$, we obtain more complicated cluster structures than those associated to twist knots, albeit still of finite type, as we show in the following section.

\section{Augmentation varieties of Legendrian double twist knots}\label{sec: DGA computations}

In this section, we compute for even $c \ge 2$, the commutative Chekanov-Eliashberg DGA of the double twist knot $\La[w, c]$ and its augmentation variety. Similar computations for certain twist knots have also been carried out in \cite[Proposition 4.10]{GaoRutherford}. To set the notation, let $w_i$ denote the $i$-th tangle in the region $w$, counted from left to right. By Theorem~\ref{thm: z+z-words}, we continue to assume that $w$ is of the form $(ZS)^{z^+}(SZ)^{z^-}(SS)^{\frac{m}{2}}$ for even $m$; throughout this section, any mention of $w$ and $m$ refers to this fixed choice of parameters.

\begin{center}
    \begin{figure}[h!]{ \includegraphics[width=.6\textwidth]{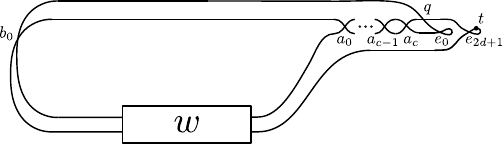}}\caption{Lagrangian projection of a knot Legendrian isotopic to $\La[w,c]$ used for DGA computations.} 
    \label{fig: 2bridgeDGALagrangian}\end{figure}
\end{center} 

For our purpose of computing augmentations induced from exact Lagrangian fillings, we introduce a positive and a negative crossing to the front in Figure~\ref{fig: base_case} (right) by performing a Reidemeister II move at the right cusp and compute the DGA using Ng's resolution \cite{Ng03} of the front, shown in Figure~\ref{fig: 2bridgeDGALagrangian}. We label the generators of the DGA as follows. On the right of Figure~\ref{fig: 2bridgeDGALagrangian}, let $e_0,e_{2d+1}$ denote the two Reeb chords corresponding to the cusps, $q$ the negative crossing, and $a_0, a_1,\ldots,a_c$ the positive crossings ordered from left to right. Let $b_0$ denote the positive crossing on the left of Figure~\ref{fig: 2bridgeDGALagrangian}. For each tangle $w_j$, let $e_j$ denote the Reeb chord corresponding to the right cusp of the tangle and $b_j$ the Reeb chord corresponding to the twist. Finally, place a base point $t$ next to $e_{2d+1}$. The $\Z$-valued gradings of the generators are determined by:

\begin{align*}
    |e_j| &= 1, && 0 \leq j \leq 2d+1, \\
    |a_j| &= \begin{cases}
        -m & \quad \ \ \text{if } j \text{ is odd} \\
        m & \quad \ \ \text{if } j \text{ is even}
    \end{cases}, && 0 \leq j \leq c+1,\\
    |q| &= 1 - m, \\
    |b_j| &= |t| = 0, && 0 \leq j \leq 2d.
\end{align*}

To compute the differential of $\SA(\La[w, c])$, we introduce a family of recursively defined polynomials that compactly package the resulting terms, known as the \dfn{continuant polynomials}. We use them to parametrize all asymptotic configurations of an immersed disk near a sequence of positive crossings in a Legendrian link. 

Continuants previously appeared explicitly in this context in \cite{Hughes2021b} where they were used to describe the Chekanov-Eliashberg DGA of Legendrian $(2,n)$-torus links and parametrize their augmentations. Similar computations also appear in \cite{Kalman-braid,GSW,CGGS1} for braid-positive Legendrian links, as well as \cite{CSHW2023} for Legendrian 2-bridge links.
These polynomials also naturally reflect the $A_n$-type cluster structure of the augmentation variety; see \cite[Section 4.2]{CSHW2023} for more details. Continuant polynomials are defined recursively by
\begin{equation*}
    K_n(x_1,\ldots,x_n) = x_1K_{n-1}(x_2,\ldots,x_n) + K_{n-2}(x_3,\ldots,x_n)\quad \text{with}\quad K_0() = 1,\ \text{and}\ K_1(x_1) = x_1.
\end{equation*}
For instance, $K_2(x_1,x_2) = x_1x_2 + 1$, $K_3(x_1,x_2,x_3) = x_1x_2x_3 + x_1 + x_3$, and $K_4(x_1,x_2,x_3,x_4) = x_1x_2x_3x_4 + x_1x_2 + x_1x_4 + x_3x_4 + 1$. 
We shall associate these polynomials to the region of positive crossings $a_0,\ldots,a_c$ and use the shorthand
\begin{equation*}
    K_{c+1} = K_{c+1}(a_0,\ldots,a_c), \qquad K_{c-1}^M = K_{c-1}(a_1,\ldots,a_{c-1}),\qquad
    K_c^L = K_c(a_0,\ldots,a_{c-1}).
\end{equation*}
where the superscript $L$ can be understood as denoting the leftmost $c$ crossings in the region and $M$ denotes the middle $c-1$ crossings.

Finally, let $W^{ZS}_{z^+} := \prod_{i = 1}^{z^+}(1+b_{2i-2}b_{2i-1})$ and $W^{SZ}_{z^-} := \prod_{i = 2}^{z^-}(1+b_{2i-1}b_{2i})$; we use $W^{ZS}_{z^+}$ and $W^{SZ}_{z^-}$ below to parametrize the asymptotic configurations of an immersed disk near a sequence of $SZ$'s and $ZS$'s, respectively. These polynomials do not appear in our description of the augmentation variety of $\La[w, c]$ below; we include them for the sake of completeness.

We are now in position to state the differential of $\SA(\La[w, c])$ over $\Z[t^{\pm 1}]$. The non-vanishing components of the differential of $\SA(\La[w, c])$ are given by
    \begin{align*}
    \partial e_0 &= \begin{cases}
        \alignbox{1 + b_0K_{c}^L} & \text{if } z^+ \not = 0\ \text{or}\ m \not = 0 \\
        \alignbox{1 + b_0K_{c}^L + W^{SZ}_{z^-}K_{c-1}^M} & \text{if } z^+,m  = 0
    \end{cases}, \\
    \partial e_j &= \hphantom{\left\{\vphantom{\begin{cases}0\\0\end{cases}}\right.} \alignbox{1 + b_{j-1}b_j} \quad 1 \leq j \leq 2d, \\
    \partial e_{2d+1} &= \begin{cases}
        \alignbox{t^{-1} + K_{c}^L b_{2d}} & \text{if } z^- \not = 0\ \text{or}\ m \not = 0 \\
        \alignbox{t^{-1} + K_{c}^L b_{2d} - K_{c-1}^MW^{ZS}_{z^+}} & \text{if } z^-,m  = 0
    \end{cases}, \\
    \partial q &= \begin{cases}
        \alignbox{b_0K_{c+1}} & \text{if } z^+ \not = 0\ \text{or}\ m \not = 0 \\
        \alignbox{b_0K_{c+1} + W^{SZ}_{z^-}K_{c}^R} & \text{if } z^+,m  = 0
    \end{cases}, \\
    \partial a_0 &= \begin{cases}
        \alignbox{0} & \text{if } z^- \not = 0\ \text{or}\ m \not = 0 \\
        \alignbox{W^{ZS}_{z^+}} & \text{if } z^-,m  = 0
    \end{cases},
    \end{align*}

Having established the precise DGA structure, we now proceed to characterize the graded augmentation varieties of $\La[w,c]$.

\begin{proposition}\label{proposition: aug_variety}
    For $m' \not = 1$, the $m'$-graded augmentation variety of $\La[w, c]$ over a field $\Bbbk$ is
    \begin{equation*}
        \Aug_{m'}(\La[w,c])\cong \begin{cases}
         \left\{(a_1,\ldots,a_c) \in \Bbbk^c \ |\ K_c(a_1,\ldots,a_c) \not = 0 \right\}, & m'|m \\
         \{pt\} & m' \nmid m
        \end{cases}.
    \end{equation*}
\end{proposition}
\begin{proof}
    An $m'$-graded augmentation $\epsilon$ of $\La[w, c]$ is defined by $\epsilon \circ \partial a_0 = 0$, $\epsilon \circ \partial q = 0$, and $\epsilon \circ \partial e_j$ for $j = 0,\ldots, 2d+1$; moreover, $\epsilon(a) \not = 0$ only if $|a| = 0 \mod m'$. In particular, note that $\epsilon(e_j) = 0$, $\epsilon(a_i) = 0$ if $m' \nmid m$, and $\epsilon(q) = 0$ if $m' \not = m-1$. Since $\La[w, c]$ is a knot, \cite[Theorem 1.1]{Leverson} also ensures that $\epsilon(t) = -1$.
    
    We simplify the system of equations that define an augmentation of $\La[w,c]$ as follows. First, in the case when $z^+,m = 0$, observe that $\epsilon(W_{z^-}^{SZ}) = \epsilon(\prod_{i = 1}^{z^-} \partial e_{2i}) = \prod_{i = 1}^{z^-} \epsilon \circ \partial e_{2i} = 0$. Thus, the equations $\epsilon \circ \partial e_0 = 0$ and $\epsilon \circ \partial q = 0$ reduce to
    \begin{equation*}
        \epsilon(b_0K_c^L) = -1\quad \text{and}\quad \epsilon(b_0K_{c+1}) = 0,
    \end{equation*}
    respectively in all cases. The latter condition is equivalent to $\epsilon(K_{c+1}) = 0$. 
    
    Similarly, in the case when $z^-,m = 0$, $\epsilon(W_{z^-}^{SZ}) = 0$ so that the equation $\epsilon \circ \partial a_0 = 0$ is vacuously true and $\epsilon \circ \partial e_{2d+1} = 0$ reduces to
    \begin{equation*}
        \epsilon(K_c^Lb_{2d}) = -1,
    \end{equation*}
    in all cases. Now, the system of equations $\epsilon(1+b_{j-1}b_j) = \epsilon \circ \partial e_j = 0$ for $j = 1,\ldots,2d$ implies that the values of $\epsilon(b_j)$ are determined inductively from $\epsilon(b_0)$. In particular, $\epsilon(b_j) = -1/\epsilon(b_0)$ if $j$ is odd and $\epsilon(b_j) = \epsilon(b_0)$ if $j$ is even. Thus, the conditions $\epsilon(b_0K_c^L) = -1$ and $\epsilon(K_c^Lb_{2d}) = -1$ are equivalent, and are both equivalent to $\epsilon(K_c^L) \not = 0$. By \cite[Lemma 3.6]{CSHW2023}, we know that $\epsilon(K_{c+1}) = 0$ implies $\epsilon(K_c^L) \not = 0$. Moreover, note from the recurrence relation of the continuants that $\epsilon(K_{c+1}) = 0$ if and only if $\epsilon(K_c^R) \not = 0$, in which case $a_0$ is determined. Therefore, the naïve augmentation variety of $\La[w, c]$ is an affine variety given by
    \begin{equation*}
        \widetilde{\Aug}_{m'}(\La[w,c])\cong \begin{cases}
         \left\{(a_1,\ldots,a_c) \in \Bbbk^c \ |\ K_c(a_1,\ldots,a_c) \not = 0 \right\} & m'|m \\
         \left\{q \in \Bbbk \right\} & m' = m-1 \\
         \{pt\} & m' \nmid m, m' \not = m-1
        \end{cases}.
    \end{equation*}
    Since $c$ is even, the unique $m'$-graded augmentation $\epsilon$ of $\La[w,c]$ when $m' \nmid m, m' \not = m-1$ is the one that sends $b_0$ to $-1$. Then, $\epsilon(b_j)$ is determined as above, $\epsilon(t) = -1$, and $\epsilon(a) = 0$ on all other generators $a$.

    If $m' \mid m-1$, then an augmentation is non-vanishing only on generators with a zero differential, meaning no DGA homotopy can exist. Assume that $m' = m-1$. It suffices to show that in this case, any two $m'$-graded augmentations $\epsilon_1,\epsilon_2$ of $\La[w,c]$ are DGA homotopic. We define a map $h:\SA(\La[w,c]) \to \Bbbk$ on the generators of $\SA(\La[w,c])$ as follows. First, define $h(t) = -1$, $h(q) = 0$, 
    $h(e_j) = 0$ for $j = 1, \ldots, 2d+1$, and $h(a_j)= 0$ for $j = 1,\ldots,c$. Define $h(b_0) = -1$. Then, $h(b_j) = -h(b_0)$ if $j$ is odd and $h(b_j) = h(b_0)$ if $j$ is even. Finally, define 
    \begin{equation*}
        h(a_0) = \epsilon_2(q)-\epsilon_1(q).
    \end{equation*}
    We see that for $a \not = q$, $h(\partial a) = 0 = \epsilon_1(a)-\epsilon_2(a)$, and $h(\partial q) = h(b_0K_{c+1}) = \epsilon_1(q)-\epsilon_2(q)$ 
    due to the recurrence relation defining $K_{c+1}$ and the fact that $c$ is even. Thus, $h$ is a DGA homotopy between $\epsilon_1$ and $\epsilon_2$.
\end{proof}

We end this section with several immediate applications of Proposition~\ref{proposition: aug_variety}. First, by Theorem~\ref{thm: z+z-words}, we obtain a family of distinct max-tb Legendrian double twist knots whose graded augmentation varieties are isomorphic.

\begin{corollary}
For $m_1,m_2 \not = 0$, $\Aug_{m_1}(\La[w, c])\cong \Aug_{m_2}(\La[w', c])$ if $m_1,m_2 \mid m$ or $m_1,m_2 \nmid m$. 
\end{corollary}

Second, we can fully characterize the Maslov-0 fillable representatives of $\La[w, c]$ and give a partial characterization of Maslov-$m$ fillability of $\La[w, c]$.

\begin{proposition}\label{proposition: Maslov-m fillability}
$\La[w, c]$ admits a Maslov-0 filling if and only if $w=(ZS)^{z^+}(SZ)^{z^-}$. For $w=(ZS)^{z^+}(SZ)^{z^-}(SS)^{\frac{m}{2}}$, $\La[w,c]$ does not admit a Maslov-$m'$ filling if $m'$ does not divide $m$. 
\end{proposition}

\begin{proof}
In Section~\ref{ssec: filling}, we construct Maslov-0 fillings of $\La((ZS)^{z^+}(SZ)^{z^-})$, proving one direction of the statement. We now obstruct the existence of Maslov-0 fillings of $\La[(ZS)^{z^+}(SZ)^{z^-}(SS)^{\frac{m}{2}}, c]$ when $m>0$. In this case, we know by Proposition~\ref{proposition: aug_variety} that $\Aug(\La[w,c])\cong \{pt\}$. By \cite[Proposition 1.2]{GaoRutherford}, however, any Maslov-$0$ exact Lagrangian filling of $\La[w,c]$ would induce an injective map $(\Bbbk^\ast)^{c/2}\hookrightarrow\Aug(\La[w,c])$. Therefore, there is no Maslov-$0$ filling of $\La[w,c]$ when $m>0$. The obstruction of Maslov-$m$ fillability of $\La[w,c]$ follows the same argument.
\end{proof}

\begin{remark} 
We expect that $\La[w, c]$ admits a Maslov-$m$ filling if and only if $w$ is equivalent (in the sense of Theorem~\ref{thm: z+z-words}) to a word of the form $(ZS)^{z^+}(SZ)^{z^-}(SS)^{\frac{m}{2}}$. Given that we construct Maslov-$m$ fillings in Section~\ref{ssec: filling}, the difficulty in proving such a statement lies in obstructing Maslov-$m'$ fillings for $m'$ dividing $m$. In particular, the obstruction used to prove Proposition~\ref{proposition: Maslov-m fillability} fails, as e.g. $\Aug_2(\La[w,c])\cong \Aug_2(\La[w', c])$ for all words $w, w'$ of the same length. 
\end{remark}

 \begin{remark}
Given a word $w=(ZS)^{z^+}(SZ)^{z^-}$, Proposition~\ref{proposition: Maslov-m fillability} tells us that we can perform an operation swapping $SZ$ for $ZS$ or vice-versa to obtain a new Maslov-0 fillable Legendrian representative of the same knot. Corollary~\ref{cor: General_braidification} gives a method for producing $(-1)$-closures of positive braids for possibly distinct Legendrian representatives of a broader class of knot type related by pairs of $SZ$ or $ZS$ swaps than our double twist knots. While there are examples of distinct max-tb Legendrian representatives that are not related by swapping $S$ with $Z$ \cite[Section 7]{EtnyreNg03}, those known to the authors do not admit orientable Lagrangian fillings. 
It is not known whether there exist Maslov-0 Lagrangian fillable Legendrian representatives of the same knot type that are not related by swapping $SZ$ with $ZS$ (or vice versa). As a result, we can think of the method of proof for Theorem~\ref{thm: twist_braids} as possibly giving a method for obtaining different Maslov-0 Lagrangian fillable representatives from an initial one.
\end{remark}

As corollaries of Propositions~\ref{proposition: Maslov-m fillability} and Theorem~\ref{thm: z+z-words}, together with the construction of the fillings in Section~\ref{ssec: filling}, we can enumerate the Maslov-$m$ fillable Legendrian representatives of double twist knots.

\begin{corollary}
There are exactly $\lceil \frac{d+1}{2}\rceil$ distinct Maslov-$0$ fillable max-tb Legendrians representatives of $K_{-2d-1, 2}$ and at least $\lceil \frac{d+1}{2}\rceil$ distinct Maslov-0 fillable max-tb Legendrian representatives of $K_{-2d-1, c}$ for $c > 2$ even.
\end{corollary}

\begin{corollary}
For $c, m\geq 2$ even, 
there are at least $d+1-\frac{m}{2}$ Maslov-$m$ fillable Legendrian representatives of $K_{-2d-1, c}$.
\end{corollary}

\section{Cluster structures on augmentation varieties $\Aug_m(\La[w,c])$}
\label{sec: cluster-double-twist}

In this section, we prove that the $m$-graded augmentation variety $\Aug_m(\La[w, c])$ of $\La$ is a cluster variety with every cluster seed corresponding to some Maslov-$m$ exact Lagrangian filling of $\La[w, c]$. Drawing upon the strategies developed in in \cite[Section 5]{CSHW2023}, we first define a cluster structure on $\Aug_m(\La[w, c])$ by identifying it with a subvariety of the cluster variety $\Gr_{2,c+3}^\circ$, a subset of the affine cone of the Grassmannian $\Gr(2,c+3)$. For each cluster seed, we then construct a decomposable Maslov-$m$ exact Lagrangian filling such that the pullback of the induced DGA map is an open embedding of an algebraic torus onto the cluster torus chart. Finally, we discuss how the geometry of the exact Lagrangian fillings we constructed recovers the quivers and cluster mutations of the cluster seeds in $\Aug_m(\La[w, c])$.

\subsection{$\Aug_m(\La[w,c])$ and Grassmannian cluster varieties} 
For the first step, we recollect some necessary background on the Grassmannian cluster variety $\Gr_{2,n}^\circ$. We refer readers to \cite[Section 4.2]{CSHW2023} for more details. First, we define $\Gr_{2,n}^\circ$ as the quotient space
\begin{equation*}
    \SL_2\backslash \{M \in \Mat_{2\times n}\ |\ P_{i,i+1}(M) \not = 0\ \text{for all}\ 1 \leq i < n\ \text{and}\ P_{1,n}(M) \not = 0\},
\end{equation*}
where $P_{i,j}$ is the \dfn{$(i,j)$-th Plücker coordinate}, defined as the determinant of the submatrix formed by the $i$th and $j$th columns of the matrix $M$. By a result of Fomin and Zelevinsky \cite{FominZelevinsky_ClusterII}, the cluster seeds of $\Gr_{2,n}^\circ$ correspond to the triangulations of an $n$-gon. In more details, given a triangulation $\tau$ of an $n$-gon whose vertices are labeled clockwise by $1,2,\ldots,n$, the cluster torus chart $U_\tau \subset \Gr_{2,n}^\circ$ is an open subset cut out by the conditions $P_e := P_{i,j} \not = 0$ for all diagonals or sides $e$ connecting two distinct vertices $i<j$ in $\tau$. In particular, the Plücker coordinates $P_e$ are the cluster variables. The quiver $Q_\tau$ is given by placing a mutable vertex at every diagonal of $\tau$ and a frozen vertex at every side of $\tau$, drawing a counterclockwise $3$-cycle for each triangle in $\tau$, and then deleting any arrows between frozen vertices. Finally, a seed mutation in the direction of $P_{i,j}$ corresponds to a flip of the diagonal connecting $i,j$, which transforms the triangulation $\tau$ to another triangulation $\tau'$.

To connect $\Gr_{2,n}^\circ$ with the augmentation variety, we need to describe $\Gr_{2,n}^\circ$ as an affine variety. Define two special matrices 
\begin{equation*}
    B(x):= \begin{pmatrix}
    x & -1 \\
    1 & 0 
    \end{pmatrix}\quad \text{and}\quad D(x):= \begin{pmatrix}
    x & 0 \\
    0 & x^{-1} 
    \end{pmatrix}.
\end{equation*}
\cite[Section 4.2]{CSHW2023} gives two characterizations of $\Gr_{2,n}^\circ$ as an affine subvariety in $\Bbbk_{a_1,\ldots,a_{n-1}}^{n-1} \times (\Bbbk^\ast)^{n-1}_{p_1,\ldots,p_{n-1}}$ determined by a polynomial in terms of $B(a_j)$ and $D(p_j)$. The characterization relevant to our work is presented below, where $(M)_{ij}$ denotes the $(i,j)$th entry of a matrix $M$.

\begin{proposition}[Proposition 4.15 and 4.16, \cite{CSHW2023}]\label{proposition: gr_eqns} $\Gr_{2,n}^\circ$ is an open affine subvariety in $\Bbbk_{a_1,\ldots,a_{n-2}}^{n-2} \times (\Bbbk^\ast)^{n-1}_{p_1,\ldots,p_{n-2},p_n}$ cut out by the inequality
\begin{equation*}
    (B(a_1)D(p_1)B(a_2)D(p_2)\cdots B(a_{n-2})D(p_{n-2}))_{11} \not = 0.
\end{equation*}
Under this identification, $P_{j,j+1} = p_j$ for $1 \leq j < n$ and $P_{1,n} = p_n$. Moreover, if $p_i= 1$ for $1 \leq i \leq j-1$ and $i = n$, then $P_{j,n} = K_{j-1}(a_{1},\ldots,a_{j-1})$.
\end{proposition}

The proof follows from the observation that by an $\SL_2$-action, every point in $\Gr_{2,n}^\circ$ can be represented by a matrix whose column vectors $v_1,\ldots,v_n$ are $v_1 = \begin{pmatrix}
    1  \\
    0 
\end{pmatrix}$, $v_n = \begin{pmatrix}
    0 \\
    p_n
\end{pmatrix}$, and
\begin{equation*}
    v_j = B(a_1)D(p_1)B(a_2)D(p_2)\cdots B(a_{j-1})D(p_{j-1})v_1,\quad \text{for all}\ 2 \leq j < n,
\end{equation*}
with the condition $P_{n-1,n} = \det(v_{n-1},v_n) \not = 0$. The second part of Proposition~\ref{proposition: gr_eqns} then follows from a direct computation. In particular, the appearance of the continuants can be anticipated by the equality
\begin{equation}
    \label{eq:continuant-B}
    \begin{pmatrix}
    K_n(x_1,\ldots,x_n) & -K_{n-1}(x_1,\ldots,x_{n-1}) \\
    K_{n-1}(x_2,\ldots,x_n) & -K_{n-2}(x_2,\ldots,x_{n-1})
    \end{pmatrix} = B(x_1)\ldots B(x_n),
\end{equation}
as observed in \cite[Lemma 3.3]{CSHW2023}. We now define a cluster structure on the $m$-graded augmentation variety $\Aug_m(\La[w, c])$.

\begin{proposition}
\label{proposition: twist-cluster}
    The $m$-graded augmentation variety $\Aug_m(\La[w, c])$ is a cluster variety with an initial cluster seed given by the following quiver and initial cluster variables (the boxed cluster variable is frozen):

    \begin{equation*}
        K_1(a_1)\rightarrow K_2(a_1,a_2)\rightarrow K_3(a_1,a_2,a_3)\rightarrow \cdots \rightarrow K_{c-1}(a_1,\dots, a_{c-1})\rightarrow \boxed{K_c(a_1,\dots, a_c)}.
    \end{equation*}
\end{proposition}
\begin{proof}
    We employ the same argument used in \cite[Theorem 5.1]{CSHW2023}, detailing the steps as follows. Since $P_{j,j+1} = p_j$ is frozen for $1 \leq j \leq n$, by \cite[Proposition  4.8]{CSHW2023}, the subvariety $V \subset \Gr_{2,c+3}^\circ$ cut out by $p_1 = \cdots = p_c = p_{c+2} = 1$ is also a cluster variety with the cluster seeds obtained from cluster seeds of $\Gr_{2,c+2}^\circ$ by specializing $p_j = 1$. Consider, in particular, the triangulation of the $(c+2)$-gon consisting of diagonals connecting the vertex $c+2$ with vertices $2,\ldots,c$. Following the construction of the cluster structure on $\Gr_{2,c+2}^\circ$ and $V$ above, we see that its corresponding mutable cluster variables are
    \begin{equation*}
        P_{j,c+2} = K_{j-1}(a_{1},\ldots,a_{j-1}),
    \end{equation*}
    for $j = 2,\ldots,c$, where we used Proposition~\ref{proposition: gr_eqns}. Similarly, the only frozen cluster variable remaining is
    \begin{equation*}
        p_{c+1} = P_{c+1,c+2} = K_c(a_{1},\ldots,a_c).
    \end{equation*}
    The triangulation gives the quiver prescribed by the proposition.

    It suffices to show that $V$ is the non-vanishing locus of $K_c(a_1,\ldots,a_c) \not = 0$. Indeed, by Proposition~\ref{proposition: gr_eqns} and Equation~\ref{eq:continuant-B}, setting $p_j = 1$ for $j \not = c+1$ yields
    \begin{equation*}
        (B(a_1)B(a_2)\cdots B(a_c))_{11} = K_c(a_1,\ldots,a_c) \not = 0,
    \end{equation*}
    as desired.
\end{proof}

\subsection{Exact Lagrangian fillings and cluster charts}
\label{ssec: filling}

Our next step is to construct, for every cluster torus chart on $\Aug_m(\La[w, c])$, an exact Lagrangian filling of $\La[w, c]$ whose induced family of augmentations coincides with the given chart. As in \cite{CSHW2023} and \cite{GSW}, our strategy is to apply a sequence of pinching cobordisms described in Section~\ref{ssec:contact} to some permutation of the Reeb chords $a_1,\ldots,a_c$. These moves reduce the knot to the max-tb unknot $U$, which we finally cap with its unique disk filling. We obtain all of the desired torus charts by varying the order in which the Reeb chords are pinched.

To begin with, we observe that the specified Reeb chords are contractible and simple in Figure~\ref{fig: 2bridgeDGALagrangian}, as required by a pinch cobordism. Indeed, for $i \not = 0$, every Reeb chord $a_i$ is trapped in between two positive crossings in the front diagram so that a local move given in \cite[Section 3.1]{Hughes2021b} shows that they are contractible. For $1 \leq i \leq c$, the disks with multiple positive punctures including one at $a_i$ are the bigons with the other positive puncture at $a_{i-1}$ or $a_{i+1}$. Those with a positive puncture at $a_c$ consists of a bigon with the other positive puncture at $a_{c-1}$ and a disk with a positive puncture at $e_0$ and a negative puncture at $q$. In particular, there is no disk with multiple positive punctures at the same Reeb chord $a_i$ and these Reeb chords are simple. 

Proceeding inductively, note that after pinching a number of Reeb chords in $a_1,\ldots,a_c$, the rest of them remain contractible and simple by a same argument. We use a permutation $P = (i_1,\ldots,i_c)$ of the indices in $\{1,\ldots,c\}$ to denote a \dfn{pinching sequence} where the Reeb chords $a_{i_1},\ldots,a_{i_c}$ of $\La[w, c]$ are pinched in the order listed. Let $\La_j[w, c]$ denote the Legendrian link obtained after pinching the Reeb chords $a_{i_1},\ldots,a_{i_j}$. Recall from Section~\ref{ssec:ce-dga} that compared to $\La_{j-1}[w,c]$, $\La_j[w,c]$ has two additional base points $s_j$ and $-s_j^{-1}$ placed on the top and bottom of the $0$-resolution of $a_{i_j}$. In the end, we obtain the max-tb unknot $U:= \La_{c}[w,c]$ with $2c+1$ base points $s_1,\ldots,s_c,-s_1^{-1},\ldots,-s_c^{-1},$ and $t$. Denote by $L_P$ the cobordism from $U$ to $\La[w, c]$ via the pinching sequence $P$ and $\Phi_{L_P}$ its induced DGA map. 

As in \cite{CSHW2023}, we prove that each pinch move gives a cluster localization of the cluster variety $\Aug_m(\La[w, c])$, so that the full sequence restricts to a cluster torus chart. To this end, we recall the following proposition proven in \cite[Proposition 4.22]{CSHW2023} about a specific form of cluster localization of $\Gr_{2,n}^\circ$ that we apply below. 

\begin{proposition}[Proposition 4.21 and 4.22, \cite{CSHW2023}] 
\label{proposition: cluster-localization}
Consider $\Gr_{2,n-1}^\circ$ as the non-vanishing locus of $(B(b_1)D(r_1)B(b_2)D(r_2)\cdots B(b_{n-2})D(r_{n-2}))_{11}$ and set $r_0 := r_{n-1}$. For each $1 \leq i < n$, define a map $\phi_i: \Gr_{2,n-1}^\circ \times (\Bbbk^\ast)^2_{u,v} \to \Gr_{2,n}^\circ$ by
\begin{equation*}
    \phi_i^*(a_j) = \begin{cases}
        b_j & \text{if}\ j < i-1, \\
        b_{i-1}+\frac{v}{ur_{i-1}} & \text{if}\ i > 1\ \text{and}\ j = i-1, \\
        \frac{r_{i-1}}{uv} & \text{if}\ j = i, \\
        b_i+\frac{u}{vr_{i-1}} & \text{if}\ j < n-1\ \text{and}\ j = i+1, \\
        b_{j-1} & \text{if}\ j > i+1. 
    \end{cases}\quad \phi_i^*(p_j) = \begin{cases}
        r_j & \text{if}\ j < i-1, \\
        u & \text{if}\ j = i-1, \\
        v & \text{if}\ j = i, \\
        r_{j-1} & \text{if}\ j > i.
    \end{cases}\
\end{equation*}
After adding a $3$-cycle $u\rightarrow r_{i-1} \rightarrow v\rightarrow u$ of frozen vertices to each cluster seed of $\Gr_{2,n-1}^\circ$, $\phi_i$ becomes an isomorphism of cluster varieties onto its image, which is the cluster localization at $P_{2,n}$ if $i = 1$ and at $P_{i-1,i+1}$ otherwise.
\end{proposition}

The following statement proves the second part of Theorem~\ref{thm: main}:

\begin{theorem}\label{thm: cluster_charts}
    The induced morphism $\phi_{L_P}:\Aug_m(U) \to \Aug_m(\Lambda[w,c])$ is an open embedding of an algebraic torus onto a cluster torus chart on $\Aug_m(\Lambda[w,c])$. Conversely, every cluster torus chart in $\Aug_m(\Lambda[w,c])$ is the image of $\phi_{L_P}$ for some pinching sequence $P$.
\end{theorem}
\begin{proof}
    We begin by verifying that the augmentation variety $\Aug_m(U)$ of the unknot is isomorphic to $(\Bbbk^\ast)^c$, parametrized by the variables $s_1,\ldots,s_c$. An inspection of the unknot $U$ obtained by resolving $a_1,\ldots,a_c$ in Figure~\ref{fig: 2bridgeDGALagrangian} shows that the differential of the DGA of $U$ is given by
    \begin{align*}
    \partial e_0 &= \hphantom{\left\{\vphantom{\begin{cases}0\\0\end{cases}}\right.} \alignbox{1 + b_0s_1^{-1}\cdots s_c^{-1}} \\
    \partial e_{2d+1} &= \hphantom{\left\{\vphantom{\begin{cases}0\\0\end{cases}}\right.} \alignbox{t^{-1} + s_1\cdots s_cb_{2d}} \\
    \partial e_j &= \hphantom{\left\{\vphantom{\begin{cases}0\\0\end{cases}}\right.} \alignbox{1 + b_{j-1}b_j} \quad 1 \leq j \leq 2d, \\
    \partial q &= \begin{cases}
        \alignbox{b_0a_0s_1\cdots s_c} & \text{if } z^+ \not = 0\ \text{or}\ m \not = 0 \\
        \alignbox{(b_0a_0 + W^{SZ}_{z^-})s_1\cdots s_c} & \text{if } z^+,m  = 0
    \end{cases}, \\
    \partial a_0 &= \begin{cases}
        \alignbox{0} & \text{if } z^- \not = 0\ \text{or}\ m \not = 0 \\
        \alignbox{W^{ZS}_{z^+}} & \text{if } z^-,m  = 0
    \end{cases}.
\end{align*}
Arguing as in the proof of Proposition~\ref{proposition: twist-cluster}, we deduce that $\Aug_m(U)$ is cut out by $s_1\cdots s_c \not = 0$ in $\Bbbk^c_{s_1,\ldots,s_c}$ . The other augmentation values are determined by $\epsilon(b_j) = \epsilon(s_1\cdots s_{c+1})$ for even $j$, $\epsilon(b_j) = 1/\epsilon(s_1\cdots s_{c+1})$ for odd $j$, $\epsilon(t) = 1$, and $\epsilon(e_j) = \epsilon(a_0) = \epsilon(q) = 0$. This non-vanishing locus is precisely $(\Bbbk^\ast)^c_{s_1,\ldots,s_c}$ as claimed.

Next, we prove that the pullback $\phi_{L_P}$ of the cobordism DGA map $\Phi_{L_P}$ embeds this torus onto a cluster chart. It suffices to consider only the images of the Reeb chords $a_1,\ldots,a_c$ under $\Phi_{L_P}$, which define the augmentation variety $\Aug_m(\La[w, c])$. Consider an intermediate pinch cobordism from $\La_{i_{j-1}}(w,c)$ to $\La_{i_j}(w,c)$. To match with the equation in Proposition~\ref{proposition: cluster-localization} in the end, we relabel the relevant Reeb chords in $\La_{i_j}(w,c)$ from left to right by $a_1,\ldots,a_{c+1-j}$ so that the Reeb chord resolved is $a_i$. Denote by $\Phi_{L_i}$ the DGA map induced by this pinch cobordism. According to the formula for a pinch cobordism map described in Section~\ref{ssec:ce-dga}, $\Phi_{L_i}$ is the identity for all those Reeb chords except for $a_i$ and its adjacent chords, $a_{i-1}$ (if $i > 1$) and $a_{i+1}$ (if $i < c+1-j$), in which case they are
\begin{equation*}
    \Phi_{L_i}(a_{i-1}) = a_{i-1} + u^{-2}s^{-1}_i,\quad \Phi_{L_i}(a_i) = s_i,\quad \Phi_{L_i}(a_{i+1}) = a_{i+1} + v^{-2}s^{-1}_i,
\end{equation*}
where $u$ (resp. $v)$ is the product of $s_j$ that are already sandwiched between $a_{i-1}$ and $a_i$ (resp. $a_i$ and $a_{i+1}$) before the pinching. This map coincides with the map $\phi^*_i$ in Proposition~\ref{proposition: cluster-localization} by substituting $r_{i-1}$ with $us_iv$. Thus, each pinch cobordism induces a cluster localization of the cluster variety $\Aug_m(\La[w, c])$ and after $c$ localizations, we obtain a cluster torus chart with local coordinates $s_j$'s that can be identified with $\Aug_m(U)$.

It remains to specify a corresponding pinching sequence $P$ for every cluster chart on $\Aug_m(\La[w, c])$. Let $\tau$ be a triangulation of the $(c+2)$-gon associated to the cluster seed of $\Aug_m(\La[w, c])$. The construction of $P$ follows the same recipe as in \cite[Section 2.3.1]{Hughes2021b} and \cite[Proposition 5.11]{CSHW2023}: a cluster localization corresponds to removing and recording a vertex with no incident diagonals; since this operation preserves a triangulation of the remaining $n$-gon, we can proceed by induction. In more details, we homotope $\tau$ so that all its vertices line up in a straight line and its unique ``frozen" side is represented by a largest semicircular arc at the bottom, connecting the leftmost and rightmost vertices. Similarly, the remaining sides of the $(c+2)$-gon are represented by consecutive small semicircular arcs on the top and the diagonals in $\tau$ become semicircular arcs sandwiched in the middle. We label from the second leftmost vertex to the second rightmost vertex by $a_1,\ldots,a_c$. Then, a pinching sequence corresponding to $\tau$ is obtained by going over the diagonals from top to bottom while reading off and deleting the unique vertex opposite to each of the diagonals. Note that if there are multiple diagonals on the same level, then any choice of ordering of these diagonals will work. See Figure~\ref{fig: TriangulationEx} for an example.
\end{proof}

Since every cluster chart in $\Aug_m(\La[w, c])$ corresponds to a triangulation of the $(c+2)$-gon, it follows from Theorem~\ref{thm: cluster_charts} that pinching $a_1,\ldots,a_{c}$ in some order induces at least $C_{c}$ distinct Hamiltonian isotopy classes of Maslov $m$ exact Lagrangian fillings of $\La[w, c]$, where $C_{c}$ is the $c$-th Catalan number. Moreover, observe from the construction at the end of the proof of Theorem~\ref{thm: cluster_charts} that two pinching sequences induce the same triangulation only if they differ by swapping the orders of pinching at pairs of Reeb chords $a_i$ and $a_j$ with $|j-i| \ge 2$. By \cite[Lemma 5.14]{CSHW2023} or more specifically the local model provided in \cite[Proposition 3.1]{Hughes2021b}, $a_i$ and $a_j$ are simultaneously contractible in this case so that the fillings induced by the two pinching sequences are both Hamiltonian isotopic to the filling obtained by pinching $a_i$ and $a_j$ at once. Hence, pinching sequences that induce the same cluster chart yield Hamiltonian isotopic fillings. Finally, all the fillings we constructed are obtained by successive $1$-handle attachments and a $0$-handle attachment at the end; thus, they are all smoothly isotopic. In summary, we obtain the following corollary.

\begin{corollary}
    There are $C_{c}$ smoothly isotopic but Hamiltonian non-isotopic Maslov-$m$ exact Lagrangian fillings of $\La[w, c]$ constructed from pinching sequences of Reeb chords $a_1,\ldots,a_c$.
\end{corollary}

As discussed in Section~\ref{ssec: braid_vties}, the quivers associated to the cluster seeds in $\Aug_m(\La[w, c])$ are expected to be the intersection quivers of the $\L$-compressing systems and the cluster mutations should correspond to Lagrangian disk surgeries along the $\L$-compressing disks. While the exact Lagrangian fillings $L_P$ we constructed do not arise from Legendrian weaves where one can identify $\L$-compressing disks directly via \cite[Section 2.4]{CasalsZaslow} and \cite[Section 3]{CasalsWeng}, note that the interesting topology of these fillings is concentrated in the pinching region. Since the existence of $\L$-compressing disks is a local property, it suffices to restrict our focus to this specific part of the fillings. In particular, by repeated applications of \cite[Proposition 3.1]{Hughes2021b} that the saddle cobordism is Hamiltonian isotopic to the $D_4^-$ cobordism relative to the boundary, we can realize the Legendrian lift of this region as an inductive weave to which the prior known results can be applied. Indeed, using the local model of the $D_4^-$ cobordism given in Figure~\ref{fig: D4minus}, we see that the front diagram of the weave begins with $c+1$ strands on the top corresponding to the crossings $a_0,\ldots,a_c$ in the front diagram of $\La[w,c]$, followed by a sequence of $c$ trivalent vertices $v_i$ with the right incoming edge coming from strand $i$ corresponding to the pinching of $a_i$ from top to bottom. The outgoing strand of a trivalent vertex is labeled by the same index as its left incoming strand so that we can proceed inductively. In the end, there is a single outgoing strand at the bottom corresponding to the crossing $a_0$.\footnote{As this is a local model, this differs slightly from the convention in, e.g. \cite{CGGLSS}, where the bottom edge represents an additional crossing.} See, for example, Figure~\ref{fig: InitialWeave} (left) for the weave associated to the initial cluster seed of $\Aug_4(\La[w,4])$, and Figure~\ref{fig: TriangulationEx} (right) for the weave corresponding to the pinching sequence $P = (1,4,6,3,5,2)$ applied to $\La[w,6]$. 

\begin{center}
    \begin{figure}[h!]{ \includegraphics[width=\textwidth]{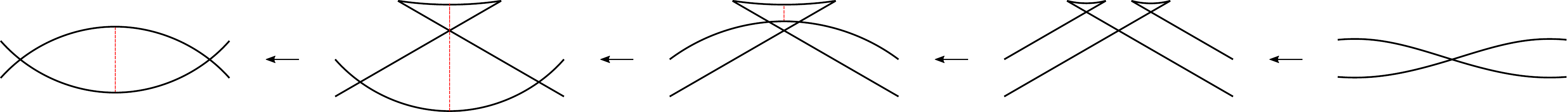}}\caption{$D_4^-$ cobordism depicted in the front projection. The dashed red line depicts a contractible Reeb chord.
    }
    \label{fig: D4minus}\end{figure}
\end{center} 

Recall from Section~\ref{sssec:weave-homology} each short $I$-cycle between two adjacent trivalent vertices represents a cycle in $H_1(L_P)$ that bounds a Lagrangian disk, all of which are linearly independent. The outgoing edge incident to the bottom-most trivalent vertex represents a relative cycle in the weave, which becomes another cycle in $H_1(L_P)$ once the unknot is capped off in $L_P$. Thus, there is a one-to-one correspondence between these cycles and the trivalent vertices in the weaves, assigning each cycle $\gamma_i$ to the top vertex $v_i$ of its defining edge. In particular, this collection consists of exactly $c$ cycles that span $H_1(L_P)$ since $b_1(L_P) = \tb(\La[w, c])+1 = c$. Following the combinatorial recipe given in Section~\ref{sssec:weave-homology}, we can compute the algebraic intersection numbers among these cycles by examining how the corresponding edges intersect at the trivalent vertices in the weave. The intersection quiver of the filling then consists of $c$ vertices $v_1,\dots  v_c$ corresponding to the trivalent vertices in the weave, the first $c-1$ of which are mutable. The $c$-th vertex is frozen and corresponds to a choice of absolute cycle starting at the final trivalent vertex of the weave; to avoid later confusion, we will choose this cycle so that it only intersects the $(c-1)$-st cycle. An arrow is directed from $v_i$ to $v_j$ if the geometric intersection number $\langle \gamma_i, \gamma_j \rangle$ is $+1$. See Figure~\ref{fig: InitialWeave} (right) for the intersection quiver associated to the initial weave in the prior example. In particular, observe that the intersection quiver for the initial weave coincides with the quiver given in Proposition~\ref{proposition: twist-cluster} and the dual graph of the weave coincides with the $(c+2)$-gon triangulation constructed therein. As explained around \cite[Definition 4.18]{CasalsZaslow}, since a Legendrian mutation of a weave corresponds to a flip of a diagonal in a triangulation, we see that the mutations in the cluster variety $\Aug_m(\La[w, c])$ are generated by the Lagrangian disk surgeries of the fillings constructed.

\begin{center}
\begin{figure}[h!]{ \includegraphics[width=.6\textwidth]{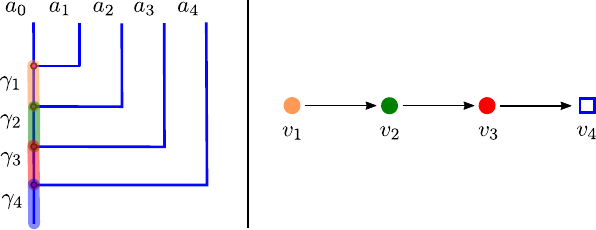}}\caption{Initial weave for $\Aug_4(\La[w,4])$ (left) and the associated quiver (right).}
    \label{fig: InitialWeave}\end{figure}
\end{center}

\subsection{Cluster coordinates from $\L$-compressing cycles}\label{ssec: cluster_coordinates}

We conclude this section with a discussion of a method for obtaining cluster coordinates from the $\L$-compressing cycles of a Lagrangian filling of $\La[w, c]$ as a geometric supplement to the previous subsection. This method can be viewed as  a Floer-theoretic version of the microlocal sheaf-theoretic approach of \cite{CasalsWeng}, where the authors extract cluster coordinates by computing the microlocal parallel transport along relative cycles dual to $\L$-compressing cycles. We expect that these two methods agree with each other and with \cite[Section 5]{CGGLSS}, though a proof in full generality lies beyond the scope of this work.\footnote{By \cite{casals2025comparingclusteralgebrasbraid}, the cluster structures in \cite{CasalsWeng} and \cite{CGGLSS} are already known to agree with one another. See \cite[Section 1.2 (C)]{casals2026categoricallusztigcyclesweave} for discussion on Lusztig cycles and their duals.}

Recall that by Theorem~\ref{thm: cluster_charts} the toric charts $\mathcal{C}_P$ given by the image of the morphism $\phi_{L_P}$ are cluster charts of $\Aug_m(\La[w, c])$. The functions $\{s_i\}_{i=1}^c$ and their inverses parametrize $\mathcal{C}_P$, but in general do \emph{not} extend to global regular functions on $\Aug_m(\La[w, c])$. Cluster coordinates are necessarily regular functions and they also parametrize the same algebraic torus. As a result, there is a monomial change of coordinates taking the functions $\{s_i\}_{i=1}^{c}$ to the restriction of the cluster variables to $\mathcal{C}_P$. A priori, it is not clear how to obtain this monomial change of coordinates from the information of the filling. As we explain below, the key input of the $\L$-compressing cycles is the ability to determine the monomial change of coordinates that produces the cluster variables.  

Given a filling $L_P$, we let $\gamma_i\in H_1(L_P)$ be the homology cycle passing through the $i$-th and the $(c+1)$-st saddle points of $L_P$. The variable $s_i$ then corresponds to the Poincar\'e dual of  $\gamma_i$. Denote by $\mathfrak{B}_P$ the basis $\{\gamma_i\}_{i=1}^{c}\subseteq H_1(L_P)$ and by $\mathfrak{B}_P^\vee$ the dual basis $\{\gamma_i^\vee\}_{i=1}^{c}$. After identifying a basis $\mathfrak{B}_{\L}$ of $\L$-compressing cycles of $L_P$, one can express the dual basis $\mathfrak{B}_{\L}^\vee$ in terms of $\mathfrak{B}_P^\vee$. Denote by $M_{P\L}$ the matrix corresponding to this change of dual bases. 

Under the change of basis specified by $M_{P\L}$, we associate a product of $s_i$ variables to the dual of an $\L$-compressing cycle. The cluster variables of $\Aug_m(\La[w, c])$ are then obtained as the pullback of this product. In order to state this relationship precisely, we identify the diagonals of the triangulation associated to $L_P$ in the proof of Theorem~\ref{thm: cluster_charts} with the dual basis $\mathfrak{B}_\L^\vee$ and index them accordingly.

The following statement, together with Proposition~\ref{proposition: twist-cluster}, completes the proof of Theorem~\ref{thm: main}.

\begin{theorem}
   Let $L_P$ be a Lagrangian filling of $\La[w, c]$  constructed via the pinching sequence $P$ with associated triangulation $T_P$. The restriction of the cluster variable $P_{k,l}$ associated to the $j$-th diagonal of $T_P$ to the toric chart $\mathcal{C}_P$ is given by 
   $$P_{k, l}|_{\mathcal{C}_P} = \prod_{i=1}^{c}s_i^{(M_{P\L})_{ij}}.$$
\end{theorem}

\begin{proof}
   
Following our discussion at the end of  Section~\ref{ssec: filling}, we can produce a local model of the $\L$-compressing cycles of a decomposable filling of $\La[w,c]$ determined by a pinching sequence $P$ by constructing a 2-weave $\w_P$ dual to a triangulation $T_P$. By Theorem~\ref{thm: cluster_charts}, the toric chart $\mathcal{C}_P$ is completely determined by the pinching sequence, and therefore, we can restrict to this local 2-weave picture to perform the necessary computations to prove the statement of the theorem.

By \cite[Lemma 4.1]{Hughes2021b}, the statement is equivalent to the equation $\prod_{i=1}^{c}s_i^{(M_{P\L}^\vee)_{ij}}=s_{k+1}\dots s_{l-1}.$ After restricting to the local 2-weave model, we largely follow the ideas appearing in Section 6.1 of Version 1 of the (now published) preprint \cite{CGGS1Preprint}, adjusting as necessary to make the role of the $\L$-compressing cycles more apparent. As above, we denote by $\gamma_i$ the cycles of $\mathfrak{B}_P$. Since we identify a trivalent vertex $v$ of the weave $\w_P$ with a cycle $\gamma_v\in \mathfrak{B}_P$ by connecting the saddle point at $v$ to the saddle point created by the final pinch in the sequence $P$, we can see that the short {\sf I}-cycles of $\mathfrak{B}_\L$ can be written as $\eta_i=\gamma_{v_{top}}-\gamma_{v_{bottom}}$ for $v_{top}$ the upper trivalent vertex of the short {\sf I}-cycle $\eta_i$ and $v_{bottom}$ the lower trivalent vertex. Given a trivalent vertex $v$, denote by $\mathcal{R}(v)$ the set of Reeb chords lying between the leftmost and rightmost branches of the binary subtree rooted at $v$ opening upwards. From the equation for $\eta_i$ above, we see that $\mathfrak{B}_\L^\vee$ is obtained from $\mathfrak{B}_P^\vee$ by $$\eta_j^\vee=\sum_{\rho\in \mathcal{R}(v_{top})} \gamma^\vee_\rho.$$ As a result, we have that the expression associated to the dual relative cycle to the $\L$-compressing cycle $\eta_j$ is exactly $\prod_{\rho\in \mathcal{R}(v_{top})}s_\rho=s_{v_{k+1}}\dots s_{v_{l-1}}$, where $k+1$ and $l-1$ are the labels of the vertices appearing at the top left and top right of the binary subtree rooted at $v_{top}$. Note that for the final trivalent vertex $v_{final}$, representing the frozen cycle of the quiver, we have that $\mathcal{R}(v_{final})=\{v_1, \dots v_c\}$, ensuring that the frozen variable is equal to $s_1\dots s_{c}$.
\end{proof}

\begin{center}
\begin{figure}[h!]{ \includegraphics[width=.8\textwidth]{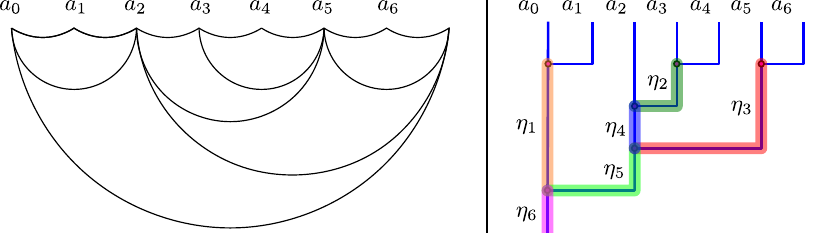}}\caption{The triangulation $T_P$ (left) and the dual weave (right) corresponding to the pinching sequence $P$ given in Example~\ref{ex: cluster_coordinates}. The labels $a_0, \dots, a_6$ denote Reeb chords of $\La[w,c]$ while $\eta_1, \dots, \eta_5 \in H_1(L_P)$ denote $\L$-compressing cycles of $L_P$ and $\eta_6$ completes this collection to a homology basis.}
    \label{fig: TriangulationEx}\end{figure}
\end{center} 

\begin{ex}\label{ex: cluster_coordinates}
Consider the triangulation given in Figure~\ref{fig: TriangulationEx} and its dual weave representing a local model for a decomposable filling $L_P$ of $\La[w, 6]$ specified by the pinching sequence $P= (1,4,6,3,5,2)$.  We denote the cycles of $\mathfrak{B}_P$ by $\gamma_1, \dots, \gamma_6$, where $\gamma_i$ is the cycle formed by the saddle points given by pinching the Reeb chord $a_i$. Then we have that $\mathfrak{B}_\L=\{\gamma_1-\gamma_2, \gamma_4-\gamma_3, \gamma_6-\gamma_5, \gamma_3-\gamma_5, \gamma_5-\gamma_2, \gamma_2\}$. If we further denote by $\gamma_i^\vee$ the dual relative cycles in $\mathfrak{B}_P^\vee$, then the dual basis for $\mathfrak{B}_\L$ is $$\mathfrak{B}_{\L}^\vee=\{\gamma_1^\vee, \gamma_4^\vee, \gamma_6^\vee, \gamma_3^\vee + \gamma_4^\vee, \gamma_3^\vee+\gamma_4^\vee+\gamma_5^\vee+\gamma_6^\vee, \gamma_1^\vee+\dots +\gamma_6^\vee\}.$$ 
We can express the set of cluster variables corresponding to the triangulation $\mathcal{T}_P$ restricted to the cluster chart $\mathcal{C}_P$ as $$\{s_1, s_4, s_6, s_3s_4, s_3s_4s_5s_6, s_1s_2s_3s_4s_5s_6\}.$$ The pullback of these monomials to $\Aug(\La[w, 6])$ is then given by the continuants $$\{a_1, a_4, a_6, 1+a_3a_4, 1+a_3a_4+a_3a_6+a_5a_6+a_3a_4a_5a_6, K_6(a_1, \dots, a_6)\},$$ which are global regular functions on $\Aug_m(\La[w,c])$.       
\end{ex}

\begin{remark}
    Note that in Example~\ref{ex: cluster_coordinates}, we choose to include $\gamma_2$ in the basis $\mathfrak{B}_\L$ despite the fact that $\gamma_2$ is not an $\L$-compressing cycle. This is because the first five cycles represent a presumed maximal linearly independent collection of $\L$-compressing cycles of the filling $L_P$, and we must complete this set of $\L$-compressing cycles to a basis in order to find the dual basis $\mathfrak{B}_\L^\vee$. The choice of a cycle other than $\gamma_2$ is precisely the quasi-equivalence that our cluster structures are defined up to; see \cite[Section 3.7]{CasalsWeng} for more details.  
\end{remark}

\section{Legendrian cables with infinitely many fillings}
\label{sec: cable_computation}

In this section, we give an example of two smoothly isotopic Legendrians $\La$ and $\La'$ that both admit \emph{infinitely many} fillings and show that their augmentation varieties have identical cluster structures. We begin by defining Legendrian cables, as this will allow us to construct $\La$ and $\La'$ as positively sloped cables of twist knots.

\subsection{Braid varieties from Legendrian cables}\label{ssec: Legendrian cables}
We follow the conventions of \cite{DaltonEtnyreTraynor} for constructing Legendrian cables; see also Section 1.2 of loc. cit for more details. In addition, we fix the convention that if $\La$ is equipped with a Maslov potential, then each pushoff of $\La$ also carries the same Maslov potential.

Given integers $p, q,$ and $n$ with $p$ and $n\geq 1$ and $\gcd(p,q)=1$, and a max-tb Legendrian $\La$, the standard Legendrian $(np, nq)$ cable $\La_{(np, nq)}$ of $\La$ is formed by the following process:
\begin{enumerate}
    \item If $\frac{q}{p}=\tb(\La)$, then $\La_{(np, nq)}=\La_{(n, nq)}$ is the $n$-copy of $\La$. A front for $\La_{(np, nq)}$ can be obtained by taking pushoffs in the positive Reeb direction to form $n$ copies of a front projection of $\La$. 
    \item If $\frac{q}{p}> \tb(\La)$, we let $s\in \N$ be defined by the equation $\frac{q}{p}=\tb(\La)+\frac{s}{p}$. To construct the $\La_{(np, nq)}$ cable, we first form the $np$-copy of $\La$. The cable is then obtained by replacing a trivial $np$-stranded tangle of the $np$-copy by the positive braid $(\sigma_{np-1}\dots\sigma_1)^{ns}$.
    \item If $\frac{q}{p}> \tb(\La)$, we let $s \in \N$ be defined by the equation $\frac{q}{p}=\lceil \frac{q}{p}\rceil + \frac{s}{p}$. To construct one of two possible lesser-sloped cables, $\La^\pm_{(np, nq)}$, we first form the $p$-copy of $\La$. We  replace a trivial $p$-stranded tangle with either $s$ $p$-stranded $Z$ tangles or $s$ $p$-stranded $S$ tangles; see Figure~\ref{fig: NegativeCables}. Finally, we take the $n$-copy of the result to obtain $\La^{\pm}_{(np, nq)}$.
\end{enumerate}

\begin{center}
    \begin{figure}[h!]{ \includegraphics[width=.4\textwidth]{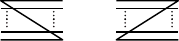}}\caption{An $n$-stranded $Z$ tangle (left) and $S$ tangle (right).}
    \label{fig: NegativeCables}\end{figure}
\end{center} 

Equipped with the construction above, we give a sufficient condition for when the augmentation variety of a cable admits a cluster structure in the sense of \cite{CGGLSS}. 

\begin{proposition}\label{proposition: cluster_cables}
Let $\La=\La(\beta\Delta_m)$ be a knot with $\beta\in \Br_m^+$. Suppose that $\frac{q}{p}\geq \frac{m}{2}+\tb(\La)$ with $p \geq 1$ and $\gcd(p, q)=1$. Then  $\Aug(\La_{np, nq})$ admits a cluster structure for all $n \ge 1$. 
\end{proposition}

\begin{proof}
We prove the statement by  showing that for $\La$ satisfying the hypotheses above, $\La_{(np, nq)}$ is isotopic to $\La(\beta'\Delta_{npm})$ for a braid $\beta'\in \Br_{npm}^+$ and then applying \cite[Theorem 1.1]{CGGLSS}. To that end, consider $\La=\La(\beta\Delta_m)$ as in the statement of the proposition and assume that $\frac{q}{p}=\frac{m}{2}+\tb(\La)$; the case where $\frac{q}{p}>\frac{m}{2}+\tb(\La)$ will follow from the case where we have equality.

First, we note that the $np$-copy of $\La$ is again a $(-1)$-framed closure of a positive braid. We denote the $np$-copy of $\beta$ by $\beta'$.

As the half-twist element $\Delta$ of the braid group can be characterized uniquely by the property that every pair of strands crosses exactly once, we can factor $\Delta_{npm}$ as the product of the $np$-copy of $\Delta_m$ with the braid $\Delta_{np}\Sh^{np}(\Delta_{np})\dots \Sh^{np(m-1)}(\Delta_{np}) $. Note in particular that the shifted copies of $\Delta_{np}$ lie on disjoint collections of strands and that together they account for the crossing of each of the $np$ strands created when replacing a single strand of $\La$ by $np$ strands to form the cable. When $\frac{q}{p}=\frac{m}{2}+\tb(\La)$, we have that $\frac{s}{p}=\frac{m}{2}$. Therefore, the cable $\La_{(np, nq)}$ is formed from the $np$ copy by replacing a trivial $np$ tangle by $(\sigma_{np-1}\dots\sigma_1)^{ns}=\Delta_{np}^{\frac{2m}{2}}=\Delta_{np}^m.$ 

Finally, we claim that since $\La(\beta\Delta_m)$ is a knot, we can obtain $\La(\beta'\Delta_{npm})$ from $\La_{(np, nq)}$ by isotoping the cable so that we replace $\Delta_{np}^m$ by $\Delta_{np}\Sh^{np}(\Delta_{np})\dots \Sh^{np(m-1)}(\Delta_{np}) $. The condition that $\La(\beta\Delta_m)$ is a knot implies that if we consider the Coxeter projection $\pi:\Br_m\to S_m$ of $\beta\Delta_m$, then the permutation $\pi(\beta\Delta_m)$ is a single cycle; equivalently, $(\pi(\beta\Delta_m))^k$ has no fixed points for $1\leq k<m$.

If we assume without loss of generality that $\Delta_{np}^m$ was added to the $np$-copy of $\La$ by replacing the lowest $np$ strands immediately after $\beta'$, then this isotopy can be realized by passing $\Delta_{np}^{m-1}$ counterclockwise around the braid axis $m-1$ times, leaving a copy of $\Sh^{knp}(\Delta_{np})$ each time we pass through all of the crossings of $\beta'$. Since $\pi(\beta\Delta_m)$ is a single cycle, we obtain $\Delta_{np}\Sh^{np}(\Delta_{np})\dots \Sh^{np(m-1)}(\Delta_{np})$ after these $m-1$ revolutions. Therefore, $\La_{(np, nq)}\cong \La(\beta'\Delta_{npm})$, as desired.  
\end{proof}

    As in the statement of \cite[Theorem 1.1]{CGGLSS}, we allow for the possibility that $\Aug(\La_{(np, nq)})$ is empty. In order to ensure that $\Aug(\La_{(np, nq)})$ is nonempty, we require that the braid $\beta'$ that satisfies $\La(\beta'\Delta_{npm})\cong \La_{(np,nq)}$ has Demazure product $\delta(\beta')=w_0$; see Section~\ref{ssec: braid_vties} for more details.  See also \cite[Theorem 1.6]{Sabloff2022satellite} for an alternative sufficient condition for ensuring that the cable $\La_{(np, nq)}$ is fillable, and hence that $\Aug(\La_{(np, nq)})$ is nonempty. Note that the authors of \cite{Sabloff2022satellite} use $\Delta$ to denote the full twist in the braid group.

\subsection{Equivalence of infinite-type cluster structures}\label{ssec: cable_computation}

Let $\beta, \beta'\in \Br_4$ be given by the following braid words: 
\begin{align*}    \beta&=\sigma_1\sigma_1\sigma_3\sigma_3\sigma_3\sigma_2\sigma_2\sigma_3\sigma_1\sigma_2\sigma_2\Delta_4\\ 
    \beta'&=\sigma_1\sigma_1\sigma_1\sigma_2\sigma_3\sigma_3\sigma_2\sigma_2\sigma_1\sigma_2\sigma_2\Delta_4
\end{align*}

As explained in \cite[Section 4.10]{CasalsWengPreprint}, the $(-1)$-closures $\La(\beta)$ and $\La(\beta')$ are two distinct Legendrian representatives of $m(7_2)$.\footnote{Alternatively, it is possible to produce equivalent braid words for $\beta$ and $\beta'$ from the algorithm given in Section~\ref{ssec: braidification} starting with the words $w=ZSZS$ and $w'=ZSSZ$.} Moreover, Proposition~\ref{proposition: Maslov-m fillability} implies that $\La(\beta)$ and $\La(\beta')$ are the only two Maslov-0 fillable Legendrian representatives of $m(7_2)$.

 Consider the $(2, 6)$-cables of $\La=\La(\beta)$ and $\La'=\La(\beta')$. 
 Following the proof of Proposition~\ref{proposition: cluster_cables}, these cables are isotopic to the $(-1)$-closures of the following positive braids:
\begin{align*}\label{eq: cablebraids}
    \beta_{(2, 6)}&=(\sigma_2\sigma_1\sigma_3\sigma_2)^2(\sigma_6\sigma_5\sigma_7\sigma_6)^3(\sigma_4\sigma_3\sigma_5\sigma_4)^2(\sigma_6\sigma_5\sigma_7\sigma_6)(\sigma_2\sigma_1\sigma_3\sigma_2)(\sigma_4\sigma_3\sigma_5\sigma_4)^2\Delta_8\\
    \beta'_{(2,6)}&=(\sigma_2\sigma_1\sigma_3\sigma_2)^3(\sigma_4\sigma_3\sigma_5\sigma_4)(\sigma_6\sigma_5\sigma_7\sigma_6)^2(\sigma_4\sigma_3\sigma_5\sigma_4)^2(\sigma_2\sigma_1\sigma_3\sigma_2)(\sigma_4\sigma_3\sigma_5\sigma_4)^2\Delta_8
\end{align*}

We claim that these two braids correspond to the only two decomposably fillable Legendrian representatives of the $(2, 6)$-cable of $m(7_2)$.

\begin{proposition}\label{proposition: cable_reps}
    $\La(\beta_{(2,6)})$ and $\La(\beta'_{(2,6)})$ are the only two decomposably Maslov-0 fillable Legendrian representatives of the $(2,6)$-cable of $m(7_2)$.
\end{proposition}

\begin{proof}
We first describe the max-tb representatives of the $(2, 6)$-cable of $m(7_2)$ following the classification given by \cite[Theorem 1.7]{CEMR25}.
By \cite[Theorem 1.3]{CEMR25}, the knot type $m(7_2)$ (also denoted by $K_{-6}$ in the notation of loc. cit.) is uniformly thick, meaning that $w(m(7_2))=\overline\tb(m(7_2))$. Here $w(K)$ denotes the width of the knot $K$, i.e. the supremum of the slopes of the dividing curves on the boundary of solid tori with convex boundary and core isotopic to $K$. Since $\overline{tb}(m(7_2))=1$, we have that the $(2,6)$-cable of $m(7_2)$ has slope $\frac{3}{1}$, which is greater than $w(m(7_2))$ and therefore our cable satisfies the hypotheses of \cite[Theorem 1.7]{CEMR25}. By \cite[Theorem 3.3]{EtnyreNgVertesi13}, the Legendrian representatives of $m(7_2)$ are non-destabilizable if and only if they maximize Thurston-Bennequin number. Therefore, \cite[Theorem 1.7]{CEMR25} implies that the max-tb representatives of the $(2,6)$-cable of $m(7_2)$ are precisely the representatives obtained by applying the cabling procedure described at the beginning of Section~\ref{ssec: Legendrian cables} to each of the max-tb representatives of $m(7_2)$.

We claim that a Legendrian representative of the $(2,6)$-cable of $m(7_2)$ is decomposably Maslov-0 fillable if and only if the corresponding max-tb representative of $m(7_2)$ is. We construct decomposable weave fillings of $\La(\beta_{(2,6)})$ and $\La(\beta'_{(2,6)})$ below. Therefore, to verify our claim, we argue that the remaining two Legendrian representatives, which we denote $\La_1$ and $\La_2$ do not admit 0-graded normal rulings with enough switches to be a canonical ruling induced by a filling of these Legendrian representatives. We construct $\La_1$ and $\La_2$ by applying the cabling procedure to the fronts $\La[ZSSS, 2]$ and $\La[SSSS, 2]$ and see that they both have eight left cusps and satisfy $\tb(\La_1)=\tb(\La_2)=8$. Direct computation then shows that the $0$-graded rulings of $\La_1$ and $\La_2$ have at most ten switches, implying the minimum Euler characteristic of a 0-graded ruling $\rho$ of $\La_1$ or $\La_2$ is $\chi(\rho)=8-10=-2$. Proposition~\ref{proposition: canonical_ruling_obstruction} therefore obstructs the existence of decomposable fillings of $\La_1$ and $\La_2$. 
\end{proof}

\begin{center}
    \begin{figure}[h!]{ \includegraphics[width=\textwidth]{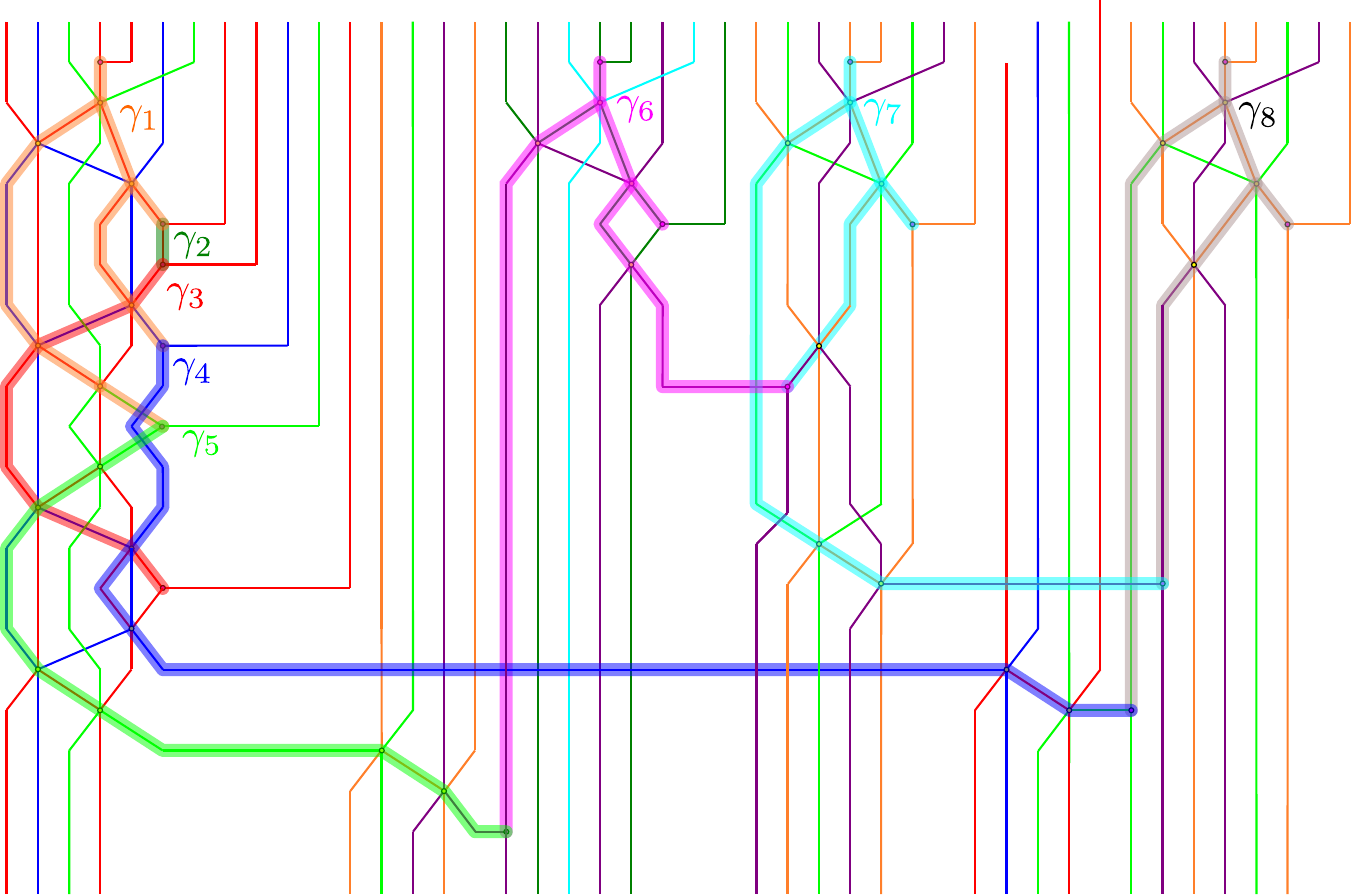}}\caption{Legendrian weave filling of the $m(7_2)$ cable $\La(\beta_1)$.} 
    \label{fig: CableWeave1}\end{figure}
\end{center} 

\begin{center}
    \begin{figure}[h!]{ \includegraphics[width=\textwidth]{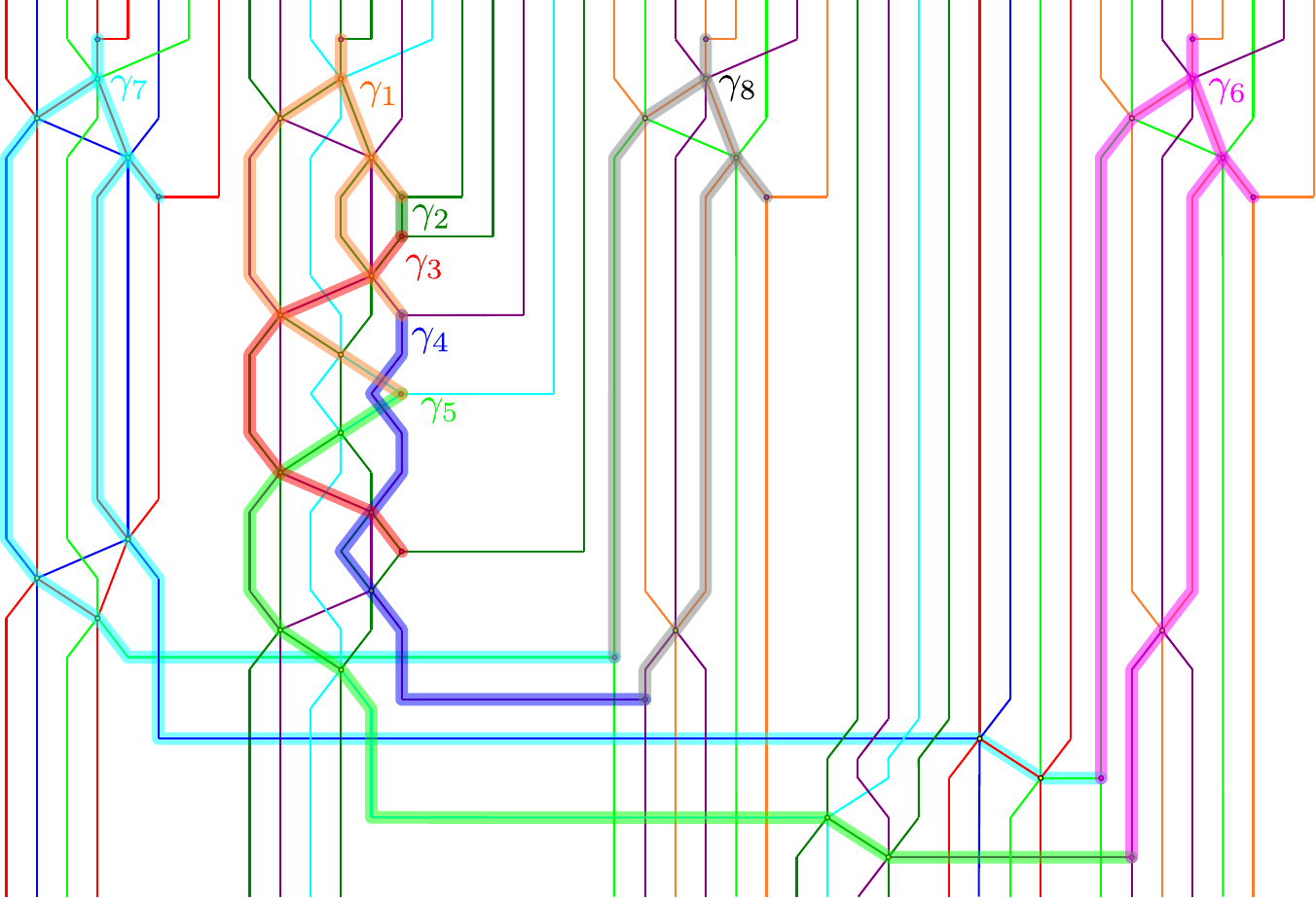}}\caption{Legendrian weave filling of the $m(7_2)$ cable $\La(\beta_2)$.}
    \label{fig: CableWeave2}\end{figure}
\end{center} 

We expect that the statement of Proposition~\ref{proposition: cable_reps} remains true after removing the adjective `decomposable,' though any proof we are aware of would require somewhat involved computations of the DGAs of the two representatives.

We now prove the equivalence of the cluster structures associated to $\La_{(2,6)}$ and $\La_{(2,6)}'$.

\begin{theorem}\label{thm: infty-equivalent-body}
    The cluster varieties $\Aug(\La_{(2,6)})$ and $\Aug(\La_{(2,6)}')$ are quasi-equivalent and have infinitely many cluster seeds.
\end{theorem}

\begin{proof}
Figures~\ref{fig: CableWeave1} and \ref{fig: CableWeave2} depict (right inductive) Legendrian weave fillings of $\La(\beta_{(2,6)})$ and $\La(\beta'_{(2,6)})$ together with their labeled Lusztig cycles. The mutable part of each of their intersection quivers is given in Figure~\ref{fig: CableQuivers} where the color and subscript on the vertex label matches the color and subscript of the corresponding Lusztig cycle in the weave. Starting from the quiver for $\La(\beta_1)$ depicted on the left of the figure, we obtain identical mutable subquivers after the following mutation sequence obtained from the SageMath cluster algebra package \cite{sagemath}: $v_5, v_3, v_5, v_1, v_5, v_6.$ 

 \begin{center}
    \begin{figure}[h!]{ \includegraphics[width=.4\textwidth]{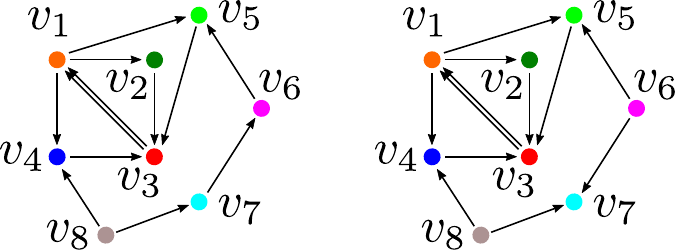}}\caption{Mutable subquivers of the weave fillings of $\La(\beta_1)$ (left) and $\La(\beta_2)$ (right).}
    \label{fig: CableQuivers}\end{figure}
\end{center}
To see that the cluster varieties we obtain are of infinite type, we appeal to the classification of finite type cluster algebras \cite[Theorem 1.4]{FominZelevinsky_ClusterII}. If either of our cluster varieties were of finite type, then their quivers would be mutation equivalent to an arbitrary orientation of a Dynkin diagram of type $A_8, D_8,$ or $E_8$. In types $A_n$ and $D_n$, any quiver mutation equivalent to an arbitrarily oriented Dynkin diagram is classified in \cite{VatneBuan, Vatne} and none of them have arrows with multiplicity greater than one. To distinguish our quivers from a quiver mutation equivalent to an $E_8$ quiver, we use the $\delta$ invariant of \cite{casals2023binaryinvariantmatrixmutation}: this invariant vanishes by direct computation for the quivers depicted in Figure~\ref{fig: CableQuivers}, whereas the $\delta$ invariant is non-vanishing for any quiver mutation equivalent to the $E_8$ Dynkin quiver.
\end{proof}
  
Note that the braid words $\beta_{(2,6)}$ and $\beta_{(2,6)}'$ are chosen carefully to simplify the construction of the weaves pictured in Figures~\ref{fig: CableWeave1} and $\ref{fig: CableWeave2}$. As such, it's not particularly clear how to meaningfully generalize the example we discuss in Theorem~\ref{thm: infty-equivalent-body}, especially since the mutation sequence we use in the proof does not appear to be readily obtainable from the geometry of the weaves or the braid words themselves. While one could likely produce an infinite family of Legendrians with quasi-equivalent cluster structures of infinite type by increasing the slope of the cable, any computation involving e.g. other twist knots seems likely to become increasingly involved. 

\begin{remark}
We wish to remark here on the number of $\L$-compressing cycles bounded by our weave fillings. From Section~\ref{ssec: filling}, we see that the number of $\L$-compressing cycles we identify for any filling $L$ we construct is one fewer than the rank of $H_1(L)$. In any example that we have checked, this remains true of fillings of any cable of $\La[w,c]$ so long as the slope is at least $\tb(\La[w, c])$. This phenomenon seems to be fairly robust: for the Maslov-0 fillable representative of $7_4$ given in \cite{ChongchitmateNgPre}, one can show that no genus-one filling bounds an $\L$-compressing disk, while we identify $|H_1(L)|-2$ $\L$-compressing disks in any filling $L$ of the non-negatively sloped cables of $7_4$ that we investigated. In general, we expect that cabling preserves the difference between the rank of $H_1(L)$ and the number of $\L$-compressing cycles that the filling bounds.
\end{remark}

\subsection{Constructing infinitely many fillings.}
We now verify that $\La(\beta_{(2,6)})$ and $\La(\beta'_{(2, 6)})$ both admit infinitely many fillings. Note that this does not follow immediately from the fact that the corresponding augmentation varieties admit infinitely many cluster seeds because surjectivity of the map from fillings to cluster seeds is only established for braid positive Legendrians in \cite[Theorem 1.1]{CasalsGao23}. Their result is done using the technology of quivers with potentials, and there are significant challenges associated with generalizing the approach to the case of $\La(\beta\Delta)$. We therefore require an explicit construction of Lagrangian cobordisms between our Legendrian links and Legendrian links that are already known to admit infinitely many Lagrangian fillings. We begin by defining a technical condition allowing us to distinguish such fillings.

\begin{definition}
    A Legendrian link $\La$ is aug-infinite if the collection of all $\Z$-valued augmentations of $\SA(\La, \partial)$ induced by its orientable, Maslov-0 exact Lagrangian fillings is infinite.
\end{definition}

Note that any aug-infinite Legendrian link necessarily admits infinitely many fillings, as each filling can only induce finitely many $\Z$-valued augmentations \cite[Proposition 7.3]{CasalsNg}.

The following statement, together with Theorem~\ref{thm: infty-equivalent-body} prove Theorem~\ref{thm: infty-equivalent}: 

\begin{proposition}
    $\La_{(2, 6)}$ and $\La'_{(2, 6)}$ both admit infinitely many fillings.
\end{proposition}

\begin{proof}
    We show that $\La_{(2,6)}$ and $\La_{(2,6)}'$ are both aug-infinite by exhibiting an exact Lagrangian cobordism to the Legendrian $\La(\beta_{11})=\La(\widetilde{A}_2)$ studied in \cite{CasalsNg}. This link is isotopic to the $(-1)$-closure of $\beta=(\sigma_2\sigma_1\sigma_3\sigma_2)\Delta_4$. By \cite[Proposition 7.4]{CasalsNg}, $\La(\widetilde{A}_2)$ is aug-infinite, and we we wish to conclude that $\La_{(2,6)}$ and $\La_{(2,6)}'$ are  aug-infinite as well. To that end, we consider the Lagrangian pigtail closure of the braids $\beta_{(2,6)}$ and $\beta'_{(2,6)}$; see Figure~\ref{fig: PigtailClosure}. To verify that these pigtail closures are isotopic to the $-1$-closures we used to define $\La_{(2,6)}$ and $\La'_{(2,6)}$, see \cite[Section 2.2]{CasalsNg}. Our braids are admissible in the sense of \cite[Definition 2.5]{CasalsNg} because they contain $\Delta_8$. In the case of $\La_{(2,6)}'$, we will also perform an isotopy that rotates the first twenty crossings of our braid to obtain the braid below, which we will also denote by $\beta'_{(2,6)}$:
     $$\beta'_{(2,6)}=(\sigma_4\sigma_3\sigma_5\sigma_4)^2(\sigma_2\sigma_1\sigma_3\sigma_2)(\sigma_4\sigma_3\sigma_5\sigma_4)^2(\sigma_6\sigma_7\sigma_5\sigma_6)^3(\sigma_4\sigma_5\sigma_3\sigma_4)(\sigma_2\sigma_3\sigma_1\sigma_2)^2\Delta_8$$

\begin{center}
    \begin{figure}[h!]{ \includegraphics[width=.4\textwidth]{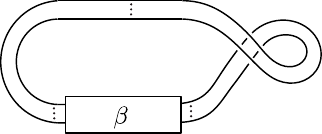}}\caption{Lagrangian pigtail closure of an admissible positive braid $\beta$.}
    \label{fig: PigtailClosure}\end{figure}
\end{center} 
We index the crossings of $\beta_{(2,6)}$ and $\beta'_{(2,6)}$ from left to right. Below, we give two pinching sequences of proper contractible Reeb chords that result in a disjoint union of $\La(\widetilde{A}_2)$ with four unlinked max-tb Legendrian unknots. For each of the two Legendrians indicated, resolve the listed crossings in order:

\begin{align*}
    \La_{(2,6)}&: a_5, a_8, a_{25}, a_{26}, a_{27}, a_{28}, a_{41}, a_{42}, a_{43}, a_{44}\\
    \La_{(2,6)}'&: a_5, a_6, a_7, a_8, a_{13}, a_{14}, a_{15}, a_{16}, a_{41}, a_{44} 
\end{align*}

If we label the strands from bottom to top, then the pinching sequences produce pigtail closures of the following two braids: \begin{align*}
    \Delta_4\Sh^4(\sigma_2\sigma_1\sigma_3\sigma_2)^3(\sigma_4\sigma_3\sigma_5\sigma_4)(\sigma_6\sigma_5\sigma_7\sigma_6)(\sigma_2\sigma_1\sigma_3\sigma_2)(\sigma_4\sigma_3\sigma_5\sigma_4)\Delta_8,\\
    (\sigma_4\sigma_3\sigma_5\sigma_4)(\sigma_2\sigma_1\sigma_3\sigma_2)(\sigma_4\sigma_3\sigma_5\sigma_4)\Sh^4(\sigma_2\sigma_1\sigma_3\sigma_2)^3(\sigma_4\sigma_3\sigma_5\sigma_4)\Delta_4\Delta_8,
\end{align*} both of which are isotopic to the disjoint union of four unlinked unknots with $\La(\widetilde{A}_2)$. Thus, since $\La(\widetilde{A}_2)$ is aug-infinite, both $\La_{(2,6)}$ and $\La_{(2,6)}'$ are aug-infinite as well by \cite[Proposition 7.5]{CasalsNg}.
\end{proof}

\bibliographystyle{alpha}
\bibliography{main.bib}

\end{document}